\documentclass[11pt,a4paper]{article}%
\usepackage[centertags]{amsmath}
\usepackage{amsfonts}
\usepackage{amssymb}
\usepackage{amsthm}
\usepackage{epsfig}
\usepackage{setspace}
\usepackage{ae}
\usepackage{eucal}
\usepackage[usenames]{color}%
\usepackage{graphicx}

\theoremstyle{plain}
\newtheorem{thm}{Theorem}[section]
\newtheorem{lem}[thm]{Lemma}

\theoremstyle{definition}
\newtheorem{defi}[thm]{Definition}

\theoremstyle{remark}

\newtheorem{rem}[thm]{Remark}

\begin{document}

\title{On the multivariate Burgers equation and the incompressible Navier-Stokes equation (part III)}
\author{J\"org Kampen }
\maketitle

\begin{abstract}
For the incompressible Navier-Stokes equation with data of polynomial decay we first obtain for given regular data at a time step $l$ time-local contraction results with respect to a $C^1\left( (l-1,l), H^{m}\right) $-norm.  For an extended controlled scheme with local limit functions $l\rightarrow v^{r,*,\rho,l}_i(l,.)=v^{*,\rho,l}_i(l,.)+r^l_i$ and growth control function $r^l_i$ we obtain a global linear upper bound and an absolute upper bound for the Leray projection term with respect to a $C^1\left( (0,l), H^{m}\right) $-norm, where the latter is independent the time step number $l$. In order to obtain the absolute global upper bound certain growth consumption functions are added in the definition of the dynamically defined control functions, while for the linear global upper bound it suffices to consider a simplified control function. Furthermore, the higher order correction terms of the local controlled scheme preserve some order of spatial polynomial decay. The schemes discussed in \cite{KNS} and \cite{K3} are simplified in the sense that the estimates are achieved without the use of some properties concerning the adjoint of local fundamental solutions with variable first or second order coefficient terms. We note that the pointwise and absolute convergence of the local functional series and their first order time derivatives and their spatial derivatives leads to a constructive approach for existence of local classical solutions and of global regular (smooth) solutions via the controlled scheme. We also introduce an extended scheme with small time foresight such that the existence of uniform global upper bounds can be proved and even long time behavior can be investigated. Furtherore the global implications for the uncontrolled Navier Stokes equation are stated explicitly, where for the simplest control function a global upper bound for the components of the velocity solution functions of the uncontrolled Navier Stokes equation in $C^1\left(\left[0,T\right] , H^m \right)$ for $m\geq 2$ and arbitrary time horizon $T>0$ is obtained which depends linearly on the time horizon $T$. We also discuss the relation to an alternative approach of auto-controlled schemes, i.e., schemes without external control function proposed in \cite{KAC} recently which underlines that the local contraction results are essential in order to obtain global results for the incompressible Navier Stokes equation.   
\end{abstract}



\section{Introduction}
We start with some remarks about current research on existence and uniqueness of the Navier Stokes equation. Leray's notion of weak solutions in \cite{L}, his proof of the existence of global weak solutions, the technique of smoothing first order coefficients, and the conjecture of singular vorticity, are a framework for much of the following research. There are well-written surveys on this, and there is no need to repeat this here. However, we want to make some remarks concerning recent research, where we do not pretend to cover present research or to give a fair account. It seems that some researchers  still relate the concept of turbulence to the existence of  singular sets of vorticity, although the concept of singularity involves the concept of infinity, and for this reason a close link between both concepts seems to be suspicious even on an abstract level. Cantor introduced the modern concept of infinity as infinitum in abstracto in order to avoid the pretension of infinities in nature (especially local infinities). It is by no means obvious that the concept of singular behavior should be a notion of the concept of turbulence which is an empirical concept. If global existence of a regular solution and uniqueness holds for the incompressible Navier Stokes equation for some class of natural data with polynomial decay of a certain order (as we argue), and if this (or similar extensions with variable viscosity) is the right equation in order to study  the concept of turbulence (as we surmise), then other features of this equation and its solution may lead us to a theory of turbulence beyond the concept of Kolmogorov. Some relations between mathematical properties of equations and their solutions may be explored by simulations, but the existence of singularities can be neither proved or disproved by simulations which explore small parts of bounded intervals of rational numbers. It has been stated quite clearly by many researchers (notably by Constantin and Temam, cf. \cite{CFNT}, \cite{FJRT}, \cite{FJKT} and the references therein) that other dynamical properties such as the existence of global strange attractors, bifurcation behavior etc. may be no less important for turbulence than the question of uniqueness and global existence. Among the latter two properties physicists tend to prefer the former, since it tells us that the law described by the equation is deterministic, while the latter is something which a meaningful equation should have anyway. Some physicists may even say that an equation should have smooth global solutions in order to be meaningful at all. Global existence of smooth solutions can help to prove uniqueness, and time-local contraction results are useful in this respect also.
From this point of view you may look at the time-local contraction result below this way: at each time step $l\geq 1$ having determined the data $v^{l-1}_i(t^l_0,.),~1\leq i\leq n$, for some time $t^l_0>0$ we determine a time-local fixed point
of a map
\begin{equation}
\mathbf{f}\rightarrow \mathbf{v}^{f,l},
\end{equation}
where $\mathbf{f}^l=(f^l_1,\cdots ,f^l_n)^T$ is a function such that for all $1\leq i\leq n$ and on some local time interval $[t^l_0,t^l_1]$ the functions $f^l_i$ with $f^l_i(t,.)\in H^2,~[t^l_0,t^l_1]$ are Lipschitz continuous with respect to time, and the function $\mathbf{v}^l=(v^l_1,\cdots,v^l_n)^T$ satisfies the equation
\begin{equation}\label{Navlerayf}
\left\lbrace \begin{array}{ll}
\frac{\partial v^{l}_i}{\partial t}-\nu\sum_{j=1}^n \frac{\partial^2 v^{l}_i}{\partial x_j^2} 
+\sum_{j=1}^n f^l_j\frac{\partial f^l_i}{\partial x_j}=\\
\\ \hspace{1cm}\sum_{j,m=1}^n\int_{{\mathbb R}^n}\left( \frac{\partial}{\partial x_i}K_n(x-y)\right) \sum_{j,m=1}^n\left( \frac{\partial f^l_m}{\partial x_j}\frac{\partial f^{l}_j}{\partial x_m}\right) (t,y)dy,\\
\\
\mathbf{v}^{l}(t^l_0,.)=\mathbf{v}^{l-1}(t^l_0,.).
\end{array}\right.
\end{equation}
This time local fixed point determines a time-local unique classical solution of the incompressible Navier-Stokes equation in its Leray projection form. Uniqueness is implied only in a rather limited space (from the mathematical analyst point of view - the physicist may be satisfied), but this is not our main issue here. The techniques of Koch and Tartaru are certainly stronger in this respect (cf. \cite{KT}). Our main issue in this paper is to define global fixed points of the map (\ref{Navlerayf}) in strong norms, and for this we need some additional ideas which go beyond a clever analysis of local iterations in a Banach space. The question of global smooth existence for strictly positive viscosity $\nu>0$ has been reduced to the question of $H^1$-regularity by T. Tao (cf. \cite{Tao}). This may be supplemented by weak derivative schemes converging in $H^1$. In contrast we follow a classical approach here, where local contraction results in strong norms are supplemented by the idea of dynamically defined bounded regular control functions which control the growth of the Leray projection term.  
 Another idea may be the reduction of the Navier-Stokes equation to other equations which may be easier to solve such as reductions by Cole-Hopf transformations. However such transformations are very special, and natural extensions such as variable viscosity or incompressible Navier-Stokes equations on manifolds cannot be treated this way. Therefore we stick to the original equation and define a global scheme for the incompressible Navier-Stokes equation in its Leray projection form. For simplicity we present it in a form which assumes constant viscosity $\nu>0$, but slightly more complicated variations of the scheme can be defined which allow to generalize the argument to a certain class of equations with variable first and second order coefficients as indicated in \cite{KB2,K3,KNS}. The use of the adjoint of the fundamental solution of scalar linear parabolic equations is a quite powerful tool in order to extend results of the present paper to such kind of systems. Note that the adjoint can be used also for equations with regular variable second order coefficients which satisfy a natural ellipticity condition.  Further generalisations seem to be possible. Indeed it seems that the systems which satisfy a H\"{o}rmander conditions such as the one proposed by Mattingly  ( and as we indicated in \cite{K3,KIF2}) in the case of dimension $2$ define a natural class of systems with global smooth existence in higher dimensions $n\geq 3$ as well (if we consider them without noise, of course). All these systems are characterized by the existence of smooth densities which disappear in the inviscid limit. This does not mean that global solutions of the incompressible Euler equation may not be constructed via inviscid limits of global solutions of the incompressible Navier-Stokes equation in its Leray projection form. This is possible if the estimates involved can be done with constants which are independent of the size of the viscosity. It seems that some fine tuning of the analysis presented here is possible which relates the order of polynomial decay of the data to the existence of global regular solutions. It seems that we have to be more restrictive for the incompressible Euler equation than for the incompressible Navier-Stoke equation in order to obtain global regular inviscid limits  if this is possible at all. Anyway, looking for solutions of the incompressible Euler equation via inviscid limits is -speaking metaphorically - like the effect of a torch in big space from a very specific angle. Young inequality estimates for convolutions with the Gaussian of a heat equation with viscosity $\nu$ as a factor of the Laplacian are certainly useful in this respect. But note that the Leray projection form has an other status for the incompressible Euler equation as it has for the incompressible Navier-Stokes equation. Especially it seems unlikely that an inviscid limit (if it exists in a given situation) is a unique solution of the incompressible Euler equation. 
 It seems that many researchers do not expect global smooth existence and uniqueness for the corresponding incompressible Euler equation in case of dimension $n\geq 3$. Let us point out why we agree with this view. If there exists a global smooth and unique solution to the incompressible Euler equation, then arbitrary derivatives of the incompressible Euler equation have a global unique solution. For higher order derivatives we may apply the product rule to the equation terms and get more and more extended sums for the Leray projection term, and it may be possible then to construct nonzero solutions to the incompressible Euler equation with zero Leray projection term and zero divergence (incompressibility). This solution may then be at the same time a singular solution of the multivariate Burgers equation with a finite singular behavior. It is even likely that for some higher order derivative such solutions can be constructed with data of finite energy. Furthermore, is seems that there are singular solutions of the vorticity form of incompressible Euler equations with a singular point $(T,0)$, which are inviscid limits of global solutions of an incompressible Navier Stokes type equation. Such incompressible Navier Stokes type equations are obtained on spatially bounded domains (cones) by natural transformations which may be trivially extended to the whole domain (cf. \cite{KE1}).
In addition to our remark above concerning the question of uniqueness with regard to the incompressible Euler equation, note that the existence of multiple weak solutions has been shown by Wiedemann recently. If it is true (as we think) that the incompressible Navier-Stokes equation and a considerable larger class of systems with first order coupling satisfy global smooth existence and uniqueness, and if it is true that the incompressible Euler equation satisfies neither of these two fundamental properties of having a unique and a global smooth solution, then its seems quite likely that viscosity $\nu=0$ is a bifurcation point of other qualitative behavior. Preservation of properties for the inviscid limit have to be studied with caution for every property, and, on the other hand, we have to be very cautious with respect to the view that the incompressible Euler equation is an approximative model of the incompressible Navier-Stokes equation with small viscosity. 
Furthermore, the existence of multiple weak solutions should make us cautious concerning any 'numerical evidence' for the type of singularity expected. So much for a preface.

Next we reconsider the definition of the old local scheme defined in (\cite{KB2}).        First, we recall a naive simple scheme in brief which leads to local solutions. For time step $l\geq 1$, a substep iteration number $k\geq 1$, and a small time step size $\rho_l$ along with $ \frac{1}{l}\lesssim \rho_l$ we consider a time-local functional scheme for functions $v^{\rho,l,k}_i:[l-1,l]\times {\mathbb R}^n\rightarrow {\mathbb R},~1\leq i\leq n, k\geq 1$ with limit
\begin{equation}
v^{\rho, l}_i:=v^{\rho,l-1}_i(l-1,.)+\sum_{k=1}^{\infty} \delta v^{\rho,l,k}_i,
\end{equation}
where for $1\leq i\leq n$  we have $v^{\rho,l-1}_i(l-1,.)\in H^2\cap C^2$, and   
$\delta v^{\rho,l,k}_i= v^{\rho,l,k}_i- v^{\rho,l,k-1}_i$ are functions defined on the domain $[l-1,l]\times {\mathbb R}^n$ along with $v^{\rho,l,0}:=v^{\rho,l-1}_i(l-1,.)$. The functions $v^{\rho,l,k}_i$ satisfy time-local approximations of the incompressible Navier-Stokes equations of the form
\begin{equation}\label{Navleray}
\left\lbrace \begin{array}{ll}
\frac{\partial v^{\rho,l,k}_i}{\partial \tau}-\rho_l\nu\sum_{j=1}^n \frac{\partial^2 v^{\rho,l,k}_i}{\partial x_j^2} 
+\rho_l\sum_{j=1}^n v^{\rho,l,k-1}_j\frac{\partial v^{\rho,l,k}_i}{\partial x_j}=\\
\\ \hspace{1cm}\rho_l\sum_{j,m=1}^n\int_{{\mathbb R}^n}\left( \frac{\partial}{\partial x_i}K_n(x-y)\right) \sum_{j,m=1}^n\left( \frac{\partial v^{\rho,l,k-1}_m}{\partial x_j}\frac{\partial v^{\rho,l,k-1}_j}{\partial x_m}\right) (\tau,y)dy,\\
\\
\mathbf{v}^{\rho,l,k}(l-1,.)=\mathbf{v}^{\rho,l-1}(l-1,.),
\end{array}\right.
\end{equation}
and on the domain $[l-1,l]\times {\mathbb R}^n$ for $l\geq 1$. This scheme is a time-local scheme, where the initial data are assumed to be bounded, i.e.,
\begin{equation}
\max_{1\leq i\leq n}{\big |}v_i^{\rho,l-1}(l-1,.){\big |}_{H^2}\leq \tilde{C}
\end{equation}
for some finite constant $\tilde{C}$, and where $H^m$ denotes the standard Sobolev space of order $m$.
Furthermore, the function $K_n$ is the fundamental solution of the $n$-dimensional Laplacian. We are interested in the case $n\geq 3$. For $l=1$ we define
$\mathbf{v}^{\rho,l-1}(l-1,.)=\left(v^{\rho,l-1}_1(l-1,.),\cdots ,v^{\rho,l-1}_n(l-1,.)\right)^T =\mathbf{h}(.)=\left(h_1,\cdots ,h_n \right)^T $. It is natural to assume that
\begin{equation}
h_i\in \cap_{s\in {\mathbb R}}H^s\left({\mathbb R}^n\right) 
\end{equation}
for all $1\leq i\leq n$. This assumption is implied by the slightly stronger assumption that we have polynomial decay of arbitrary order of the initial value functions (cf. \cite{KB2} for a short discussion). However, the essence of this paper is that data with regularity $h_i\in H^m\cap C^m$ for $m\geq 2$ lead to globally existent regular velocity components $v_i$ with $v_i(t,.)\in H^m\cap C^m$ for all $1\leq i\leq n$ and $t\in [0,\infty)$. Furthermore, for $m\geq 2$ we get classical global solution, i.e., differentiability of the velocity functions with respect to time.  
First we shall show that for rather strong norms we can choose a certain step size $\rho_l>0$ such that we have a local contraction result. We shall determine this step size explicitly. The subscript $l$ indicates that the step size may decrease with the time step number $l\geq 1$, but it is clear that it should satisfy at least $\rho_l\sim \frac{1}{l}$ for a global scheme. In order to get a global scheme we want to show that a strong local contraction result can be combined with the analysis of a global equation for the controlled function $\mathbf {v}^r=:\mathbf{v}+\mathbf{r}=\left(v^{r}_1,\cdots ,v^r_n \right)^T $  with a regular bounded control function $\mathbf{r}=\left(r_1,\cdots ,r_n \right)^T$. We define a contolled scheme where the controlled velocity component functions $v^r_i$ and the components of the control funcition are defined dynamically, and then we derive upper bounds for the controlled velocity components $v^r_i$ and for the control function components $r_i$ in $C^0\left(\left[0,T\right],H^m\right)$ for $m\geq 2$ for arbitrary $T>0$. The control function will be only Lipshitz at some discrete times, but it turns out that we first get upper bounds for the uncontrolled velocity component functions $v_i=v^{r}_i-r_i\in C^0\left(\left[0,T\right],H^m\right)$ and then $v_i=v^{r}_i-r_i\in C^0\left(\left[0,T\right],H^m\right)$ for arbitrary $T>0$. This leads to global regular  existence for the uncontrolled incompressible Navier Stokes equation.
More precisely, in this part III of articles on the multivariate Burgers equation and the incompressible Navier Stokes equation we make several contributions. First we define a simplified scheme which will allow us to simplify the proof of local contraction in $H^m$-based spaces for $m\geq 2$ in the sense that the argument based on the adjoint of the fundamental solution considered in \cite{KB2} is avoided. Second, we improve the result concerning the local contraction by showing that a certain kind of polynomial decay is preserved for a controlled value function $v^{r,*,\rho,l}_i$ together with higher order regularity. Locally, the control functions $r^l_i$ (where similarly as in the notation for the velocity value functions the upper index of the time step number $l\geq 1$ indicates a restriction to the domain $[l-1,l]\times {\mathbb R}^n$) 'can take some load' such that they have a polynomial decay of smaller order, but this slightly smaller polynomial decay is preserved as well. Third, we show more explicitly than before that the scheme proposed is an approach of constructing classical solution rather than a mere numerical scheme. Fourth, we show that an extension of the scheme with a simplified control function leads to a linear bound of the growth of the solution with respect to the time step number $l\geq 1$. This linear upper bound holds with respect to the $C^0\left((l-1,l), H^{m}\right) $-norm and the $C^1\left((l-1,l), H^{m}\right)$-norm for appropriate $m\geq 2$ and where we take suprema with respect to time.  The price to pay for a simplified control function is that we have only a global linear bound of the Leray projection term and a linearly decreasing time-step size. For this reason we also consider extended control functions with certain consumption functions which lead to an upper bound of the Leray projection term which is independent of the time step number. We mentioned earlier in \cite{KNS} that a controlled scheme of the incompressible Navier Stokes equation can be designed which allows for a constant step size and for a numerical stabilization of computations. Here we propose an alternative proof method for the sake of global existence and regularity, not for numerical purposes in the first place. We shall see that the local contraction estimate is an essential step. The next section discusses the relation of local and global controlled solutions. Then in section 3 we introduce a simplified scheme. In section 4 we state the local contraction results. In section 5 we state the essential global results for the uncontrolled Navier Stokes equation, which can be derived rather straightforwardly from the controlled scheme together with the local contraction results.  In section 6 we prove that polynomial decay of a certain order is preserved for the higher order correction terms by the local controlled scheme.  Then in section 7 we prove  local contraction results for this simplified scheme, where we improve the local result in the $C^0\left((l-1,l), H^{m}\right)$ function space of \cite{KB2}. In section 8 we add some remarks for related norms, and in section 9 and section 10 we prove the existence of a uniform global bound and a global linear bound for the Leray projection term respectively. Additional information with respect to the analysis in section 8 and section 9 is provided in (\cite{KHyp}). We close the article with some final remarks in section 11.

\section{The relation of local and global solutions}
In this section we make three observations. First, local contraction results in the function space $C^0\left((l-1,l), H^{m}\right)$ and $C^1\left((l-1,l), H^{m}\right)$ along with the norms
\begin{equation}
 {\big |}f{\big |}_{C^0\left((l-1,l)\times H^m\left({\mathbb R}^n\right) \right) }:=\sup_{\tau\in (l-1,l)}|f(\tau,.)|_{H^m\left({\mathbb R}^n\right) },
\end{equation}
and
\begin{equation}
\begin{array}{ll}
 {\big |}f{\big |}_{C^1\left((l-1,l)\times H^m\left({\mathbb R}^n\right) \right) }:=&\sup_{\tau\in (l-1,l)}|f(\tau,.)|_{H^m\left({\mathbb R}^n\right) }\\
 \\
 &+{\big |}D_{\tau}f{\big |}_{C^0\left((l-1,l)\times H^m\left({\mathbb R}^n\right) \right) },
 \end{array}
\end{equation} 
for $m\geq 2$ lead to local existence results of smooth solutions. Second, local contraction results can be used within controlled Navier Stokes equation systems in order to design global schemes with global controlled value functions and globally defined  control functions. Third, we observe that these global controlled Navier-Stokes equation systems lead to the existence of global solutions. We consider the local scheme defined in the introduction, and which was also considered in \cite{KB2}. We note that the rules for the adjoint of scalar parabolic equations hold also for a considerable class of scalar parabolic equations with variable second order coefficients. Therefore this scheme can be generalized to more general equations with variable viscosity. Later we shall consider some variations of these observations for the simplified scheme of this paper. 

In order to prove global existence results from local existence results we introduce a control function $\mathbf{r}=\left(r_1,\cdots ,r_n \right)^T:[0,\infty)\times {\mathbb R}^n\rightarrow {\mathbb R}^n$. In general the function $\mathbf{r}$ is chosen in a class of bounded functions with bounded spatial derivatives up to second order which are Lipschitz continuous with respect to time. This is the regularity of the control function which we defined in \cite{KNS} and also in \cite{K3}. More precisely, those control functions are Lipschitz at times equal to the time step numbers $l\geq 1$ (with respect to transformed time coordinates $\rho_l\tau=t$) and are differentiable with respect to time elsewhere. The simplified control functions $r_i,~1\leq i\leq n$ defined later in this paper have the advantage that they are time differentiable for all times. The equation for the controlled velocity function 
\begin{equation}
\mathbf{v}^{r}:=\mathbf{v}+\mathbf{r},
\end{equation}
with $\mathbf{v}=\left(v_1,\cdots v_n\right)^T$ is more complicated than the original incompressible Navier-Stokes equation, of course. However, we may choose the control functions $r_i,~1\leq i\leq n$, and using the semi-group property of the operator we may choose the control function $\mathbf{r}$ time step by time step - with certain restrictions of course.  Indeed, the control functions $r_i,~1\leq i\leq n$, are constructed inductively with respect to the time step number $l\geq 1$ using the information of the final data of the previous time step $l-1$, and of the initial data $\mathbf{h}$ at the first time step. In a direct approach the effect of an appropriate choice is that at each time step is that we have to solve inhomogeneous local incompressible Navier-Stokes equations with consumption source terms on the right side, which control global boundedness. As long as the control function is itself globally bounded ( an upper bound for the control function which is linear with respect to time suffices) we solve a problem which is equivalent to the incompressible Navier-Stokes equation.  It can be shown then within the analysis of the scheme that the control functions are indeed globally bounded and satisfy some other convenient properties. In case of simplified control functions without consumption functions we prove the existence of a global linear bound of the control functions. Let us look at these ideas in more detail.
Note that the function $\mathbf{v}^r=\left(v^r_1,\cdots v^r_n\right)^T=\left(v_1+r_1,\cdots v_n+r_n\right)^T$
satisfies the equation   
\begin{equation}\label{Navleray}
\left\lbrace \begin{array}{ll}
\frac{\partial v^r_i}{\partial t}-\nu\sum_{j=1}^n \frac{\partial^2 v^r_i}{\partial x_j^2} 
+\sum_{j=1}^n v^r_j\frac{\partial v^r_i}{\partial x_j}=\\
\\
\frac{\partial r_i}{\partial t}-\nu\sum_{j=1}^n \frac{\partial^2 r_i}{\partial x_j^2} 
+\sum_{j=1}^n r_j\frac{\partial v^r_i}{\partial x_j}+\sum_{j=1}^n v^r_j\frac{\partial r_i}{\partial x_j}-\sum_{j=1}^n r_j\frac{\partial r_i}{\partial x_j}
\\
\\+\int_{{\mathbb R}^n}\left( \frac{\partial}{\partial x_i}K_n(x-y)\right) \sum_{j,k=1}^n\left( v^r_{k,j}v^r_{j,k}\right) (t,y)dy\\
\\
-2\int_{{\mathbb R}^n}\left( \frac{\partial}{\partial x_i}K_n(x-y)\right) \sum_{j,k=1}^n\left( v^r_{k,j}r_{j,k}\right) (t,y)dy\\
\\
-\int_{{\mathbb R}^n}\left( \frac{\partial}{\partial x_i}K_n(x-y)\right) \sum_{j,k=1}^n\left( r_{k,j}r_{j,k}\right) (t,y)dy,\\
\\
\mathbf{v}^r(0,.)=\mathbf{h}.
\end{array}\right.
\end{equation}
If we can solve this equation for controlled velocity functions $v^r_i,~1\leq i\leq n$ which are globally Lipschitz in time and have bounded spatial derivatives of second order  for an appropriate control function space $R$ (where we construct $r\in R$ time-step by time-step), then we can prove that the uncontrolled summand $v_i\in C^{1,2}\left(\left[0,\infty\right)\times {\mathbb R}^n\right)$ for $1\leq i\leq n$ of the controlled value function $\mathbf{v}^r$ is a global classical solution of the incompressible Navier Stokes equation. The construction is done time-step by time step on domains $\left[l-1,l\right]\times {\mathbb R}^n,~l\geq 1$, where for $1\leq i\leq n$ the restriction of the control function component $r_i$ to $\left[l-1,l\right]\times {\mathbb R}^n$ is denoted by 
$r^l_i$. Here, the local equation is defined in terms of transformed time coordinates $t=\rho_l\tau$. Note that we should have a time-step size greater or equal to $\rho_l\sim \frac{1}{l}$ at each time step in order to obtain a global scheme. In the discretized scheme we denote time by $\tau$ without a time step number index $l$ for simplicity, i.e., we shall know from the context of the equation that $\tau\in [l-1,l]$ at each time step $l$. The local functions $v^{r,\rho,l}_i$ with $v^{r,\rho,l}_i(\tau,x)=v^{r,l}_i(t,x)$ are defined inductively on $\left[l-1,l\right]\times {\mathbb R}^n$ along with the control function $r^l$ via the Cauchy problem for
\begin{equation}
\mathbf{v}^{r,\rho,l}:=\mathbf{v}^{\rho,l}+\mathbf{r}^l.
\end{equation}
Here, $\mathbf{v}^{\rho,l}=\left(v^{\rho,l}_1,\cdots ,v^{\rho,l}_n \right)^T$ is the time transformed solution of the incompressible Navier Stokes equation (in Leray projection form) restricted to the domain
$\left[l-1,l\right]\times {\mathbb R}^n$, and where $v^{\rho,l}_i(\tau,x)=v^{l}_i(t,x)$ for $\tau\in [l-1,l]$, and $v^{l}_i,~1\leq i\leq n$ denotes the restriction of a solution of the incompressible Navier Stokes equation (in Leray projection form) to the domain
$\left[\sum_{m=1}^{l-1}\rho_m,\sum_{m=1}^l\rho_m\right]\times {\mathbb R}^n$. We do not add a superscript $\rho$ to the control function $\mathbf{r}^l$ for notational simplicity, as we consider control functions only in transformed time coordinates anyway in the following. Note that the local solution function at time-step $l\geq 1$,  
\begin{equation}
\mathbf{v}^{r,\rho,l}=\left(v^{\rho,l}_1+r^l_1,\cdots v^{\rho,l}_n+r^l_n\right)^T,
\end{equation}
satisfies the equation   
\begin{equation}\label{Navleraycontrolledlint}
\left\lbrace \begin{array}{ll}
\frac{\partial v^{r,\rho,l}_i}{\partial \tau}-\rho_l\nu\sum_{j=1}^n \frac{\partial^2 v^{r,\rho,l}_i}{\partial x_j^2} 
+\rho_l\sum_{j=1}^n v^{r,\rho,l}_j\frac{\partial v^{r,\rho,l}_i}{\partial x_j}=\\
\\
\frac{\partial r^l_i}{\partial \tau}-\rho_l\nu\sum_{j=1}^n \frac{\partial^2 r^l_i}{\partial x_j^2} 
+\rho_l\sum_{j=1}^n r^l_j\frac{\partial v^{r,\rho,l}_i}{\partial x_j}\\
\\
+\rho_l\sum_{j=1}^n v^{r,\rho,l}_j\frac{\partial r^l_i}{\partial x_j}-\rho_l\sum_{j=1}^n r^l_j\frac{\partial r^l_i}{\partial x_j}
\\
\\+\rho_l\int_{{\mathbb R}^n}\left( \frac{\partial}{\partial x_i}K_n(x-y)\right) \sum_{j,k=1}^n\left( \frac{\partial v^{r,\rho,l}_k}{\partial x_j}\frac{\partial v^{r,\rho,l}_j}{\partial x_k}\right) (\tau,y)dy\\
\\
-2\rho_l\int_{{\mathbb R}^n}\left( \frac{\partial}{\partial x_i}K_n(x-y)\right) \sum_{j,k=1}^n\left( \frac{\partial v^{r,\rho,l}_k}{\partial x_j}\frac{\partial r^l_j}{\partial x_k}\right) (\tau,y)dy\\
\\
-\rho_l\int_{{\mathbb R}^n}\left( \frac{\partial}{\partial x_i}K_n(x-y)\right) \sum_{j,k=1}^n\left( \frac{\partial r^l_k}{\partial x_j}\frac{\partial r^l_j}{\partial x_k}\right) (\tau,y)dy,\\
\\
\mathbf{v}^{r,\rho,l}(l-1,.)=\mathbf{v}^{r,\rho,l-1}(l-1,.).
\end{array}\right.
\end{equation}
At the first time step the function $r^1_i$ may be chosen equal to zero, but a choice proportional to $\mathbf{h}$ with positive proportionality is useful. For the simplified control function a natural choice is $\delta r^l_i=-\delta v^{\rho,1,1}_i$ with the possible convention that $r^{l-1}_i=r^0_i\equiv 0$. Similar for extensions of the control function with consumption terms. Since we have a local existence theory, we can do it either way. For a time step $l>1$ the functions $r^{l}_i,1\leq i\leq n$ are assumed to be chosen inductively, where the rule of choice is determined by the following considerations. Here we reconsider the ideas of \cite{KNS} first, and then we turn to the possibilities of a simplified control function which we consider in this paper as an alternative. The right side of the first equation in (\ref{Navleraycontrolledlint}) is determined by the choice of the functions $r^l_i$. Note that we do not know the solution function $v^{r,\rho,l}_i,~1\leq i\leq n$ at the beginning of time step $l\geq 1$, which appears on the right side.
However, since we proceed in small time steps, we know this function approximately (if certain local results hold which we shall obtain later). Hence, we approximate the functions $v^{r,\rho,l}_i$ on the right side of (\ref{Navleraycontrolledlint}) in order to choose an equation which determines $r^l_i$, where at time step $l\geq 1$ the right side of (\ref{Navleraycontrolledlint}) may be modified by substituting $v^{r,\rho,l}_i$ by the data from the last time step, i.e., by the function $v^{r,\rho,l-1}_i(l-1,.)$. In order to control the growth with respect to time we shall observe that it suffices to define $\delta r^l_i=-\delta v^{r,\rho,l,1}_i$ for the control function increment at time step $l\geq 1$. However, the simplest control function which leads to global regular upper bounds (first in $C^0\left([0,T],H^m\cap C^m\right)$ for arbitrary time $T>0$ for the controlled velocity function components $v^r_i$ and the control function components $r_i$), and then in  $C^1\left([0,T], H^m\cap C^m \right)$ for the uncontrolled velocity function components $v_i$) may a control function increment of the form
\begin{equation}\label{simpleincrement1}
\delta r^l_i=\int_{l-1}^l\int_{{\mathbb R}^n}\frac{-v^{r,\rho,l-1}_i(l-1,y)}{C}G_l(\tau-s,x-y)dyds,
\end{equation}
where $G_l$ is the fundamental solution of $G_{l,\tau}-\rho_l\nu\Delta G_l=0$  on $[l-1,l]\times {\mathbb R}^n$, and where $C>0$ is some constant which is actually an upper bound of the controlled velocity value function with respect to some regular norm.
We shall observe that both definitions of control function increments lead to a linear upper bound with respect to time of the control function and of the controlled velocity functions. In order to obtain a global upper bound which is independent of the time step number $l\geq 1$ we may in addition introduce 'source term functions' related to the equations of the control functions, and which are determined time-step by time-step, such that they serve as consumption terms of growth (of the solution) at each time step. The reader which is mainly interested in global regular existence of the uncontrolled Navier Stokes equation may skip the following notes on more sophisticated control functions which lead to regular upper bounds which are independent of the time step number. In the following keep in mind the effect of the simple control function scheme with control function increments as in (\ref{simpleincrement1}), where you may start at the first time step with $r^1_i(0,.):=\frac{h_i}{C}$ with $C\geq {\big |}h_i{\big |}_{H^m\cap C^m}$ for some $m\geq 2$ (if we want to construct global regular upper bounds for the velocity componenet functions in $C^1\left([0,\infty),H^m\cap C^m \right)$: for small time step size $\rho_l$ there is at most linear growth of control function while the upper bound for the controlled velocity functions evaluated at time $l$, i.e. ${\big |}v^{r,\rho,l}_i(l,.){\big |}_{H^m\cap C^m}\leq C$, is preserved as $l$ increases. This means that we can use a constant time step size for this simple scheme.  The effect after a finite number of time steps is that depending on the choice of $C$ and $\rho$ at space-time points $(\tau,x)$ where the controlled velocity function component values $v^{r,\rho,l}_i(\tau,x)$ and the control function values $r^l_i(\tau,x)$ have different signs, the control function value becomes small compared to the constant $C$ if a) the upper bound $C>0$ is large and $b)$ the time step size $\rho>0$ is small on the scale of $C$. For example, as $\rho\sim \frac{1}{C^k}$ for some positive integer $k\geq 3$ the controlled velocity function values become smaller at these points as $k$ increases (for fixed proportionality). This observation can be used in order to construct global regular upper bounds which are stronger with respect to time. However, as we shall see, it is quite an effort to get a stronger result than we get with the simple control function with increments as in (\ref{simpleincrement1}). Note that with this control function we get even a logarithmic upper bound withrespect to time if we have a decreasing time step size $\rho_l\sim \frac{1}{C^kl}$.  Furthermore, we shall also discuss the relation to an alternative construction of time-independent regular global upper bounds via auto-controlled schemes sketched in \cite{KAC}. The auto-controlled schemes simplify some steps of the construction, but, we shall observe that external control functions are of independent interest as they allow us to prove the existence of global regular upper bounds for a larger class of models.   Note that the following ideas are quite closely related to another method where we first define the equation for the control functions $r^l_i$ such that a solution for $r^l_i$ leads to some source terms on the right side of of (\ref{Navleraycontrolledlint}), which are then constructed time-step by time-step such that they control the growth of the controlled velocity function. These source terms serve as 'growth consumption terms', and are denoted by $\phi^l_i,~1\leq i\leq n$. Accordingly, we derive an equation from the right side of (\ref{Navleraycontrolledlint}) of the form
\begin{equation}\label{controllint}
\left\lbrace \begin{array}{ll}
\frac{\partial r^l_i}{\partial \tau}-\rho_l\nu\sum_{j=1}^n \frac{\partial^2 r^l_i}{\partial x_j^2} 
+\rho_l\sum_{j=1}^n r^l_j\frac{\partial v^{r,\rho,l-1}_i}{\partial x_j}\\
\\
+\rho_l\sum_{j=1}^n v^{r,\rho,l-1}_j\frac{\partial r^l_i}{\partial x_j}-\rho_l\sum_{j=1}^n r^l_j\frac{\partial r^l_i}{\partial x_j}
\\
\\+\rho_l\int_{{\mathbb R}^n}\left( \frac{\partial}{\partial x_i}K_n(x-y)\right) \sum_{j,k=1}^n\left( \frac{\partial v^{r,\rho,l-1}_k}{\partial x_j}\frac{\partial v^{r,\rho,l-1}_j}{\partial x_k}\right) (l-1,y)dy\\
\\
-2\rho_l\int_{{\mathbb R}^n}\left( \frac{\partial}{\partial x_i}K_n(x-y)\right) \sum_{j,k=1}^n\left( \frac{\partial v^{r,\rho,l-1}_k}{\partial x_j}(l-1,y)\frac{\partial r^l_j}{\partial x_k}(\tau,y)\right)dy\\
\\
-\rho_l\int_{{\mathbb R}^n}\left( \frac{\partial}{\partial x_i}K_n(x-y)\right) \sum_{j,k=1}^n\left( \frac{\partial r^l_k}{\partial x_j}\frac{\partial r^l_j}{\partial x_k}\right) (\tau,y)dy=\phi^l_i,\\
\\
\mathbf{r}^l(l-1,.)=\mathbf{r}^{l-1}(l-1,.).
\end{array}\right.
\end{equation}
Note that at this time step $l\geq 1$ the functions $v^{r,\rho,l-1}_i$ are taken from the previous time step $l-1$, and they are evaluated at $l-1$ with respect to the time variable., i.e., $v^{r,\rho,l-1}_i=v^{r,\rho,l-1}_i(l-1,.)$, where for $l=1$ we take $v^{r,\rho,0}_i(0,.)=h_i(.)$. All functions with superscript $l-1$ appearing in problems at time step $l\geq 1$ are always intended to be evaluated at time $\tau=l-1$. The equation in (\ref{controllint}) for the control function looks as complicated as a Navier-Stokes equation. However, we need to solve it locally at most (small time-step size $\rho_l$), and we shall see how we can do this (and is also well-known by other reasoning). Later we shall observe that we can avoid solving a local Navier Stokes type equation for the control function. We could also linearize the equation for (\ref{controllint}). From a constructive perspective this is quite natural: we can construct the control function by starting with the construction of a approximation $ \mathbf{r}^{*,l}=\left( r^{*,l}_1,\cdots ,r^{*,l}_n\right)^T$ which is determined inductively at time step $l\geq 1$ by the equation
\begin{equation}\label{controllint*}
\left\lbrace \begin{array}{ll}
\frac{\partial r^{*,l}_i}{\partial \tau}-\rho_l\nu\sum_{j=1}^n \frac{\partial^2 r^{*,l}_i}{\partial x_j^2} 
+\rho_l\sum_{j=1}^n r^{*,l-1}_j\frac{\partial v^{r,\rho,l-1}_i}{\partial x_j}\\
\\
+\rho_l\sum_{j=1}^n v^{r,\rho,l-1}_j\frac{\partial r^{*,l-1}_i}{\partial x_j}-\rho_l\sum_{j=1}^n r^{*,l-1}_j\frac{\partial r^{*,l-1}_i}{\partial x_j}
\\
\\+\rho_l\int_{{\mathbb R}^n}\left( \frac{\partial}{\partial x_i}K_n(x-y)\right) \sum_{j,k=1}^n\left( \frac{\partial v^{r,\rho,l-1}_k}{\partial x_j}\frac{\partial v^{r,\rho,l-1}_j}{\partial x_k}\right) (l-1,y)dy\\
\\
-2\rho_l\int_{{\mathbb R}^n}\left( \frac{\partial}{\partial x_i}K_n(x-y)\right) \sum_{j,k=1}^n\left( \frac{\partial v^{r,\rho,l-1}_k}{\partial x_j}(l-1,y)\frac{\partial r^{*,l-1}_j}{\partial x_k}(l-1,y)\right)dy\\
\\
-\rho_l\int_{{\mathbb R}^n}\left( \frac{\partial}{\partial x_i}K_n(x-y)\right) \sum_{j,k=1}^n\left( \frac{\partial r^{*,l}_k}{\partial x_j}\frac{\partial r^{*,l}_j}{\partial x_k}\right) (l-1,y)dy=\phi^{*,l}_i,\\
\\
\mathbf{r}^{*,l}(l-1,.)=\mathbf{r}^{*,l-1}(l-1,.).
\end{array}\right.
\end{equation}
Note that the difference $r^l_i-r^{*,l}_i$ becomes small with small time-step size $\rho_l$ while the source functions $\phi^l_i$ or $\phi^{*,l}_i$ do not depend on the step size (they can be chosen this way).
Anyway the following idea of a consumption function based on these source functions remains the same.  
We may choose consumption functions $\phi^l_i$ according to the growth which we observe at the previous time step. However, we may improve on this idea of a consumption function where we use the fact that we can solve the incompressible Navier Stokes equation locally in time. This allows us to choose the control function at time $l-1$ by looking at the behavior of the uncontrolled velocity function with controlled data at time $l$. This way we can a) ensure that {\it after} each time step (not during each time step in general, but this does not matter) we can ensure that the control function and the controlled velocity function have the same sign. This ensures that the controlled velocity function and the control function are uniformly bounded over time in a more direct argument. Furthermore, we can b) estimate the long time behavior as time goes to infinity. We call this extended scheme the scheme with small foresight.  First, note that it is convenient that we can avoid the involved equations for the controlled velocity functions $v^{r,\rho,l}_i,~1\leq i\leq n$ and for the control functions $r^{l}_i,~1\leq i\leq n$. If for $l\geq 1$ the data $v^{r,\rho,l-1}_i(l-1,.),~1\leq i\leq n$ and $r^{l-1}_i(l-1,.)$ are determined, then we may first compute that uncontrolled velocity function at time step $l\geq 1$ with controlled data. This is the solution $v^{r^{l-1},\rho,l}_i,~1\leq i\leq n$ of the incompressible Navier Stokes equation on $[l-1,l]\times {\mathbb R}^n$ with data $v^{r^{l-1},\rho,l-1}_i(l-1,.)=v^{r,\rho,l-1}_i(l-1,.),~1\leq i\leq n$. Note that the controlled velocity function at time step $l\geq 1$ is then determined by equations
\begin{equation}
v^{r,\rho,l}_i(\tau,x)=v^{r^{l-1},\rho,l}_i(\tau,x)+\delta r^l_i(\tau,x),~1\leq i\leq n,~(\tau,x)\in [l-1,l]\times {\mathbb R}^n.
\end{equation}
In the following we call $v^{r^{l-1},\rho,l}_i,~1\leq i\leq n$ the uncontrolled velocity function with controlled data at time step $l\geq 1$.
Then we may analyze the situation and consider $4n$ sets
\begin{equation}\label{partition1}
\begin{array}{ll}
V^{r^{l-1},l,i}_{v+r+}:=\left\lbrace x\in {\mathbb R}^n|v^{r^{l-1},\rho,l}_i(l,x)> 0~\& ~r^{l-1}_i(l-1,x)>0\right\rbrace\\
\\
V^{r^{l-1},l,i}_{v+r-}:=\left\lbrace x\in {\mathbb R}^n|v^{r^{l-1},\rho,l}_i(l,x)> 0~\& ~r^{l-1}_i(l-1,x)<0\right\rbrace\\
\\
V^{r^{l-1},l,i}_{v-r+}:=\left\lbrace x\in {\mathbb R}^n|v^{r^{l-1},\rho,l}_i(l,x)< 0~\& ~r^{l-1}_i(l-1,x)>0\right\rbrace\\
\\
V^{r^{l-1},l,i}_{v-r-}:=\left\lbrace x\in {\mathbb R}^n|v^{r^{l-1},\rho,l}_i(l,x)< 0~\& ~r^{l-1}_i(l-1,x)<0\right\rbrace.
\end{array}
\end{equation}
For each $l-1$ and each $1\leq i\leq n$ the latter four sets define a partition of the set ${\mathbb R}^n\setminus Z\left( v^{r,\rho,l-1}_i(l-1,.)\right) $, where
\begin{equation}\label{partition2}
Z\left( v^{r^{l-1},\rho,l}_i(l,.)\right) :=\left\lbrace x|v^{r^{l-1},\rho,l}_i(l,x)=0\right\rbrace. 
\end{equation}
Next for all $1\leq i\leq n$ and each time step number $l\geq 1$ define a function $g^l_i:{\mathbb R}^n\rightarrow {\mathbb R}$ by
\begin{equation}
g^l_i(y):=\left\lbrace \begin{array}{ll}
\frac{v^{r^{l-1},\rho,l}_i(l,y)}{C}+\frac{r^{l-1}_i(l-1,y)}{C^2}~\mbox{if}~y\in V^{r^{l-1},l,i}_{v+r+}\cup V^{r^{l-1},l,i}_{v-r-}\\
\\
2 v^{r^{l-1},\rho,l}_i(l,y)+\frac{r^{l-1}_i(l-1,y)}{C^2} \mbox{ else. }
\end{array}\right.
\end{equation}

The idea behind this definition is the following: if for some $1\leq i\leq n$ and the control function at time $l-1$ and the uncontrolled velocity function data at $(l-1,x)$ have the same sign, then we may define the control function increment at time step $l$ at each point $(\tau,x)\in [l-1,l]\times {\mathbb R}^n$ close to a value which  proportional to a weighted sum of the the negative of the controlled velocity function $v^{r^{l-1},\rho,l}_i(l,x)$ (looking slightly forward) and of the control function $r^{l-1}_i(l-1)$, i.e., we may define for all $(\tau,x)\in [l-1,l]\times V^{r^{l-1},l,i}_{v+r+}\cup V^{r^{l-1},l,i}_{v-r-}$ by
\begin{equation}\label{onequalsign}
\delta r^{l}_i(\tau,x)=-\int_{l-1}^{\tau}\int_{{\mathbb R}^n}g^l_i(y) p_l(\tau-s,x-y)dyds,
\end{equation}
where $p_l$ is the fundamental solution of
\begin{equation}
\frac{\partial}{\partial \tau}p-\rho_l\nu \Delta p=0.
\end{equation}
Note that for small time step size $\rho_l$ the value in (\ref{onequalsign}) is close to
\begin{equation}\label{onequalsign*}
-\int_{l-1}^{\tau}\int_{{\mathbb R}^n}\left( \frac{v^{r^{l-1},\rho,l}_i(l,y)}{C}+\frac{r^{l-1}_i(l-1,y)}{C^2}\right) p_l(\tau-s,x-y)dyds,
\end{equation}
The smoothing by convolution with a local heat kernel is useful especially if we define control functions depending on a partition as in (\ref{partition1}) and (\ref{partition2}) above which would render the control function to be only Lipschitz if we defined it without smoothing. Note that a definition on the domain of equal signs as in (\ref{onequalsign})
 will ensure that for all $x\in V^{r^{l-1},l,i}_{v+r+}\cup V^{r^{l-1},l,i}_{v-r-}$ and small time step size $\rho_l>0$ and large that for
\begin{equation}
{\big |}v^{r,\rho,l-1}_i(l-1,x)+r^{l}_i(l-1,x){\big |}\geq \frac{C}{2},~{\big |}r^{l}_i(l-1,x){\big |}\leq C^2+1
\end{equation}
we have
\begin{equation}
{\big |}v^{r,\rho,l}_i(l,x)+r^{l}_i(l,x){\big |}\leq {\big |}v^{r,\rho,l-1}_i(l-1,x)+r^{l}_i(l-1,x){\big |}.
\end{equation}

Next,  if for some $1\leq i\leq n$ and some $l$ a uncontrolled velocity function value with controlled data $v^{r^{l-1},\rho,l}_i(l,x)$ at time $l$ and the control function  $r^{l-1}_i(l-1,x)$ at time $l-1$ have different signs for some $x\in {\mathbb R}^n$, then we may define again
\begin{equation}\label{ondifferentsign}
\delta r^{l}_i(\tau,x)=-\int_{l-1}^{\tau}\int_{{\mathbb R}^n}g^l_i(y) p_l(\tau-s,x-y)dyds,
\end{equation}
for all $(\tau,x)\in [l-1,l]\times {\mathbb R}^n$. 
We note that for small time step size $\rho_l$ the value in (\ref{ondifferentsign}) is close to
\begin{equation}\label{ondifferentsign*}
\delta r^{l}_i(\tau,x)=-\int_{l-1}^{\tau}\int_{{\mathbb R}^n}\left( 2 v^{r^{l-1},\rho,l}_i(l,y)+\frac{r^{l-1}_i(l-1,y)}{C^2}\right) p_l(\tau-s,x-y)dyds,
\end{equation}
for all $(\tau,x)\in [l-1,l]\times {\mathbb R}^n$.

Well this more sophisticated control function with small foresight may be preferable from an analytical point of view, as it seems to be easier to study long time behavior with such descriptions. From a more numerical point of view we may prefer simpler control functions which do not depend on the computation of solutions of nonlinear equations (even locally) , and with some additional amount of work we may also obtain global uniform upper bounds with simpler control functions. For example, for $(\tau,x)\in [l-1,l)\times {\mathbb R}$ we may choose a source function
\begin{equation}\label{philli}
\phi^l_i(\tau,x):=-\frac{v^{r,\rho,l-1}_i(l-1,.)}{C}-\frac{r^{l-1}_i(l-1,.)}{C^2}
\end{equation}
for some constant $C$ which is constructed as a global upper bound within the construction. Note that (\ref{philli}) can also be written in the form
\begin{equation}\label{philli2}
\phi^l_i(\tau,x):=-\frac{v^{r,\rho,l-1}_i(l-1,.)+\frac{1}{C}r^{l-1}_i(l-1,.)}{C}
\end{equation}
which implies that it is somehow a scheme which is essentially equivalent to the scheme without the additional summand $-\frac{r^{l-1}_i(l-1,.)}{C^2}$. However use the representation in (\ref{philli}) in order to argue that a global linear upper bound for the velocity functions can be improved in order to get an uniform global upper bound. 
\begin{rem}\label{remarksimplecontrol}
In \cite{KNS,K3} we used source functions with a different weight. Indeed, in order to get a global linear bound for the solution and its derivatives it is sufficient to consider the source function 
\begin{equation}
\phi^l_i(\tau,x):=-\frac{v^{r,\rho,l-1}_i(l-1,.)}{C}
\end{equation}
In (\cite{KHyp}) we remarked that for the classical Navier Stokes equation system in this paper it is possible to define the simple control function increments via
\begin{equation}
\delta r^l_i(\tau,x):=-\int_{l-1}^{\tau}\int_{{\mathbb R}^n}\frac{v^{r,\rho,l-1}_i(l-1,.)}{C}p_l(s-(l-1),x-y)dyds
\end{equation}
The reason is for this possibility will be given below.
\end{rem}
Note that in this choice the source functions are independent of time. This means that we have bounded jumps in time for the source terms, but the construction of local solutions for $r^l_i$ involve time integrated source terms and this leaves us with control functions which are global Lipschitz in time and locally bounded and smooth within each time step. Note that the consumption function does not have a factor $\rho_l$ which may be small compared to $\frac{1}{C}$. Then we may use a step size $\rho_l$ which is small compared to $\frac{1}{C^2}$ and a local time contraction result may show that the consumptive behavior of the source functions $\phi^l_i$ is preserved for the equation (\ref{Navleraycontrolledlint}) if the control function is chosen according to the ideas related to (\ref{controllint}). Similar for a linearized equation for $r^l_i$. This looks rather complicated (also from a computational point of view), but together with a local contraction result this construction leads a globally bounded solution which is globally Lipschitz in time and smooth with respect to the spatial variables. And a little additional argument leads us to global regular solution of the original Navier-Stokes equation (cf. \cite{KNS}). In this paper we propose a simpler choice of the control function but the general idea is of the same origin. The sketched argument should make plausible our statement that local contraction results are crucial.
Here the local contraction result for the uncontrolled system in $C^0\left((l-1,l), H^{m}\right)$ for $m\geq 2$ is essential, because in the equations for the functional increments some source terms cancel and the additional terms are linear functionals of the control functions $r^l_i$. In \cite{KB2} we considered these contraction results in more detail. Note, however, that we choose the control function $r^l_i$ once at each time step, which means that for $k\geq 2$
\begin{equation}
\delta v^{r,\rho,l,k}_i=v^{r,\rho,l,k}_i-v^{r,\rho,l,k-1}_i=v^{\rho,l,k}_i-v^{\rho,l,k-1}_i
\end{equation}
such that a refined contraction for the higher order terms ($k\geq 2$)
\begin{equation}
{\big |}\delta v^{r,\rho,l,k}_i{\big |}_{C^0\left((l-1,l), H^{m}\right)}\lesssim \frac{1}{\sqrt{l}}{\big |}\delta v^{r,\rho,l,k-1}_i{\big |}_{C^0\left((l-1,l), H^{m}\right)}
\end{equation}
leads to a transition from local to global upper bounds as well.
Note that local existence results follow rather directly from the contraction results. In \cite{KB2} we considered local contraction results for all $l\geq 1$ and $1\leq i\leq n$, and for the functional series
\begin{equation}
v^{\rho, l}_i:=v^{\rho,l-1}_i(l-1,.)+\sum_{k=1}^{\infty} \delta v^{\rho,l,k}_i,
\end{equation}
where for $1\leq i\leq n$  recall that we assume that $v^{\rho,l-1}_i(l-1,.)\in H^2\cap C^2$, 
$\delta v^{\rho,l,k}_i= v^{\rho,l,k}_i- v^{\rho,l,k-1}_i$, $v^{\rho,l,-1}:=v^{\rho,l-1}_i$, and the functions $v^{\rho,l,k}_i$ satisfy the equations
\begin{equation}\label{Navleray2}
\left\lbrace \begin{array}{ll}
\frac{\partial v^{\rho,l,k}_i}{\partial \tau}-\rho_l\nu\sum_{j=1}^n \frac{\partial^2 v^{\rho,l,k}_i}{\partial x_j^2} 
+\rho_l\sum_{j=1}^n v^{\rho,l,k-1}_j\frac{\partial v^{\rho,l,k}_i}{\partial x_j}=\\
\\ \hspace{1cm}\rho_l\sum_{j,m=1}^n\int_{{\mathbb R}^n}\left( \frac{\partial}{\partial x_i}K_n(x-y)\right) \sum_{j,m=1}^n\left( \frac{\partial v^{\rho,l,k-1}_m}{\partial x_j}\frac{\partial v^{\rho,l,k-1}_j}{\partial x_m}\right) (\tau,y)dy,\\
\\
\mathbf{v}^{\rho,l,k}(l-1,.)=\mathbf{v}^{\rho,l-1}(l-1,.),
\end{array}\right.
\end{equation}
and on the domain $[l-1,l]\times {\mathbb R}^n$ for $l\geq 1$. For $l=1$ it is clear that the components of the function $\mathbf{v}^{\rho,l-1}(l-1,.)=\mathbf{h}(.)$ are in $H^2\cap C^2$. Consider the functional increments $\delta v^{r,\rho,l,k}_i=v^{r,\rho,l,k}_i-v^{r,\rho,l,k-1}_i$ where for $k=1$ we define $v^{r,\rho,0,l}_i=v^{r,\rho,l-1}_i(l-1,.)$.
Let us assume that we have a contraction result of the form
\begin{equation}\label{contract}
\begin{array}{ll}
\max_{i\in \left\lbrace 1,\cdots ,n\right\rbrace }|\delta v^{\rho,l,k}_i|_{C^0\left((l-1,l), H^{m}\right)}
\leq \frac{1}{2}\max_{i\in \left\lbrace 1,\cdots ,n\right\rbrace }|\delta v^{\rho,l,k-1}|_{C^0\left((l-1,l), H^{m}\right)},
\end{array} 
\end{equation}
and local contractions of the form
\begin{equation}\label{contract2}
\begin{array}{ll}
\max_{i\in \left\lbrace 1,\cdots ,n\right\rbrace }|\delta v^{\rho,l,k}_i|_{C^1\left((l-1,l), H^{m}\right)}
\leq \frac{1}{2}\max_{i\in \left\lbrace 1,\cdots ,n\right\rbrace }|\delta v^{\rho,l,k-1}|_{C^1\left((l-1,l), H^{m}\right)},
\end{array} 
\end{equation}
and where the time step size depends only on the viscosity $\nu >0$ the dimension $n$ (we assumed $n=3$ at several steps of the argument), and the size of the initial data $v^{r,\rho,l-1}_i(l-1,.)$.
Recall that for $l=1$ we define
$\mathbf{v}^{\rho,l-1}(l-1,.)=\mathbf{h}(.)$. Next we consider the sequence $\left( v^{\rho,l,k}_i\right)_{k}$ and represent each member of this list in the form
\begin{equation}
v^{\rho,l,k}_i=v^{\rho,l-1}_i(l-1,.)+\sum_{p=1}^{k}\delta v^{\rho,l,p}_i.
\end{equation}
Classical analysis for scalar parabolic equations shows that 
\begin{equation}
v^{\rho,l,k}_i\in C^{1,2}\left( [l-1,l]\times {\mathbb R}^n\right),
\end{equation}
assuming that $v^{\rho,l-1}_i(l-1,.)\in C^2$ (although a weaker assumption may be used here). 
Hence we have
\begin{equation}
v^{\rho,l,k}_i\in C^{1,2}\left( [l-1,l]\times {\mathbb R}^n\right)\cap C^1\left( [l-1,l],H^2\right) 
\subset C^1\left( [l-1,l], C^2_0\left({\mathbb R}^n \right)\right) , 
\end{equation}
where the latter space $C^2_0\left({\mathbb R}^n \right)$ is the space of twice differentiable functions where the functions themselves and there derivatives up to second order vanish at spatial infinity. Note that the latter space is a closed function space (this is true for $C^2_0(X)$ if $X$ is a locally compact Hausdorff space, and this certainly holds for ${\mathbb R}^n$ equipped with the standard topology). From the contraction results it follows that the sequence $(v^{\rho,l,k}_i)_k$ is a Cauchy sequence (for all $1\leq i\leq n$) and has a classical limit $v^{\rho,l}_i\in C^1\left( [l-1,l], C^2_0\left({\mathbb R}^n \right)\right) $ for all $1\leq i\leq n$. Finally we show that this function is indeed a classical local solution. For all $1\leq i\leq n$ we have
\begin{equation}\label{Navleray22222}
\left\lbrace \begin{array}{ll}
\frac{\partial v^{\rho,l,k}_i}{\partial \tau}-\rho_l\nu\sum_{j=1}^n \frac{\partial^2 v^{\rho,l,k}_i}{\partial x_j^2} 
+\rho_l\sum_{j=1}^n v^{\rho,l,k}_j\frac{\partial v^{\rho,l,k}_i}{\partial x_j}=\\
\\ \rho_l\sum_{j=1}^n \delta v^{\rho,l,k}_j\frac{\partial v^{\rho,l,k}_i}{\partial x_j}\\
\\
+\rho_l\sum_{j,m=1}^n\int_{{\mathbb R}^n}\left( \frac{\partial}{\partial x_i}K_n(x-y)\right) \sum_{j,m=1}^n\left( \frac{\partial v^{\rho,l,k}_m}{\partial x_j}\frac{\partial v^{\rho,l,k}_j}{\partial x_m}\right) (\tau,y)dy\\
\\
-\rho_l\sum_{j,m=1}^n\int_{{\mathbb R}^n}\left( \frac{\partial}{\partial x_i}K_n(x-y)\right) \sum_{j,m=1}^n\left( \frac{\partial \delta v^{\rho,l,k}_m}{\partial x_j}\frac{\partial v^{\rho,l,k}_j}{\partial x_m}\right) (\tau,y)dy\\
\\
-\rho_l\sum_{j,m=1}^n\int_{{\mathbb R}^n}\left( \frac{\partial}{\partial x_i}K_n(x-y)\right) \sum_{j,m=1}^n\left( \frac{\partial \delta v^{\rho,l,k}_m}{\partial x_j}\frac{\partial v^{\rho,l,k-1}_j}{\partial x_m}\right) (\tau,y)dy,\\
\\
\mathbf{v}^{\rho,l,k}(l-1,.)=\mathbf{v}^{\rho,l-1}(l-1,.).
\end{array}\right.
\end{equation}
Since we shall see that $\delta v^{\rho,l,k}_i\downarrow 0$ with respect to strong norms, and a fortiori with respect to the norm 
\begin{equation}
 |f|_{C^{1,2}\left( [l-1,l]\times {\mathbb R}^n\right)  }:=\sup_{s\in[l-1,l]\times {\mathbb R}^n}\left({\Big |}f(s,.){\Big |}+{\Big |}D_{\tau}f(s,.){\Big |}+\sum_{|\alpha|\leq 2}{\Big |}D^{\alpha}_{x}f(s,.){\Big |}\right) 
 \end{equation}
(for all multiindices $\alpha=(\alpha_1,\cdots,\alpha_n)$ with nonnegative entries $\alpha_i\geq 0$ and with $|\alpha|=\sum_{i=1}^n\alpha_i$), the function $v^{\rho,l}_i,~1\leq i\leq n$ is indeed a local classical solution of the incompressible Navier Stokes equation.

\section{A simplified scheme (local and global)}  
The considerations of the preceding section motivate the search for local contraction results with respect to strong norms. The description of the transition from local to global schemes via control functions $r^l_i$ as considered in \cite{K3,KNS} looks rather complicated, but  can it be simplified. The simplified control function for a global scheme will be defined in the last section. In this section we define a simplified local scheme (simplified compared to the scheme considered in \cite{KB2}). We show that we can introduce simplified control functions and avoid the solution of rather complicated local equations for the control function (as indicated in remark (\ref{remarksimplecontrol}) above). We note that these considerations can be generalized to models with variable viscosity. The advantage of the scheme considered in this paper is that we do not need the adjoint of the fundamental solutions and a priori estimates but can deal with convolutions and Gaussians directly.  This simplifies several estimates. 
 For $m\geq 2$ and time step $l\geq 1$ and small $\rho_l\gtrsim \frac{1}{l}$ we consider a time-local functional scheme
\begin{equation}
v^{*,\rho, l}_i:=v^{*,\rho,l-1}_i(l-1,.)+\sum_{k=1}^{\infty} \delta v^{*,\rho,l,k}_i,
\end{equation}
where for $1\leq i\leq n$  $v^{*,\rho,l-1}_i(l-1,.)\in H^m\cap C^m$, and   
$\delta v^{*,\rho,l,k}= v^{\rho,l,k}_i- v^{*,\rho,l,k-1}_i$ along with $v^{*,\rho,l,-1}:=v^{*,\rho,l-1}_i(l-1,.)$, and where the functions $v^{*,\rho,l,k}_i$ satisfy the equations
\begin{equation}\label{Navleray*}
\left\lbrace \begin{array}{ll}
\frac{\partial v^{*,\rho,l,k}_i}{\partial \tau}-\rho_l\nu\sum_{j=1}^n \frac{\partial^2 v^{*,\rho,l,k}_i}{\partial x_j^2} 
=-\rho_l\sum_{j=1}^n v^{*,\rho,l,k-1}_j\frac{\partial v^{*,\rho,l,k-1}_i}{\partial x_j}\\
\\
+\rho_l\sum_{j,m=1}^n\int_{{\mathbb R}^n}\left( \frac{\partial}{\partial x_i}K_n(x-y)\right) \sum_{j,m=1}^n\left( \frac{\partial v^{*,\rho,l,k-1}_m}{\partial x_j}\frac{\partial v^{*,\rho,l,k-1}_j}{\partial x_m}\right) (\tau,y)dy,\\
\\
\mathbf{v}^{*,\rho,l,k}(l-1,.)=\mathbf{v}^{*,\rho,l-1}(l-1,.).
\end{array}\right.
\end{equation}
 We are interested in the case $n\geq 3$, and for $l=1$ we define $\mathbf{v}^{*,\rho,l-1}(l-1,.)=\mathbf{h}(.)$. The upper script $*$ indicates the difference to the scheme of the introduction: in the present scheme all the information of the convection term for the fist substep is taken from the previous time step. Especially, we are interested in the case $n= 3$, but our methods work in all dimension $n\geq 3$. For $n=2$ the Laplacian kernel is different such that this case should be considered separately. We shall assume $n\geq 3$ in the following where we emphasize that the estimates hold in the case $n=3$ and can be adapted (with some modifications) to the case $n=2$.  If $S$ denotes the right side source term of the first equation of (\ref{Navleray*}), i.e., if we define
\begin{equation}
\begin{array}{ll}
S(\tau,y):=-\rho_l\sum_{j=1}^n v^{*,\rho,l,k-1}_j(\tau,y)\frac{\partial v^{*,\rho,l,k-1}_i}{\partial x_j}(\tau,y)+\\
\\
\rho_l\sum_{j,m=1}^n\int_{{\mathbb R}^n}\left( \frac{\partial}{\partial x_i}K_n(y-z)\right) \sum_{j,m=1}^n\left( \frac{\partial v^{*,\rho,l,k-1}_m}{\partial x_j}\frac{\partial v^{*,\rho,l,k-1}_j}{\partial x_m}\right) (\tau,z)dz,
\end{array}
\end{equation}
then on the domain $\left[l-1,l\right]\times {\mathbb R}^n$ we see that we have the representation
 \begin{equation}\label{solrep1}
 \begin{array}{ll}
 v^{*,\rho,l,k}_i(\tau,x)=\int_{{\mathbb R}^3}v^{*,\rho,l-1}(l-1,y)G_l(\tau-(l-1),x-y)dy\\
 \\
 +\int_{l-1}^{\tau}\int_{{\mathbb R}^n}S(s,y)G_l(\tau-s,x-y)dyds,
 \end{array}
 \end{equation}
where we recall that $G_{l}$ is the fundamental solution of the heat equation $\frac{\partial u}{\partial t}-\rho_l\nu\Delta u=0$ on $[l-1,l]\times {\mathbb R}^n$. Compared to the previous scheme we observe the difference that the integral expressions in (\ref{solrep1}) are convolutions. A priori estimates are easier at hand, and we do not need the adjoint of the fundamental solution in order to shift derivatives. Nevertheless, keeping book of the additional source term we may adopt arguments of \cite{KB1} and \cite{KB2} and arrive at the same result by simpler considerations. 
Next we give the reason for the possibility of simplified control functions.  
Having defined the controlled value functions $v^{r,*\rho,l-1}_i(l-1,.)$ and the control functions $r^{l-1}_i(l-1,.)$ up to step $l-1$ we may consider a local uncontrolled iteration scheme for the functions $v^{r^{l-1},*,\rho,l,k}_i$ which satisfies the equations
\begin{equation}\label{Navleray*?}
\left\lbrace \begin{array}{ll}
\frac{\partial v^{r^{l-1},*,\rho,l,k}_i}{\partial \tau}-\rho_l\nu\sum_{j=1}^n \frac{\partial^2 v^{r^{l-1},*,\rho,l,k}_i}{\partial x_j^2} 
=-\rho_l\sum_{j=1}^n v^{r^{l-1},*,\rho,l,k-1}_j\frac{\partial v^{r^{l-1},*,\rho,l,k-1}_i}{\partial x_j}\\
\\
+\rho_l\sum_{j,m=1}^n\int_{{\mathbb R}^n}\left( \frac{\partial}{\partial x_i}K_n(x-y)\right) \sum_{j,m=1}^n\left( \frac{\partial v^{r^{l-1},*,\rho,l,k-1}_m}{\partial x_j}\frac{\partial v^{r^{l-1},*,\rho,l,k-1}_j}{\partial x_m}\right) (\tau,y)dy,\\
\\
\mathbf{v}^{r^{l-1},*,\rho,l,k}(l-1,.)=\mathbf{v}^{r,*,\rho,l-1}(l-1,.)=\mathbf{v}^{*,\rho,l-1}(l-1,.)+\mathbf{r}^{l-1}(l-1,.),
\end{array}\right.
\end{equation}
where $\mathbf{r}^{l-1}(l-1,.)=\left(r^{l-1}_1(l-1,.),\cdots ,r^{l-1}_n(l-1,.)\right)^T$.
Note that this is a local iteration scheme for an uncontrolled Navier Stokes equation with controlled data.
We may then apply local contraction results which are based on a standard incompressible Navier Stokes equation (not the complicated controlled Navier Stokes equation) in order to obtain solutions of (\ref{Navleray*?}), and then we may define
\begin{equation}
r^l_i=r^{l-1}_i(l-1,.)+\delta r^{l}_i,
\end{equation}
when the local iteration with respect to the local iteration index $k$ is completed. We then define
\begin{equation}
v^{r,*,\rho,l}_i(\tau,x)=v^{r^{l-1},*,\rho,l}_i(\tau,x)+\delta r^l_i(\tau,x) \mbox{ for all $(\tau,x)\in (l-1,l]\times {\mathbb R}^n$}.
\end{equation}

In order to get global regular existence results the simplest choice for the increment $\delta r^l_i$ may be as in (\ref{simpleincrement1}). The strategy is  to observe a) the persistence of a global upper bound for the controlled velocity value functions, i.e., to observe that for each order of regularity $m\geq 2$ we have a constant $C>0$ such that
\begin{equation}\label{controlled1}
{\big |}v^{r,\rho,l}_i(l,.){\big |}_{H^m\cap C^m}\leq C
\end{equation}
for all $1\leq i\leq n$ and for all time step numbers $l\geq 1$ (,i.e., the constant $C>0$ can be chosen independently of the time step number $l$), and then to show b) (in the case of a simple control function defined as in (\ref{simpleincrement1}) that
\begin{equation}\label{controlupperbound1}
{\big |}r^{l}_i(l,.){\big |}_{H^m\cap C^m}\leq \tilde{C}+l\tilde{C}
\end{equation}
for all $1\leq i\leq n$ and for some constant $\tilde{C}>0$ which is indpendent of the time step number $l\geq 1$ as well. The upper bound \ref{controlupperbound1} is obtained via the estimate for the controlled scheme in (\ref{controlled1}) together with the estimates for the increments (\ref{simpleincrement1}) (which follow straightforwardly. The uncontrolled scheme functions $v^{\rho,l}_i=v^{r,\rho,l}_i-r^l_i$ then satisfy
\begin{equation}\label{uncontrolled1}
{\big |}v^{\rho,l}_i(l,.){\big |}_{H^m\cap C^m}\leq C_0+lC_0
\end{equation}
for some constant $C_0>0$ which is independent of the time step number $l\geq 1$. Time-local existence via contraction then implies that
\begin{equation}
\sup_{\tau\in [l-1,l]}{\big |}v^{\rho,l}_i(l,.){\big |}_{H^m\cap C^m}\leq C_1+lC_1
\end{equation}
for some constant $C_1$ independent of the time step number $l\geq 1$. This leads to global existence for the uncontrolled incompressible Navier Stokes equation. The estimate (\ref{controlupperbound1}) can be improved using more sophisticated control function schemes (as described above) or by using an auto-controlled scheme or mixtures of both. We come back to this issue in section 5, and describe the application of various variations of control functions below. As we have introduced a simplified scheme now, we mention that the different notation $v^{r,*,\rho,l,k}_i$ for the simplified scheme and $v^{r,\rho,l,k}_i$ for the local scheme used in \cite{KB2} is relevant only for the time-local arguments. The local contraction results imply uniqueness in regular function space for regular data such that in the limit we always have
\begin{equation}
v^{r,\rho,l}_i=\lim_{k\uparrow \infty}v^{r,\rho,l,k}_i=v^{r,*,\rho,l}_i=\lim_{k\uparrow infty}v^{r,*\rho,l}_i.
\end{equation}
This means that in the context of the global arguments we can use $v^{r,\rho,l}_i$ and $v^{r,*,\rho,l}_i$ interchangeably, as these are two names for the same function. This reminds us that local contraction arguments using the adjoint of the fundamentals solution lead directly to generalisations of the results of this paper to models with variable viscosity. 
\section{Statement of local contraction result}
It is essential to obtain local contraction results with respect to the norm
\begin{equation}
{\big |}f{\big |}_{C^{0}\left((l-1,l), H^{m}\right) }:=\sum_{|\alpha|\leq m}\sup_{\tau\in [l-1,l]}{\Big |}D^{\alpha}_xf(\tau,.){\Big |}_{L^2\left({\mathbb R}^n \right) }.
\end{equation}
for some $m\geq 2$.
We can also get local contraction results with respect to the norm
\begin{equation}
\begin{array}{ll}
{\big |}f{\big |}_{C^{1}\left((l-1,l), H^{m}\right) }:=\sum_{|\alpha|\leq m}\sup_{\tau\in [l-1,l]}{\Big |}D^{\alpha}_xf(\tau,.){\Big |}_{L^2\left({\mathbb R}^n \right) }\\
\\
\sum_{|\alpha|\leq m}\sup_{\tau\in [l-1,l]}{\Big |}D^{\alpha}_xD_{\tau}f(\tau,.){\Big |}_{L^2\left({\mathbb R}^n \right) }
\end{array}
\end{equation}
for some $\geq 2$.
\begin{rem}
There are variations of norms which we may define where we can obtain similar local contraction results. We may define function spaces
\begin{equation}
\begin{array}{ll}
{\big |}f {\big |}_{H^{m}\times H^{2m}}:=\sum_{p=0}^m{\Big |}\frac{\partial^p}{\partial \tau^p}f{\Big |}_{L^2\left([l-1,l]\times {\mathbb R}^n \right) }\\
\\
+\sum_{|\alpha|\leq 2m}{\Big |}D^{\alpha}_xf{\Big |}_{L^2\left([l-1,l]\times {\mathbb R}^n \right) },
\end{array}
\end{equation}
or we may even consider related function spaces with the stronger norm  (with an abuse of language)
\begin{equation}
{\big |}f {\big |}_{H^{m}\times H^{2m}}:=\sum_{p=0}^m\sum_{|\alpha|\leq 2m}{\Big |}\frac{\partial^p}{\partial \tau^p}D^{\alpha}_xf{\Big |}_{L^2\left([l-1,l]\times {\mathbb R}^n \right) }
\end{equation}
such that mixed derivatives of spatial and time variables are contained.
\end{rem}

Our approximations have classical derivatives, such that the stronger norms
\begin{equation}
\begin{array}{ll}
{\big |}f{\big |}_{C^{m}\left((l-1,l),H^{2m}\right) }:=\sup_{\tau\in [l-1,l]}\sum_{p=0}^m{\Big |}\frac{\partial^p}{\partial\tau^p}f{\Big |}_{L^2\left( {\mathbb R}^n \right) }\\
\\
+\sup_{\tau\in [l-1,l]}\sum_{|\alpha|\leq 2m}{\Big |}D^{\alpha}_xf{\Big |}_{L^2\left( {\mathbb R}^n \right) },
\end{array}
\end{equation}
or (defining a variation with mixed derivatives)
\begin{equation}
{\big |}f{\big |}_{C^{m}\times H^{2m}}:=\sup_{\tau\in [l-1,l]}\sum_{p=0}^m\sum_{|\alpha|\leq 2m}{\Big |}\frac{\partial^p}{\partial t^p}D^{\alpha}_xf{\Big |}_{L^2\left(  {\mathbb R}^n \right) }
\end{equation}
are possible alternatives for our construction.
The functions $v^{*,\rho,l,k}_i$ are defined locally on $[l-1,l]\times {\mathbb R}^n$. For each time-step $l\geq 1$ we may identify this function with the trivial extension $v^{*,\rho,l,k}_i:{\mathbb R}\times {\mathbb R}^n\rightarrow {\mathbb R}$ (denoted by the same symbol $v^{*,\rho,l,k}_i$ of notation for the sake of simplicity) which is defined to be zero on the complementary domain $\left( {\mathbb R}\setminus [l-1,l]\right) \times {\mathbb R}^n$. We may measure these extensions in $L^2\left({\mathbb R}^{n+1}\right)$ naturally  or we may consider time and spatial variables separately. We consider local contraction results with respect to the $C^{m}\left( (l-1,l), H^{2m}\right) $-norms and with respect to $H^{2m}\left((l-1,l),{\mathbb R}^n\right) $ norms as well. The latter type of norms are appropriate to for $L^2$-estimate of convolutions and derivatives of convolutions with respect to time and space. In order to motivate the local contraction results let us have a first closer look at the relation to the global bound of the Leray projection term. Given the local solution $v^{*,\rho,l}_i,~1\leq i\leq n$, locally, the Leray projection term can be estimated pointwise by
\begin{equation}
\begin{array}{ll}
\rho_l\sum_{j,m=1}^n\int_{{\mathbb R}^n}\left( \frac{\partial}{\partial x_i}K_n(x-y)\right) \sum_{j,m=1}^n\left( \frac{\partial v^{*,\rho,l}_m}{\partial x_j}\frac{\partial v^{*,\rho,l}_j}{\partial x_m}\right) (\tau,y)dy\\
\\
\leq \rho_l\sum_{j,m=1}^n\int_{{\mathbb R}^n} {\Big |}\frac{\partial}{\partial x_i}K_n(x-y){\Big |}\times\\
\\
\times \left( \sum_{j,m=1}^n\frac{1}{2}\left( \frac{\partial v^{*,\rho,l}_m}{\partial x_j}\right)^2(\tau,y)+\frac{1}{2}\left( \frac{\partial v^{*,\rho,l}_j}{\partial x_m}\right)^2 (\tau,y)\right) dy.
\end{array}
\end{equation}
Hence looking at contraction results for squared $L^2$-norms is quite natural in order to obtain a global linear bound for the Leray projection term. We shall do so. Now assume that we have a linear bound for the squared initial data at time step $l\geq 1$, i.e., that we have a bound
\begin{equation}\label{initiall}
|v^{*,\rho,l-1}_i(l-1,.)|^2_{H^2}\leq C^{l-1}:=C+(l-1)C
\end{equation}
for some constant $C>0$ (note the square in (\ref{initiall}). This implies that we have
\begin{equation}\label{initiallsimple}
|v^{*,\rho,l-1}_i(l-1,.)|_{H^2}\sim \frac{1}{\sqrt{l-1}}
\end{equation}
for $l\geq 2$. For the controlled scheme we shall see that 
\begin{equation}
\sup_{\tau\in [l-1,l]}|v^{r,*,\rho,l}_i(\tau,.)-v^{r,*,\rho,l-1}_i(l-1,.)|_{H^2}\sim \frac{1}{\sqrt{l}},
\end{equation}
and we shall see that this implies that
\begin{equation}\label{initialsizel}
|v^{r,*,\rho,l}_i(l,.)|^2_{H^2}\leq C^{l}=C+lC
\end{equation}
and $r^l_i\sim l$ inductively. 
This implies in turn that
\begin{equation}\label{initialsizelor}
|v^{*,\rho,l}_i(l,.)|^2_{H^2}\leq C^{l}=C+lC
\end{equation}
We shall see that for a step size
\begin{equation}\label{stepsizel}
0< \rho_l\leq 2n^2(C^{l-1}+1)C_GC_{K}C_s
\end{equation}
we shall have a linear growth of the Leray projection term. Here, the constant $C^{l-1}$ on the right side of (\ref{stepsizel}) above is given by (\ref{initiall}). The constants on the right side of (\ref{stepsizel}) are a priori known, and are defined for $n=3$ by
\begin{equation}
\begin{array}{ll}
C_G=\max\left\lbrace {\big |}G_l{\big |}_{L^1\times L^1}, 1\right\rbrace \\
\\
C_K=\max\left\lbrace\int_{B_1(0)}{\big |}K_{n,i}(z){\big |}dz+\sqrt{\int_{{{\mathbb R}^3}\setminus B_1(0)}{\big |}K_{n,i}(z){\big |}^2dz},1\right\rbrace ,
\end{array}
\end{equation}
and $C_s$ is the constant for weighted products in $L^2$ in case $2=s>\frac{n}{2}=3/2$, where  in the  general case for $s>\frac{n}{2}$ we have a constant $C_s>0$  such that for all $x\in {\mathbb R}^n$ the function
\begin{equation}
u(x):=\int(1+|y|^2)^{-s/2}v(x-y)w(y)dy
\end{equation}
with functions $v,w\in L^2$ satisfies
\begin{equation}
|u|_{L^2}\leq C_s|v|_{L^2}|w|_{L^2}.
\end{equation}
Again without loss of generality we may assume that $C_s\geq 1$.
 
We have the following local contraction result.
\begin{thm}\label{mainthm}
Let $n=3$. Assume that $v^{*,\rho,l-1}_i(l-1,.)\in H^2$ for $1\leq i\leq n$ with $$|v^{*,\rho,l-1}_i(l-1,.)|^2_{H^2}\leq C^{l-1}.$$
Then we have local contraction results with respect to $C^0\times H^{2}$-norm. For related stronger inductive assumptions on $v^{*,\rho,l-1}_i(l-1,.)$ we have related local contraction results with respect to the alternative norms listed above. More precisely, there is an integer $c(n)=c(3)$ (depending only on dimension $n$ in generalized versions of the argument) such that for a time-step size \begin{equation}
\rho_l\leq \frac{1}{c(n)(C^{l-1}+1)C_GC_{K}C_s}
\end{equation}
(along with  $C_G,C_K$ and $C_s$ defined above) for $k\geq 1$ we have 
\begin{equation}\label{contract}
\begin{array}{ll}
\max_{i\in \left\lbrace 1,\cdots ,n\right\rbrace }|\delta v^{*,\rho,l,k}_i|_{C^0\left( (l-1,l), H^2\right)}
\leq \frac{1}{2}\max_{i\in \left\lbrace 1,\cdots ,n\right\rbrace }|\delta v^{*,\rho,l,k-1}_i|_{C^0\left( (l-1,l), H^2\right)}.
\end{array} 
\end{equation}
For the same time-step size we also have
\begin{equation}
\begin{array}{ll}
\max_{i\in \left\lbrace 1,\cdots ,n\right\rbrace }|\delta v^{*,\rho,l,k}_i|_{C^0\left( (l-1,l), H^2\right)}^2
\leq \frac{1}{2}\max_{i\in \left\lbrace 1,\cdots ,n\right\rbrace }|\delta v^{*,\rho,l,k-1}|_{C^0\left( (l-1,l), H^2\right)}^2.
\end{array} 
\end{equation}
For $k=1$ we may assume that $c(3)\geq 1$ is chosen such that
\begin{equation}
\begin{array}{ll}
\max_{i\in \left\lbrace 1,\cdots ,n\right\rbrace }|\delta v^{*,\rho,l,1}_i|_{C^0\left( (l-1,l), H^2\right)}\\
\\
=\max_{i\in \left\lbrace 1,\cdots ,n\right\rbrace }|v^{*,\rho,l,1}-v^{*,\rho,l-1}(l-1,.)|_{C^0\left( (l-1,l), H^2\right) }\leq \frac{1}{4}.
\end{array}
\end{equation}
and similar for the squared norm.
\end{thm}
We shall see that this local contraction result is indeed essential together with linear growth of the Leray projection term enforced by a control function in order to get a global scheme. Note that for the coefficient function evaluated at $t\geq 0$ (original time coordinates) the embedding $C^{\alpha}\subset H^2$ for dimension $n=3$ (uniformly with respect to time $t\geq 0$) ensures that given the solution, this same solution can be represented in terms of the fundamental solution of certain scalar parabolic equations which involve the solution in  the first order terms. However, in this paper we shall see that we can even improve this for the controlled scheme observing that a certain form of polynomial decay is inherited. This may not true for local schemes in general, but for all the local schemes considered in \cite{KB2} and here we may observe a similar phenomenon for the higher order correction terms
\begin{equation}
\delta v^{*,\rho,l,k}_i,~k\geq 2.
\end{equation}
We have a closer look at this phenomenon in the next section.
 We say that a function $f:{\mathbb R}^n\rightarrow {\mathbb R}$ with $f\in C^m$ (i.e., $f$ has continuous partial derivatives up to order $m$) has polynomial decay of order $m\geq 1$ up to the derivatives of order $m\geq 1$ if for all multivariate partial derivatives $D^{\alpha}_xf$ with $|\alpha|=\sum_{i=1}^n|\alpha_i|$ and $|\alpha|\leq m$ we have for all $x\in {\mathbb R}^n$
\begin{equation}
{\big |}D^{\alpha}_xf(x){\big |}\leq \frac{C_{\alpha}}{1+|x|^m}
\end{equation}
for some $C_{\alpha}<\infty$. We may then use techniques considered in \cite{KE2} in order to prove that polynomial decay of a certain order $m\geq 2$ is inherited by the higher order increments $\delta v^{*,\rho,l,k}_i$ of the local scheme. 
Next we mention that theorem \ref{mainthm} holds for stronger norms as well with slightly adapted step size. We mention that we assume $n=3$, because we use this assumption at several steps of the proof. However, the argument can be adapted to larger dimension $n>3$ and also to $n=1,2$, but for these extensions additional considerations have to be made. 
We have the following local contraction result.
\begin{thm}\label{mainthm2}
Let $n=3$. Assume that $v^{*,\rho,l-1}_i(l-1,.)\in H^m$ for $1\leq i\leq n$ with $$|v^{*,\rho,l-1}_i(l-1,.)|^2_{H^m}\leq C^{l-1}.$$
Then we have local contraction results with respect to $C^0\times H^{m}$-norm. For related stronger inductive assumptions on $v^{*,\rho,l-1}_i(l-1,.)$ we have related local contraction results with respect to the alternative norms listed above. More precisely, there is an integer $c(n)=c(3)$ (depending only on dimension $n$ in generalized versions of the argument) such that for a time-step size \begin{equation}
\rho_l\leq \frac{1}{c(n)(C^{l-1}+1)C_GC_{K}C_s}
\end{equation}
(along with  $C_G,C_K$ and $C_s$ defined above) for $k\geq 1$ we have 
\begin{equation}\label{contractm}
\begin{array}{ll}
\max_{i\in \left\lbrace 1,\cdots ,n\right\rbrace }|\delta v^{*,\rho,l,k}_i|_{C^0\left( (l-1,l), H^m\right)}
\leq \frac{1}{2}\max_{i\in \left\lbrace 1,\cdots ,n\right\rbrace }|\delta v^{*,\rho,l,k-1}_i|_{C^0\left( (l-1,l), H^m\right)}.
\end{array} 
\end{equation}
For the same time-step size we also have
\begin{equation}
\begin{array}{ll}
\max_{i\in \left\lbrace 1,\cdots ,n\right\rbrace }|\delta v^{*,\rho,l,k}_i|_{C^0\left( (l-1,l), H^m\right)}^2
\leq \frac{1}{2}\max_{i\in \left\lbrace 1,\cdots ,n\right\rbrace }|\delta v^{*,\rho,l,k-1}|_{C^0\left( (l-1,l), H^m\right)}^2.
\end{array} 
\end{equation}
For $k=1$ we may assume that $c(3)\geq 1$ is chosen such that
\begin{equation}
\begin{array}{ll}
\max_{i\in \left\lbrace 1,\cdots ,n\right\rbrace }|\delta v^{*,\rho,l,1}_i|_{C^0\left( (l-1,l), H^m\right)}\\
\\
=\max_{i\in \left\lbrace 1,\cdots ,n\right\rbrace }|v^{*,\rho,l,1}-v^{*,\rho,l-1}(l-1,.)|_{C^0\left( (l-1,l), H^m\right)}\leq \frac{1}{4}.
\end{array}
\end{equation}
and similar for the squared norm. Similar statements can be made for $C^p\left( (l-1,l), H^m\right)$-norms for $p\geq 1$. 
\end{thm}
Note that at each approximation step for $m=1$ we have indeed $v^{*,\rho,l,k}_i\in C\left([l-1,l], C^2\left({\mathbb R}^n\right)\cap H^2 \right)$. 

\section{Results of global regular upper bounds and global existence results for the uncontrolled Navier Stokes equation }

Next we state some implications concerning global regular upper bounds and global regular existence of the the uncontrolled incompressible Navier Stokes equation. From the point of view of controlled schemes developed in this paper the local time contraction results described in the preceding section are essential. They have not only time-local regular existence as a consequence, but together with controlled schemes they lead to global regular upper bounds and to global regular existence straightforwardly. However, let us first comment on some relations of schemes with external control functions to auto-controlled schemes considered in \cite{KAC} recently. Let us reconsider the idea of auto-controlled schemes in the context of the uncontrolled schemes $v^{\rho,l}_i$ and $v^{*,\rho,l}_i$ for $1\leq i\leq n$ and $l\geq 1$ introduced above. If uncontrolled regular data $v^{*,\rho,l-1}_i(l-1,.)=v^{\rho,l-1}_i(l-1,.)\in H^m\cap C^m$ for $m\geq 2$ are given, then the functions are local solution functions of the local incompressible Navier Stokes equation, and, according to the local contraction results of the preceding section, the functions $v^{\rho,l}_i$ and $v^{*,\rho,l}_i$ for $1\leq i\leq n$ are identical limits of local iteration schemes, i.e.,
\begin{equation}
\begin{array}{ll}
v^{*,\rho,l}_i=v^{*,\rho,l-1}_i(l-1,.)+\sum_{k\geq 1}\delta v^{*,\rho,l,k}_i\\
\\
=v^{\rho,l}_i=v^{\rho,l-1}_i(l-1,.)+\sum_{k\geq 1}\delta v^{\rho,l,k}_i,
\end{array}
\end{equation}
where we recall that $\delta v^{*,\rho,l,k}_i=v^{*,\rho,l,k}_i-v^{*,\rho,l,k-1}$ for $k\geq 1$ along with  $v^{*,\rho,l,0}_i:=v^{*,\rho,l-1}_i(l-1,.)=v^{\rho,l}_i$. As the scheme with upper script $*$ and the scheme without upper script $*$ (where a local contraction argument is described in \cite{KB2}) lead to the same limit function in $C^1\left(\left[l-1,l\right],H^m\cap C^m\right)$ for $m\geq 2$  we may drop the upper script $*$ when we argue about global schemes in terms of these local limit functions as we do in this section.  The idea of auto-controlled schemes considered in \cite{KAC} applied to these local schemes means that at each time step $l\geq 1$ we compare on the domain $[l-1,l)\times {\mathbb R^n}$ the functions $v^{\rho,l}_i(\tau,.)$ with the functions $u^{\rho,l}_i(s,.)$, where
\begin{equation}\label{timedil}
v^{\rho,l}_i=(1+(\tau-(l-1)))u^{\rho,l}_i\left(s,x\right),
\end{equation}
and where with $\tau_l=\tau-(l-1)$ we have
\begin{equation}
s=\frac{\tau_l}{\sqrt{1-\tau^2_l}}\in [0,\infty)
\end{equation}
This time dilatation transformation leads at each time step to an equation for the functions $u^{\rho,l}_i$ on the domain $[0,\infty)\times {\mathbb R^n}$ with damping potential terms (cf \cite{KAC}). Due to the damping terms it is not difficult to show that the functions preserve upper bounds in the sense that for each $m\geq 2$ we have a step size $\rho>0$ such that for all $l\geq 1$ and all $1\leq i\leq n$ we have
\begin{equation}
{\big |}u^{\rho,l-1}_i(l-1,.){\big |}_{H^m\cap C^m}\leq C\Rightarrow {\big |}u^{\rho,l}_i(l,.){\big |}_{H^m\cap C^m}\leq C.
\end{equation}
This is leads to upper bounds for $v^{\rho,l}_i$, or, equivalently, for $v^{*,\rho,l}_i$ immediately (at least of exponential type with respect to time with global regular upper bounds of the form $C2^l$), but closer analysis shows that global regular upper bounds which are linear in time or even independent of time can be obtained by such schemes. One possibility is to observe (\ref{timedil}) on a finer time scale. For example the alternative comparison with a refined scheme with respect to time, where
\begin{equation}\label{timedil}
v^{\rho,l}_i=\left( \frac{1}{2}+(\tau-(l-1))\right) w^{\rho,l}_i\left(s,x\right),
\end{equation}
is compared on $\tau_l=\tau-(l-1)\in \left[l-1,l-\frac{1}{2}\right]\times {\mathbb R}^n$ with
\begin{equation}
s=\frac{\tau_l}{\sqrt{1-\tau^2_l}}\in \left[ 0,\frac{1}{\sqrt{3}}\right],
\end{equation}
leads to the conclusion that we can global regular upper bounds for $w^{\rho,l}_i$ are preserved by the velocity components $v^{\rho,l}_i$. These observations (we shall make this fully explicit in an elaborated version of  \cite{KAC}) lead to several conclusions:
\begin{itemize}
 \item[i)] The local contraction results of section 4 are essential in the sense that global regular existence results can be obtained from them by an argument which relies only on a time dilatation transformation argument and the observation that similar contraction results can be obtained for $u^{\rho,l}_i$ as for $v^{\rho,l}_i$.
 \item[ii)] Global regular existence for the incompressible Navier Stokes equation can be obtained from global regular upper bounds in the function space $C^1\left(\left[0,T \right],H^m\cap C^m  \right)$ for $m\geq 2$, where the time dependence of the upper bounds is in general linear for simple schemes and the upper bounds are independent of time for (a bit) more sophisticated schemes. These results can be obtained by auto-controlled schemes and by schemes with an external control function.
 \item[iii)] If we consider auto-controlled schemes on a finer time scale of time step length $\leq \frac{1}{2}$, then auto-controlled schemes may be also interesting from a numerical point of view. The time consumption by time dilatation then limited, although we have no such dilatation for schemes with external control functions. 
 \item[iv)] Auto-controlled schemes have not the peculiarity of some of the schemes with external control functions that they are only Lipschitz with respect to time at integer time points $\tau=l$. Note, however, that for more sophisticed external control functions we can avoid this.  
\end{itemize}
Since the local contraction results are such essential that the argument for global regular existence of solutions of the incompressible Navier Stokes equation can be decoupled for a large extent, we may ask why we stick  to external control functions at all. There are several reasons on an analytical level:
\begin{itemize}
 \item[a)] The auto-controlled scheme is too weak in order to obtain global existence results for certain generalised models with degeneracies. Some of the results of this paper can be preserved for highly degenerate Navier Stokes equation where the uncoupled terms satisfy a H\"{o}rmander condition. 
 \item[b)] Auto-controlled schemes do not preserve spatial polynomial decay of a certain order. Well, it seems that this is not expected by most specialists. The reason seems to be that even in the simple classical Navier Stokes equation model scheme where we have representations of locally converging approximations in terms of convolutions with Gaussians, we loose some order of polynomial decay in general. However if we include the negative of the first approximation $-\delta v^{*,r^{l-1},\rho,l}_i$ as a summand into the control function at time step $l$, then we have some polynomial decay result for the controlled scheme, and the degree of polynomial decay which can be obtained for the uncontrolled Navier Stokes equation then depends on the degree of polynomial decay which we get for the control function (and in general we get some).
 \item[c)] We discovered global existence results first by using external control functions, and obtain a prove of global existence which is independent of the proof of global existence via auto-controlled schemes.
 \item[d)] The concept of an external control function is a flexible concept, and it can be combined with auto-control functions in order to obtain global existence results which cannot be obtained by auto-controlled schemes alone.
 \end{itemize}
We close this section stating the main theorems for the incompressible Navier Stokes equation. 

\begin{thm}\label{uncontrolledthm1} Assume that $\mathbf{h}=\left(h_1,\cdots,h_n \right)^T$ is a function with $h_i\in C^m\cap H^m$ for $m\geq 2$, and assume that $\nu>0$ for the viscosity constant $\nu$.
Then there are a global regular upper bounds for the velocity solution function components $v_i$ of the Cauchy problem
\begin{equation}\label{Navlerayorg}
\left\lbrace \begin{array}{ll}
\frac{\partial v_i}{\partial t}-\nu\sum_{j=1}^n \frac{\partial^2 v_i}{\partial x_j^2} 
+\sum_{j=1}^n v_j\frac{\partial v_i}{\partial x_j}=\\
\\ \hspace{1cm}\sum_{j,m=1}^n\int_{{\mathbb R}^n}\left( \frac{\partial}{\partial x_i}K_n(x-y)\right) \sum_{j,m=1}^n\left( \frac{\partial v_m}{\partial x_j}\frac{\partial v_j}{\partial x_m}\right) (t,y)dy,\\
\\
\mathbf{v}(0,.)=\mathbf{h}(0,.),
\end{array}\right.
\end{equation}
for arbitrary time horizon $T>0$ in the function space $C^1\left(\left[0,T\right],H^m\cap C^m  \right)$. More precisely, there exists a constant $C>0$ which depends only on the initial data and the viscosity constant $\nu>0$ such that the upper bound
\begin{equation}
\sup_{t\in [0,T]}{\big |}v_i(t,.){\big |}_{H^m\cap C^m}\leq C+CT
\end{equation}
and the stronger upper bound
\begin{equation}
\sup_{t\in [0,T]}{\big |}v_i(t,.){\big |}_{H^m\cap C^m}+\sup_{t\in [0,T]}{\big |}D_tv_i(t,.){\big |}_{H^m\cap C^m}\leq C+CT
\end{equation}
exists for the component $v_i$ of the solution $\mathbf{v}=(v_1,\cdots,v_n)^T$ in the function spaces $C^0\left(\left[0,T\right],H^m\cap C^m  \right)$ and $C^1\left(\left[0,T\right],H^m\cap C^m  \right)$ respectively. As in implication we have global smooth existence for the Cauchy problem for $\nu >0$ and Schartz data, which is one version of the problem formulated in the notes of C. Fefferman in \cite{Feff}.
\end{thm}
This theorem can be proved with the simple control function in (\ref{simpleincrement1}). Note that this theorem follows from a quite straightforward argument using this simple control scheme and the time-local contraction results stated in the preceding section. As we remarked the time dependence of the upper bound can be improved to be logarithmic if we use a time step size $\rho_l\sim\frac{1}{l}$. For the alternative simple scheme which uses the negative first increment $-\delta v^{r^{l-1},*,\rho,l,1}_i$ of the local scheme for the uncontrolled Navier Stokes equation with controlled data, we need the form of decreasing time step size. This simple scheme leads to a linear upper bound of the Leray projection term. We have
\begin{thm}\label{uncontrolledthm2}
Assume that $\mathbf{h}=\left(h_1,\cdots,h_n \right)^T$ is a function with $h_i\in C^m\cap H^m$ for $m\geq 2$, and assume that $\nu>0$ for the viscosity constant $\nu$. Then for the velocity solution function $v_i$ of (\ref{Navlerayorg}) we have
\begin{equation}
 \begin{array}{ll}
{\Big |}\sum_{j,m=1}^n\int_{{\mathbb R}^n}\left( \frac{\partial}{\partial x_i}K_n(x-y)\right) \sum_{j,m=1}^n\left( \frac{\partial v_m}{\partial x_j}\frac{\partial v_j}{\partial x_m}\right) (t,y)dy{\Big |}\leq C+CT
\end{array}
\end{equation}
for some constant $C>0$ which depends only on the initial data $\mathbf{h}$ and on the viscosity constant $\nu >0$.
\end{thm}
We note that the result of (\ref{uncontrolledthm2}) can also be obtained with the methods of the simple controlled scheme which we use in order to get (\ref{uncntrolledthm1})  if we use the variation of the scheme with decreasing time step size which leads to a regular upper bound with logarithmic growth in time. Well, both theorems may be sharpened using the more sophiscat ed controlled scheme (as we do in this paper), or an auto-controlled scheme (as we do in an elaboration of \cite{KAC}. We have
\begin{thm}\label{uncontrolledthm3} Assume that $\mathbf{h}=\left(h_1,\cdots,h_n \right)^T$ is a function with $h_i\in C^m\cap H^m$ for $m\geq 2$, and assume that $\nu>0$ for the viscosity constant $\nu$.
Then there are a global regular upper bounds for the velocity solution function components $v_i$ of the Cauchy problem in (\ref{Navlerayorg})
for arbitrary time horizon $T>0$ in the function space $C^1\left(\left[0,T\right],H^m\cap C^m  \right)$ of the form
\begin{equation}
\sup_{t\in [0,T]}{\big |}v_i(t,.){\big |}_{H^m\cap C^m}\leq C
\end{equation}
and even of the form
\begin{equation}
\sup_{t\in [0,T]}{\big |}v_i(t,.){\big |}_{H^m\cap C^m}+\sup_{t\in [0,T]}{\big |}D_tv_i(t,.){\big |}_{H^m\cap C^m}\leq C
\end{equation}
for some constant $C>0$ which is independent of the time.
\end{thm}
In addition to the statements of this section there may be additional consequences concerning polynomial decay of the solution. In the next section we argue that (at least for some type of control  function) polynomial decay order is preserved for the controlled scheme. This in implies that the degree of polynomial decay for the uncontrolled velocity function components is bounded by the polynomial decay of the control function and the polynomial decay of the control function. For the controlled scheme the persistence of polynomial decay is quite natural if the increment of control function has the negative of the first order approximation increment as a summand at each time step. The convolution with the Gaussian then leads to a loss of order of the polynomial decay of the control function. Note, however, that at each time step $l$ the control function increments of the simple schemes are determined in terms of the data of the controlled velocity function increments evaluated at time $l-1$, and for this controlled velocity function we have persistence of polynomial decay - at least for the controlled schemes with the negative first order velocity increment. So some degree of polynomial decay may be preserved, but we have not estimated this exactly so far.

\section{Inheritance of polynomial decay for the higher order correction terms in the local scheme  } 
Although the emphasis of this paper is on local contraction in the function spaces $C^0\left([0,T],H\cap C^m \right)$ and $C^0\left([0,T],H\cap C^m \right)$, and the global upper bounds we get for the velocity function components in these spaces, we consider the question of polynomial decay for controlled schemes in more detail in order to give some more substance to the last remarks of the preceding section. This section is not essential for the local contraction results and the global existence results stated in the previous sections and may be skipped by the reader who is interested in that kind of results in the first place.
The considerations at the end of section 3 show us that it is essential to observe some properties for the local scheme without control function. The transfer to a controlled scheme is straightforward. Therefore we consider the inheritance of polynomial decay for the local uncontrolled scheme.   
Given $v^{*,\rho,l-1}_i(l-1,.)$ at time step $l-1$ the solution function $v^{*,\rho,l}_i,1\leq i\leq n$ is constructed via the functional series
\begin{equation}
v^{*,\rho,l}_i=v^{*,\rho,l-1}_i(l-1,.)+\delta v^{*,\rho,l,1}_i+\sum_{k\geq 2}\delta v^{*,\rho,l,k}_i.
\end{equation}
We call the terms of the last sum, i.e., the terms $v^{*,\rho,l,k}_i,k\geq 2$ the higher order correction terms. If we consider standard estimates of Gaussians around the origin, then we may expect that some order of polynomially decay may be lost by $\delta v^{*,\rho,l,1}_i$ while representations of the higher order terms contain convolutions involving products of functions and may preserve polynomial decay of a certain order. In this section we prove this. However, an even stronger result is possible: although we do not need this, the following simple estimates are very remarkable, and they have consequences for the preservation of polynomial decay of the local scheme considered here. This can simplify some steps of the local contraction proof, and it makes precise some remarks of \cite{KNS}. Therefore, let us consider this in more detail. 
Let us assume that the function $v^{*,\rho,l-1}_i(l-1,.),~1\leq i\leq n$ have been determined at the previous time step with
\begin{equation}\label{inductivehyplminus1preface}
\max_{1\leq i\leq n}\sum_{|\alpha|\leq 2}\sup_{y\in {\mathbb R}^3}\left(1+|y|^2\right) {\Big |}D^x_{\alpha}v^{*,\rho,l-1}_i(l-1,y){\Big |}\leq C
\end{equation}
for some constant $C>0$. In the local scheme we then construct the solution for the first approximation for small time-step size as a perturbation of $v^{*,\rho,l,1}_i=v^{*,\rho,l-1}_i(l-1,.)+\delta v^{*,\rho,l,1}_i$ which has the representation (for $\tau\in [l-1,l]$)
\begin{equation}\label{solrep1111}
 \begin{array}{ll}
 v^{*,\rho,l,1}_i(\tau,x)=\int_{{\mathbb R}^3}v^{*,\rho,l-1}_i(l-1,y)G_l(\tau-(l-1),x-y)dy\\
 \\
 +\int_{l-1}^{\tau}\int_{{\mathbb R}^n}S^1(s,y)G_l(\tau-s,x-y)dyds,
 \end{array}
 \end{equation}
where
\begin{equation}
\begin{array}{ll}
S^1(\tau,y):=-\rho_l\sum_{j=1}^n v^{*,\rho,l-1}_j(\tau,y)\frac{\partial v^{*,\rho,l-1}_i}{\partial x_j}(l-1,y)+\\
\\
\rho_l\sum_{j,m=1}^n\int_{{\mathbb R}^n}\left( \frac{\partial}{\partial x_i}K_n(y-z)\right) \sum_{j,m=1}^n\left( \frac{\partial v^{*,\rho,l-1}_m}{\partial x_j}\frac{\partial v^{*,\rho,l-1}_j}{\partial x_m}\right) (l-1,z)dz.
\end{array}
\end{equation}
Now let us consider the first term on the right side of (\ref{solrep1111}), where the corresponding term of the solution has the form
\begin{equation}\label{solrep1111b}
 \begin{array}{ll}
\int_{{\mathbb R}^3}v^{*,\rho,l-1}_i(l-1,y)\frac{1}{\left( 2\sqrt{\pi\rho \nu(\tau-(l-1))}\right) ^n}\exp\left(-\frac{(x-y)^2}{4\rho_l\nu(\tau-(l-1))} \right) dy.
 \end{array}
 \end{equation}
For $\tau=l-1$ we have (\ref{inductivehyplminus1preface}) by assumption, so let us assume that $\tau>l-1$. Let 
\begin{equation}
B_{\frac{|x|}{2}}:=\left\lbrace |z|\leq \frac{|x|}{2}\right\rbrace 
\end{equation}
and consider (\ref{solrep1111}) as a sum of 
\begin{equation}\label{solrep1111a}
 \begin{array}{ll}
\int_{B_{\frac{|x|}{2}}}v^{*,\rho,l-1}_i(l-1,y)\frac{1}{\left( 2\sqrt{\pi\rho \nu(\tau-(l-1))}\right) ^n}\exp\left(-\frac{(x-y)^2}{4\rho_l\nu(\tau-(l-1))} \right) dy,
 \end{array}
 \end{equation}
and 
\begin{equation}\label{solrep1111b}
\begin{array}{ll}
\int_{{\mathbb R}^n\setminus B_{\frac{|x|}{2}}}v^{*,\rho,l-1}_i(l-1,y)\frac{1}{\left( 2\sqrt{\pi\rho \nu(\tau-(l-1))}\right) ^n}\exp\left(-\frac{(x-y)^2}{4\rho_l\nu(\tau-(l-1))} \right) dy.
\end{array}
\end{equation}
For the first part and $x\neq 0$ consider the estimate similar as in
(\ref{simpleest}) below, i.e., the estimate
\begin{equation}\label{simpleest3}
  \begin{array}{ll}
\sup_{|y|\leq \frac{|x|}{2}}{\big |}G_{l}(t-s,x-y){\big |}
\leq {\big |}C(t-(l-1))^{m-n/2}\frac{2^m}{|x|^{2m}} {\big |},
\end{array}
\end{equation}
which leads to polynomial decay of any order $2m>0$ for the part (\ref{solrep1111a}) in $x$ with the assumption (\ref{inductivehyplminus1preface}), and for the second part (\ref{solrep1111b}) we have the upper bound
\begin{equation}\label{solrep1111b}
\begin{array}{ll}
\int_{{\mathbb R}^n\setminus B_{\frac{|x|}{2}}}\frac{C}{1+|y|^2}\frac{1}{2\sqrt{\pi\rho_l \nu(\tau-(l-1))}^n}\exp\left(-\frac{(x-y)^2}{4\rho_l\nu(\tau-(l-1))} \right) dy\\
\\
\leq \int_{{\mathbb R}^n\setminus B_{\frac{|x|}{2}}}\frac{4C}{4+|x|^2}\frac{1}{2\sqrt{\pi\rho_l \nu(\tau-(l-1))}^n}\exp\left(-\frac{(x-y)^2}{4\rho_l\nu(\tau-(l-1))} \right) dy\\
\\
\leq \frac{\tilde{C}}{1+|x|^2},
\end{array}
\end{equation}
where the  constant $\tilde{C}>0$ which is independent of $\tau-(l-1)>0$ is proportional to the $L^1$-norm of the heatkernel for $\tau>(l-1)$. For the second term in (\ref{solrep1111}) we observe that these terms contain products of functions $v^{*,\rho,l-1}_i(l-1,.)$ and derivatives ( which we assumed to be of polynomial decay of order 2 (\ref{inductivehyplminus1preface})) This implies that we have expressions with polynomial decay of order $4$ in convolutions with the Gaussian $G_l$ and in double convolutions with the Laplacian kernel and the Gaussian. This time we have an integral with respect to time and have to consider local limits $\tau\downarrow (l-1)$. We shall see below that the decay inherited from this term may be smaller than $4$ but is still of quadratic order. It is a remarkable fact that we do not need these considerations if we consider control functions. This leads us to the expectation (far beyond the scope of this paper) that the scheme of control functions is of use beyond the scope of Gaussian estimates.

This is another motivation of our construction of control functions $r^l_i$ at each time step defining an equation for $v^{r,*,\rho,l,k}_i=v^{*,\rho,l,k}_i+r^l_i$ along with
$v^{r,*,\rho,l-1}_i(l-1,.)=v^{*,\rho,l-1}_i(l-1,.)+r^{l-1}_i(l-1,.)$  below. We shall define $r^l_i$ dynamically such that
\begin{equation}
\begin{array}{ll}
v^{r,*,\rho,l}_i=v^{r,*,\rho,l-1}_i(l-1,.)+\delta v^{*,\rho,l,1}_i+\delta r^l_i+\sum_{k\geq 2}\delta v^{*,\rho,l,k}_i\\
\\
=v^{r,*,\rho,l}_i=v^{r,*,\rho,l-1}_i(l-1,.)+\sum_{k\geq 2}\delta v^{*,\rho,l,k}_i.
\end{array}
\end{equation}
Hence, for the controlled scheme we have preservation of polynomial decay if we have preservation of polynomial decay for the higher order correction terms. 
The problem then is whether polynomial decay of a certain order of terms $\delta v^{*,\rho,l,k}_i,~k\geq 2$ and $v^{*,\rho,l,k}_i,~k\geq 2$ is inherited from $\delta v^{*,\rho,l,1}_i$ involve $v^{*,\rho,l,1}_i$. However, the representations of the higher order correction terms $\delta v^{*,\rho,l,k}_i,~k\geq 2$ and of higher order approximations $v^{*,\rho,l,k}_i,~k\geq 2$ involve convolutions with products of lower order terms which are known to have polynomial decay of a certain order $m\geq 2$, such that the products have polynomial decay of order $2m$. Some polynomial decay may be lost in the estimates of the double convolution with a partial derivative of a Laplacian kernel and with the Gaussian and a partial derivative with a Gaussian, but this loss does not depend on the order of polynomial decay of the mentioned lower order terms.

Next we have a closer look at these phenomena. Consider the Cauchy problem in (\ref{Navleray*}). The solution of this linear problem can be represented in the form
\begin{equation}\label{Navleray*solrep}
 \begin{array}{ll}
v^{*,\rho,l,k}_i(\tau,x)=\int_{{\mathbb R}^n}v^{*,\rho,l-1}_i(l-1,y)G_l(\tau,x-y)dy+\\
\\
-\int_{l-1}^{\tau}\int_{{\mathbb R}^n}\rho_l\sum_{j=1}^n v^{*,\rho,l,k-1}_j\frac{\partial v^{*,\rho,l,k-1}_i}{\partial x_j}(s,y)G_l(\tau-s,x-y)dyds\\
\\
+\int_{l-1}^{\tau}\int_{{\mathbb R}^n}\rho_l\sum_{j,m=1}^n\int_{{\mathbb R}^n}\left( \frac{\partial}{\partial x_i}K_n(z-y)\right) \times \\
\\
\times \sum_{j,m=1}^n\left( \frac{\partial v^{*,\rho,l,k-1}_m}{\partial x_j}\frac{\partial v^{*,\rho,l,k-1}_j}{\partial x_m}\right) (\tau,y)dz G_l(\tau-s,x-y)dyds.
\end{array}
\end{equation}
We say that
\begin{equation}
\begin{array}{ll}
v^{*,\rho,l,k}_i \mbox{ is of polynomial decay of order $m\geq 2$}\\
\\
\mbox{ for derivatives up to order $p\geq 0$ if for some finite  $C>0$}\\
\\
\sum_{|\alpha|\leq p}\sup_{\tau\in [l-1,l]}|D^{\alpha}_x v^{*,\rho,l,k}_i(\tau,y)|\leq \frac{C}{1+|y|^m}.
\end{array}
\end{equation}
Similarly for the functional increments $\delta v^{*,\rho,l,k}_i$.
The spaces of functions of polynomial decay of order $m\geq 2$ form an algebra. Especially, if $v^{*,\rho,l,k-1}_i,v^{*,\rho,l,k}_i$ are of polynomial decay of order $m\geq 2$ for derivatives up to order $p\geq 0$, then we have that functional increments $\delta v^{*,\rho,l,k}_i$ are of polynomial decay of order $m\geq 2$ and for derivatives up to order $p\geq 0$ . Moreover products of such functions have polynomial decay of order $2m$ for derivatives up to order $p$. As we pointed out, the fundamental solution evaluated at time $\tau-s=1$, i.e., the function
\begin{equation}
G_l(1,z)=\frac{1}{\left( 2\sqrt{\pi\rho \nu}\right) ^n}\exp\left(-\frac{z^2}{4\rho_l\nu} \right) 
\end{equation}
is clearly of polynomial decay of any order. However if $s\uparrow\tau$ we have to take care of weakly singular behavior, i.e.,
\begin{equation}
G_l(\tau-s,x-y)=\frac{1}{\left( 2\sqrt{\pi\rho \nu(\tau-s)}\right) ^n}\exp\left(-\frac{(x-y)^2}{4\rho_l\nu(\tau-s)} \right) 
\end{equation} 
need to be consired carefully as $s$ is close to $\tau$ and $x$ is close to $y$.
In this situation it is helpful that the terms in (\ref{Navleray*solrep}) involve products of polynomial decay of order $m\geq 2$ inductively and may preserve polynomial decay of the same order. 
Such terms appear in the representation of all increments $\delta v^{*,\rho,l,k}$ for $k\geq 1$. The representations of these functional increments consist of convolutions involving products of value functions with the fundamental solution $G_l$. For a partition of unity $\phi_B, 1-\phi_B$ with an appropriate function $\phi_B\in C^{\infty}$ supported in a ball $B$ of radius $1$ around the origin we can use
\begin{equation}
{\big |}\phi_BG_l(\tau-s,x-y){\big |}\leq \frac{C}{(\tau-s)^{\mu}|x-y|^{n-2\alpha}}
\end{equation}
for $\alpha\in (0,1)$, and then use convolution estimates. For the complementary function $(1-\phi_l)G_l$ we may use rough estimates.
These considerations motivate the following definition.
\begin{defi}
Assume that for all $1\leq i\leq n$ the functions $v^{*,\rho,l,1}_i,\delta v^{*,\rho,l,1}_i$ have polynomial decay of order $m\geq 2$ for derivatives up to order $p$. We say that polynomial decay of order $m$ for derivatives up to order $p\geq 0$ is inherited for the higher order correction terms $\delta v^{*,\rho,l,k}_i,~k\geq 2$, if for all $1\leq i\leq n$ these higher order terms have polynomial decay of order $m\geq 2$ for derivatives up to order $p$.
\end{defi}
 We have
\begin{lem}
Polynomial decay of order $m\geq 2$ for derivatives up to order $p\geq 0$ is inherited by the higher order correction terms.
\end{lem}

\begin{proof}
We consider the case $m=2$ and $p=2$. Higher order polynomial decay up to higher order derivatives can be proved similarly.
We consider the representation of the functional increment in (\ref{NavlerayII*}). 

We may use partitions of unity $\phi_B,1-\phi_B$ and estimates for the truncated Gaussian as considered in (\ref{simpleest}) below, i.e., the estimate
\begin{equation}\label{simpleest2}
  \begin{array}{ll}
{\big |}\phi_B(x-y)G_{l}(t-s,x-y){\big |}
\leq {\big |}\phi_1(x-y)C(t-s)^{m-n/2}\frac{1}{(x-y)^{2m}} {\big |},
\end{array}
\end{equation}
for $m=1$ and some constant $C>0$. Here we choose $m=1$ since we want to have local spatial integrability. Note that with this choice we have integrability with respect to time, too. We have
\begin{equation}
\sum_{|\alpha|\leq 2}\max_{1\leq j\leq n}\sup_{s\in [l-1,l]}|D^{\alpha}_x v^{*,\rho,l,1}_j(s,y)|\leq \frac{C}{1+|y|^2},
\end{equation}
\begin{equation}
\sum_{|\alpha|\leq 2}\max_{1\leq j\leq n}\sup_{s\in [l-1,l]}|D^{\alpha}_x\delta v^{*,\rho,l,1}_j(s,y)|\leq \frac{C}{1+|y|^2},
\end{equation}
and inductively for $k-1\geq 1$ we assume
\begin{equation}
\sum_{|\alpha|\leq 2}\max_{1\leq j\leq n}\sup_{s\in [l-1,l]}|D^{\alpha}_x\delta v^{*,\rho,l,k-1}_j(s,y)|\leq \frac{C}{1+|y|^2}.
\end{equation}
We consider the right side of (\ref{NavlerayII*}) term by term. For $n=3$ we have
\begin{equation}
\begin{array}{ll}
{\Big |}-\int_{l-1}^{l}\int_{B}\rho_l\sum_{j=1}^n \delta v^{*,\rho,l,k-1}_j\frac{\partial v^{*,\rho,l,k-1}_i}{\partial x_j}(s,y)\times\\
\\
\times \phi_B(x-y)G_l(\tau-s,x-y)dyds{\Big |}\\
\\
\leq \int_{l-1}^{l}\int_{B}\frac{C}{1+|y|^2}\frac{C}{1+|y|^2}{\big |}\phi_1(x-y)C(t-s)^{2-3/2}\frac{1}{(x-y)^{2}} {\big |}dy\\
\\
\leq \frac{\tilde{C}}{1+|x|^{4+2-3}}
\leq \frac{\tilde{C}}{1+|x|^{2}}
\end{array}
\end{equation}
 for some constant $\tilde{C}$ (generous in the last step). The complement of the truncated Gaussian, i.e., the function $(\tau,z)\rightarrow (1-\phi_B(z))G_B(\tau-s,z)$ behaves quite smoothly as $s\uparrow \tau$. Indeed, for the third term we have for all $p\geq 1$
\begin{equation}
\begin{array}{ll}
\sup_{\tau\neq s,\tau-s\in [l-1,l]}|(1-\phi_B(z))G_B(\tau-s,z)|\leq  \frac{C}{1+z^{2p}}
\end{array}
\end{equation}
for some $C>0$, hence, choosing $p=m+3$, we have for $n=3$ the rough estimate
\begin{equation}
\begin{array}{ll}
{\Big |}-\int_{l-1}^{l}\int_{B}\rho_l\sum_{j=1}^n \delta v^{*,\rho,l,k-1}_j\frac{\partial v^{*,\rho,l,k-1}_i}{\partial x_j}(s,y)\times\\
\\
\times(1-\phi_B(x-y))G_l(\tau-s,x-y)dyds{\Big |}
\leq \frac{\tilde{C}}{1+|x|^{2}}.
\end{array}
\end{equation}
There other type of terms remaining involve a double convolution. 
For the convolution with the Laplacian kernel we note the the first order derivative of the Laplacian kernel satisfies
\begin{equation}
K_{,i}(y)\sim \frac{y_i}{|y|^n},
\end{equation}
which is a decay of second order effectively. We have again the inductive assumption and convolution with the Laplacian kernel gives a decay of order $2p+2-n$, i.e., a decay of order $3$ in case $p=2$ and $n=3$. Here we note that we have polynomial decay of order $p$ for the functions $v^{*,\rho,l,k}_i$ and their derivatives up to secons order. Hence, we may use a partition of unity $\phi_B, (1-\phi_B)$ and estimate terms of the sum
\begin{equation}
K_{,i}(y)=\phi_B(y)K_{,i}(y)+(1-\phi_B(y))K_{,i}(y),
\end{equation}
and then shift derivatives in the convolutions as we need.
We have
\begin{equation}
\begin{array}{ll}
{\Big |}\int_{l-1}^{\tau}\int_{{\mathbb R}^n}\rho_l\sum_{j,m=1}^n\int_{{\mathbb R}^n}\left( \frac{\partial}{\partial x_i}K_n(z-y)\right)\times \\
\\
\times\left( \sum_{m,j=1}^n \frac{\partial \delta v^{*,\rho,l,k-1}_j}{\partial x_m}\left(\frac{\partial v^{*,\rho,l,k-1}_m}{\partial x_j}+\frac{\partial v^{*,\rho,l,k-2}_m}{\partial x_j}\right)\right)  (s,y)\times\\
\\
\times \phi_B(x-y)G_l(\tau-s,x-y)dyds{\Big |}\\
\\
\leq {\Big |}\int_{{\mathbb R}^n}\int_{{\mathbb R}^n}\frac{C}{(z-y)^2}\frac{C}{1+|y|^2}\frac{C}{1+|y|^2}dy\frac{C}{|x-z|^{2}}dz{\Big |}
\leq {\Big |}\frac{\tilde{C}}{1+|x|^{2}}{\Big |}
\end{array}
\end{equation}
for $p\geq 2$ and some constant $\tilde{C}>0$.
For the last term we have a convolution with $(1-\phi_B(x-y))G_l(\tau-s,x-y)$ which is of strong polynomial decay without singularities, and it is clear that in this case we can again do rough estimates. For the first and second order derivatives of the value functions $v^{*,\rho,l,k}_i$ we have a similar situation. We shall see that we have representations involving only derivatives up to first order of the Gaussian where we have a priori estimates of the form
\begin{equation}
{\Big |}\phi_B(x-y)G^l_{,i}(\tau-s,x-y{\Big |}\leq \frac{C}{(\tau-s)^{\alpha}(x-y)^{n+1-2\alpha}}
\end{equation}
for $\alpha \in (0.5,1)$. The proof may then proceed similarly as above, and we come back to this situation below.
\end{proof}

\section{Global regularity and growth behavior of the control function for the simple scheme}
The following considerations are not needed for the construction of global upper bounds in $C^1\left([0,T],H^m\cap C^m\right)$, but are inserted here for additional information. The reader interested in the proof of theorem \ref{uncontrolledthm1} and theorem  \ref{uncontrolledthm2} may skip this and the next section.
Consider first a simple control function with the increment
\begin{equation}\label{controlincrementobservation}
\delta r^l_i(l,x)=\int_{l-1}^l\int_{{\mathbb R}^ n}\frac{-v^{r,*,\rho,l-1}_i(l-1,y)}{C}G_l(l,x;s,y)dyds.
\end{equation}
A natural question which may be posed is: why does this definition not reduce the polynomial decay of the control functions, and therefore, with a delay of one time step the polynomial decay of the controlled velocity functions as well ? The reason is that the regularity of the integrand can be used here. If the controlled velocity function $v^{r,*,\rho,l-1}_i(l-1,.)$ is smooth, i.e., if $v^{r,*,\rho,l-1}_i(l-1,.)\in C^{\infty}$, then this holds for the increment (\ref{controlincrementobservation}) as well. In the classical model where we have constant viscosity this follows easily from the fact that 
(\ref{controlincrementobservation}) is a convolution. For an arbitrary multivariate derivative $D^{\alpha}_x$ of order $|\alpha|\geq 0$ we may use the smoothness of the controlled velocity function and write
\begin{equation}\label{convobservation}
\begin{array}{ll}
D^{\alpha}_x\delta r^l_i(l,x)=\int_{l-1}^l\int_{{\mathbb R}^ n}\frac{-v^{r,*,\rho,l-1}_i(l-1,y)}{C}D^{\alpha}_xG_l(l,x;s,y)dyds\\
\\
=\int_{l-1}^l\int_{{\mathbb R}^ n}D^{\alpha}_y\frac{-v^{r,*,\rho,l-1}_i(l-1,y)}{C}G_l(l,x;s,y)dyds,
\end{array}
\end{equation}
which shows that the controlled function increment $\delta r^l_i(l,.)$ is smooth if the controlled velocity function $v^{r,*,\rho,l}_i(l-1,.)$ of the previous time step $l-1$ is smooth. Moreover, if the control function $r^{l-1}_i(l-1,.)$ is smooth it follows that the control function $r^l_i(l,.)$, the controlled velocity function $v^{r,*,\rho,l}_i$ and the original velocity function $v^{*,\rho,l}_i(l,.)$ inherit smoothness from the previous time step. This phenomenon has some consequences for the growth behavior of the simple control function considered so far. We have already observed that convolutions of the type considered in (\ref{convobservation}) may be estimated by breaking up the integral into a local part around $x$ and a complementary part. For a ball $B_{\epsilon}(x)$ of radius $\epsilon>0$ around $x$  we may write for $0<\mu<1$
\begin{equation}\label{convobservation2}
\begin{array}{ll}
{\Big |}D^{\alpha}_x\delta r^l_i(l,x){\Big |}\leq\\
\\
{\Big |}\int_{l-1}^l\int_{B_{\epsilon}(x)}D^{\alpha}_y\frac{-v^{r,*,\rho,l-1}_i(l-1,y)}{C}\frac{C}{(l-s)^{\mu}|x-y|^{n-2\mu}}dyds{\Big |}\\
\\
+{\Big |}\int_{l-1}^l\int_{{\mathbb R}^ n\setminus B_{\epsilon}(x)}D^{\alpha}_y\frac{-v^{r,*,\rho,l-1}_i(l-1,y)}{C}G_l(l,x;s,y)dyds{\Big |}.
\end{array}
\end{equation}
The second integral on the right side of (\ref{convobservation2}) is over a domain where $G_l$ is analytic. Moreover, polynomial decay any order $p>0$ of this term is inherited from $D^{\alpha}_y\frac{-v^{r,*,\rho,l-1}_i(l-1,.)}{C}$ as we may use the convolution rule to write the second integral on the right side (\ref{convobservation2}) for $|x|\neq 0$ in the form
\begin{equation}\label{convobservation2}
\begin{array}{ll}
{\Big |}\int_{l-1}^l\int_{{\mathbb R}^ n\setminus B_{\epsilon}(x)}D^{\alpha}_y\frac{-v^{r,*,\rho,l-1}_i(l-1,x-y)}{C}G_l(l,y;s,0)dyds{\Big |}\\
\\
\leq {\Big |} \int_{l-1}^l\int_{{\mathbb R}^ n\setminus B_{\epsilon}(x)}\frac{c}{1+|x-y|^p}G_l(l,y;s,0)dyds{\Big |}\\
\\
\leq {\Big |} \int_{l-1}^l\int_{{\mathbb R}^ n\setminus B_{\epsilon}(x)}\frac{c}{1+|x-y|^p}\frac{\tilde{c}}{1+|y|^n}dyds{\Big |}\leq\frac{\tilde{C}}{|x|^p}
\end{array}
\end{equation}
The first integral on the right side of (\ref{convobservation2}) may be estimated in polar coordinates $(r,\phi_1,\cdots,\phi_n)$ using iterated partial integration. For $\mu\in \left(\frac{1}{2},1\right)$ we need only $n-1$ partial integrations to get the upper bound   
\begin{equation}
\begin{array}{ll}
{\Big |}\int_{l-1}^l\int_{B_{\epsilon}(x)}D^{\alpha}_y\frac{-v^{r,*,\rho,l-1}_i(l-1,x-y)}{C}\frac{C}{(l-s)^{\mu}|y|^{n-2\mu}}dyds{\Big |}\\
\\
\leq (n-1)C_{\epsilon}+\\
\\
{\Big |}\int_{l-1}^l\int_{B_{\epsilon}(x)}D^{n-1}_r\left( D^{\alpha}_y\frac{-v^{r,*,\rho,l-1}_i(l-1,x-y)}{C}\right)_{\mbox{pol}} \frac{C}{(l-s)^{\mu}}|y|^{2\mu-1}dyds{\Big |},
\end{array}
\end{equation}
where $C_{\epsilon}$ is an upper bound for boundary terms which are integrals over the sphere $S^n_{\epsilon}(x)$, i.e., the boundary of $B_{\epsilon}(x)$. Here we understand that
\begin{equation}
\left( D^{\alpha}_y\frac{-v^{r,*,\rho,l-1}_i(l-1,x-y)}{C}\right)_{\mbox{pol}}
\end{equation}
is the $D^{\alpha}_y\frac{-v^{r,*,\rho,l-1}_i(l-1,.)}{C}$ written in polar coordinates. Summing up these conservations we get for $\mu>\frac{1}{2}$ an upper bound $\sim 1$ for the control function increment as desired. Note that the surface terms mentioned become small for small $\epsilon>0$ for this choice of $\mu>\frac{1}{2}$.

Next concerning extended control functions we come to a similar conclusion for the additional term
\begin{equation}
-\delta v^{r^{l-1},*,\rho,l-1}_i(\tau,x),
\end{equation}
for reasons discussed  in the previous section, and the additional source term can be estimated as above.  

\section{Global regularity and growth behavior of the control function for the scheme with small foresight}
This section provides additional information in the context of the stronger result theorem \ref{uncontrolledthm3}.
The global linear upper bound is sufficient in order to prove global regular existence for the incompressible Navier Stokes equation. However, especially the scheme with small foresight introduced above allows us to sharpen the result concerning upper bounds a bit. This can be of relevance also for the investigation of long time behavior and its relation to possible singular behavior of degenerate equations. It seems that the H\"{o}rmander condition for the linear diffusion terms of a (generalized) Navier Stokes equation is the natural no-singularity-criterion for Navier Stokes equations. As we shall argue more explicitly elsewhere full degeneracy to a local Euler equation even locally seems to be linked to vorticity singularities. This motivates the considerations of uniform global upper bounds, where we reconsider the scheme with slight foresight introduced above. As the scheme with small foresight is defined in terms of local solutions of the Navier Stokes equation the definition of the control function within the scheme is essentially the same for the scheme $v^{r,\rho,l}_i$ and for the simplified scheme $v^{r,*,\rho,l}_i$. We therefore may drop an additional star in the notation the sets $V^{r^{l-1},l,i}_{v+r+},~V^{r^{l-1},l,i}_{v-r+}$ etc.

The  control function increment $\delta r^l_i=r^{l}_i-r^{l-1}_i$ is defined as follows. Recall the definition of the function $g^l_i$ in section 2. For some $C>0$ (determined below), given $l\geq 1$, and for all $(\tau ,x)\in [l-1,l]\times {\mathbb R}^n$ we consider
\begin{equation}\label{controlincrementint}
\begin{array}{ll}
\delta r^l_i(\tau,x):=
\left\lbrace \begin{array}{ll}
 -\int_{l-1}^{\tau}\int_{{\mathbb R}^n}g^l_i(y) p_l(\tau-s,x-y)dyds\\
 \\
  \mbox{ if } x\in \overline{V^{r^{l-1},l,i}_{v+r+}\cup V^{r^{l-1},l,i}_{v-r-}},\\
 \\
 -\int_{l-1}^{\tau}g^l_i(y) p_l(\tau-s,x-y)dyds\\
 \\
  \mbox{ if } x\in  V^{r^{l-1},l,i}_{v+r-}\cup V^{r^{l-1},l,i}_{v-r+}.
 
\end{array}\right.
\end{array}
\end{equation}
Here, in (\ref{controlincrementint}) we used
\begin{equation}
\overline{V^{l-1,i}_{v+r+}\cup V^{l-1,i}_{v-r-}},
\end{equation}
 the closure of the $V^{l-1,i}_{v+r+}\cup V^{l-1,i}_{v-r-}$ with respect to the standard topology on ${\mathbb R}^n$. Note that the local heat kernel smoothes the control function which would be only Lipschitz continuous otherwise.  
Now assume that at some time $l-1$ with $l\geq 1$ we have global upper bounds of the controlled velocity functions $v^{r,*,\rho,l-1}_i(l-1,.)$ and of the control functions $r^{l-1}_i(l-1,.)$ such that for some $C\gg 1$
\begin{equation}
\max_{1\leq i\leq n,~0\leq |\alpha|\leq m}\sup_{x\in {\mathbb R}^n}{\big |}D^{\alpha}_xv^{r,*,\rho,l-1}_i(l-1,x){\big |}\leq C,
\end{equation}
and
\begin{equation}
\max_{1\leq i\leq n,~0\leq |\alpha|\leq m}\sup_{x\in {\mathbb R}^n}{\big |}D^{\alpha}_xr^{l-1}_i(l-1,x){\big |}\leq C^2.
\end{equation}
In case of $l=1$ we may assume that we have the upper bound
\begin{equation}
\max_{1\leq i\leq n,~0\leq |\alpha|\leq m}\sup_{x\in {\mathbb R}^n}{\big |}D^{\alpha}_xh_i(x){\big |}\leq C
\end{equation}
for the initial data $h_i$, where we define
define
\begin{equation}\label{controlstart}
r^{0}_i(0,.):=\frac{h_i(.)}{C}
\end{equation}
in order to initialize the control function (note that we have the same sign for the control function and the controlled velocity function at this time but this does not really matter as we have proportionality $\frac{1}{C}$ for big $C>1$).
In order to observe the growth for the next to time steps from time $l-1$ to time $l+1$ we first consider an appropriate choice of the constant $C$ in relation to the time step size $\rho_l>0$.  For  simplicity we consider the growth of the value functions themselves, i.e., we mainly concentrate on growth estimates for the controlled velocity function $v^{r,*,\rho,l}_i,~1\leq i\leq n$ and for the control functions $r^{l}_i,~1\leq i\leq n$ themselves  in a first step. This is essential and simplifies the notation. The argument for spatial derivatives is quite similar.  Recall that at time $l$ we have for all $1\leq i\leq n$
\begin{equation}
v^{r,*,\rho,l}_i(l,.)=v^{r^{l-1},*,\rho,l-1}_i(l,.)+\delta r^{l}_i(l,.)
\end{equation}
where $v^{r^{l-1},*,\rho,l-1}_i,~1\leq i\leq n$ is the solution of the local uncontrolled Navier Stokes equation with controlled data, i.e., we have
\begin{equation}
v^{r^{l-1},*,\rho,l-1}_i(l-1,.)=v^{r,*,\rho,l-1}_i(l-1,.),
\end{equation}
and 
\begin{equation}\label{bothsides}
\begin{array}{ll}
\frac{\partial v^{r^{l-1},*\rho,l}_i}{\partial \tau}=\rho_l\nu\sum_{j=1}^n \frac{\partial^2 v^{r^{l-1},*,\rho,l}_i}{\partial x_j^2} 
-\rho_l\sum_{j=1}^n v^{r^{l-1},*,\rho,l}_j\frac{\partial v^{r^{l-1},*\rho,l}_i}{\partial x_j}+\\
\\
\rho_l\sum_{j,m=1}^n\int_{{\mathbb R}^n}\left( \frac{\partial}{\partial x_i}K_n(x-y)\right) \sum_{j,m=1}^n\left( \frac{\partial v^{r^{l-1},*,\rho,l}_m}{\partial x_j}\frac{\partial v^{r^{l-1},*,\rho,l}_j}{\partial x_m}\right) (\tau,y)dy
\end{array}
\end{equation}
Local time analysis of the Navier Stokes equation (e.g. using the local contraction result above shows that 
\begin{equation}
\max_{1\leq i\leq n,~0\leq |\alpha|\leq m}\sup_{x\in {\mathbb R}^n}{\big |}D^{\alpha}_xv^{r^{l-1},*,\rho,l-1}_i(l-1,x){\big |}\leq C+1,
\end{equation}
 
Integrating both sides of equation (\ref{bothsides}) from time $l-1$ to time $l$ we get
the estimate
\begin{equation}\label{deltavnew}
\begin{array}{ll}
\sup_{x\in {\mathbb R}^n}{\big |}\delta v^{r^{l-1},*\rho,l}_i(l,x){\big |}\\
\\
\leq \rho_l\nu n(C+1) +\rho_l n(C+1)^2 +\rho_l n^2C_{\mbox{kp}}(C+1)^2,
\end{array}
\end{equation}
where $C_{kp}$ is a constant which is related to the Laplacian kernel $K_n$ and to the inheritance of some order of polynomial decay. Note that this means that we can get any given upper bound $\epsilon >0$ of the left side in (\ref{deltavnew}) choosing the time step size
\begin{equation}
\rho_l=\frac{\epsilon}{\nu n(C+1) +n(C+1)^2 +n^2C_{\mbox{kp}}(C+1)^2}.
\end{equation}
Then it can be observed that for small $\epsilon >0$ we have that for all $x\in {\mathbb R}^n$ the function value
\begin{equation}
{\big |}v^{r,*,\rho,l}_i(l,x)+r^l_i(l,x){\big |}
\end{equation}
has the tendency to decrease, i.e., decreases after finitely many time steps while we have a uniform upper bound $C^2+1$ for the control functions if we start with (\ref{controlstart}). Next we shall analyze this more precisely, and we shall extend the analysis further below in the section on global uniform upper bounds, but provide the main ideas next. Note that $C$ is large and we start with (\ref{controlstart}) and with a small time step size $\rho_l$ which leads to a small upper bound $\epsilon >0$ for the moduli increments $\delta v^{r^{l-1},*,\rho,l}_i$ (and their derivatives up to order $2$ at least). Assume at time step $l-1$ the upper bounds $C$ and $C^2+1$ have been realized for the controlled velocity functions and the control function respectively (this is clearly true for $l=1$). Let us assume first that the function values $v^{r,*,\rho,l-1}_i(l-1,x)$ and $r^{l-1}_i(l-1,x)$ have the same sign which is true at time $l=1$. A local solution of the local uncontrolled incompressible Navier Stokes equation with controlled data leads to the function $v^{r^{l-1},*,\rho,l}_i(l,.)$. Then for given $1\leq i\leq n$ and $x\in {\mathbb R}^n$ we compare the signs of the values $v^{r^{l-1},*,\rho,l}_i(l,x)$ and the control function value $r^{l-1}_i(l-1,x)$. We consider some essential cases and assume that the modulus of the control function value is larger than $\frac{C^2}{2}$ at time $l-1$ such that for large $C$ it keeps its sign for several time steps. As both function values have the same sign at time step $l-1$ and $l$ at argument $x$ respectively, the control function decreases and considering several cases we then have the upper bound $C^2-\frac{1}{C^2}+\epsilon$ for the function ${\big |}r^{l}_i(l,x){\big |}$, where the $\epsilon >0$ is from an upper bound of the increment of the uncontrolled velocity function with controlled data of course. If the control function $v^{r,*,\rho,l}_i(l,x)$ still has the same sign as $v^{r,*,l-1}_i(l-1,x)$, then the modulus of the former value is strictly smaller than the value of the former. In this case both the control function value and the controlled velocity function value have decreased. If on the other hand the function values $v^{r,*,\rho,l}_i(l,x)$ and $v^{r,*,l-1}_i(l-1,x)$ have different signs, then the modulus of the former value is naturally bounded by the maximal decrease of the modulus of the control function value at $x$ for time $l-1$ to time $l$ plus a small upper bound of a velocity function increment , i.e., it is bounded by $1+\frac{1}{C^2}+\epsilon$. If in the latter case the modulus of $v^{r,*,\rho,l}_i(l,x)$ is larger then $\epsilon >0$, which is an upper bound of the increment $\delta v^{r^{l},*,\rho,l}_i(l+1,x)$, then in the next step we have different signs of $v^{r^l,*,\rho,l+1}_i(l+1,x)$ and $r^{l+1}_i(l+1,x)$. In this case we get
\begin{equation}
v^{r,*,\rho,l+1}_i(l+1,x)=v^{r^l,*,\rho,l+1}_i(l+1,x)+\delta r^l_i(l+1,x),
\end{equation}
where the control function increment value
\begin{equation}\label{controlincrementl+17}
\begin{array}{ll}
\delta r^{l}_i(l+1,x)=
-\int_{l}^{l+1}g^l_i(y)p_{l+1}(l-s,x-y)dy ds
\end{array}
\end{equation}
is close to
\begin{equation}\label{controlincrementl+18}
\begin{array}{ll}
-\int_{l}^{l+1}\left( \frac{2v^{r^{l},*,\rho,l+1}_i(l+1,y)}{C}+\frac{r^{l}_i(l,y)}{C^2}\right)p_{l+1}(l-s,x-y)dy ds
\end{array}
\end{equation}
The diffusion effect due to $p_{l+1}$ becomes arbitrary small as the time step size becomes small. This convolution with $p_l$ has a smoothing effect which is useful. This means that for small time step size the control function increment value in (\ref{controlincrementl+17}) is close to
\begin{equation}\label{closeob}
-\frac{2v^{r^{l},*,\rho,l+1}_i(l+1,x)}{C}-\frac{r^l_i(l,x)}{C^2}
\end{equation}
An upper bound for the modulus of the first summand is $\frac{2\left(1+\frac{1}{C^2}+\epsilon\right)}{C}$, and as we assumed that ${\big |}r^{l-1}_i(l-1,x){\big |}\geq \frac{C}{2}$, the upper bound for the modulus of the second summand, as we estimate several case at the same time). This means that in the case considered the control function decreases effectively form time $l$ to $l+1$ at the argument $x$. Hence, in the case considered in two time steps we have have a control function which decreases twice (recall the assumption that it is $\geq \frac{C^2}{2}$ at time $l-1$ in order to consider an 'interesting' case. The decrease of the control function part caused by the second term in (\ref{closeob}) is at least close two $1$ over two time steps in the case considered.  We get the upper bound
\begin{equation}
\begin{array}{ll}
{\big |}r^{l+1}_i(l+1,x){\big |}\leq {\big |}r^{l-1}_i(l-1,x){\big |}- \frac{2\left(1+\frac{2}{C^2}+\epsilon\right)}{C}-1-\frac{2}{C^2}-2\epsilon\\
\\
< {\big |}r^{l-1}_i(l-1,x){\big |}-1+\frac{1}{5}+\frac{1}{10}
\end{array}
\end{equation}
for $C\geq 10$ and $\epsilon =\frac{1}{C^2}$ for example. Hence we have a decreasing control function over two time steps in the considered case. There are still several possibilities for the controlled velocity function value $v^{r,*,\rho,l+1}_i(l+1,x)$ after two time steps. In the case considered we had the upper bound $1+\frac{1}{C^2}+\epsilon$ for a function $v^{r^l,*,\rho,l}_i(l+1,x)$ in the subcases where the latter function has a different sign from the control function value $r^l_i(l,x)$. It is clear then that it depends on the summands in (\ref{closeob}) whether $v^{r,*,\rho,l+1}_i(l+1,x)$ gets the same sign as $r^{l+1}_i(l+1,x)$. Several different subcases can be considered and the upshot is that the modulus of $v^{r,*,\rho,l+1}_i(l+1,x)+r^{l+1}_i(l+1,x)$ is smaller than the modulus of $v^{r,*,\rho,l-1}_i(l-1,x)+r^{l-1}_i(l-1,x)$ in the case considered. Going through all possible (finitely many) cases observations of this case  lead to the conclusion that the function $l\rightarrow v^{r,*,\rho,l}_i(l,x)+r^{l}_i(l,x)$ and $l\rightarrow r^l_i(l,x)$ have a tendency to fall simultaneously, i.e., for each $l$ there is an $m+l>l$ such that at time $m+l$ the moduli of both function values are smaller. This leads to an uniform global upper bound for the controlled velocity function and the control function, and hence to a global uniform upper bound for the velocity solution function of the incompressible Navier Stokes equation itself. Note furthermore that analogous considerations can be applied to spatial derivatives of the control functions and the controlled velocity functions. We come back to these considerations and to their possible impact on investigation of the long time behavior of velocity functions for the solutions of the incompressible Navier Stokes equation below in the section on the global upper bounds for the Leray projection term.   

\section{Proof of theorems \ref{mainthm} and \ref{mainthm2}}
Some variations of our argument do not use polynomial decay of certain orders of the initial data $v^{*,\rho,l-1}_i(l-1,.)$ and their spatial derivatives up to order $p=2$ at time step $l\geq 1$ and of the functional increment $\delta v^{\rho,l,1}_i$ and their spatial derivatives up to order $p= 2$, but it is convenient to have this, and something close to this seems to be needed. The last section showed that polynomial decay of order $2$ for derivatives up to order $2$  to the same order is inherited by the higher correction functions $\delta v^{*,\rho,l,k}_i$. 
We shall see below that for the controlled scheme  polynomial decay of the functions $v^{r,*,\rho,l,k}_i$ of order $m$ for derivatives up to order $m$ is inherited under weaker conditions even (due to the definition of the control functions $r^l_i$.
In a first variation of this proof we shall start with the assumption
\begin{equation}\label{inductivehyplminus1}
\max_{1\leq i\leq n}\sum_{|\alpha|\leq 2}\sup_{y\in {\mathbb R}^3}\left(1+|y|^2\right) {\Big |}D^x_{\alpha}v^{*,\rho,l-1}_i(l-1,y){\Big |}\leq C^{l-1}.
\end{equation}
Later we shall see that for the controlled scheme this polynomial decay of order $m\geq 2$ for derivatives up to order $m$ is inherited for weaker assumptions even. This means that we can transfer the result to the controlled scheme easier, if we work with the assumption (\ref{inductivehyplminus1}) first. In a second step we shall see that we can even get rid of this assumption and observe that the original assumption $v^{*,\rho,l-1}_i(l-1,.)\in C^2\cap H^2$ (which is in fact close) is sufficient. Furthermore we shall observe that for this first variation of argument it is effectively sufficient to have
\begin{equation}\label{inductivehyp0}
\max_{1\leq i\leq n}\sum_{|\alpha|\leq 2}\sup_{\tau\in [l-1,l],y\in {\mathbb R}^3}\left(1+|y|^{2-\alpha}\right) {\Big |}D^x_{\alpha}v^{*,\rho,l,k-1}_i(\tau,y){\Big |}\leq C^l_{k-1}
\end{equation}
for all $\tau\in [l-1,l]$ for some $\alpha \in (0,0.5)$ in the subiteration steps $k$. 
We have to control the growth of the series
\begin{equation}
(C^l_k)_k,
\end{equation}
but this will follow naturally from contraction.
 The argument of the contraction results ensure that the sequence of constants $C^l_k,k\geq 1$, starting with $C^l_0:=C^{l-1}$ has a uniform upper bound independent of the substep number $k$ of the form
 \begin{equation}
 C^l_k\leq C^{l-1}+1.
 \end{equation}
We shall choose the time step size $\rho_l$ such that this is true.
The hypothesis of form (\ref{inductivehyplminus1}) allows us to give weights to some factors in the $L^2$-estimates. An alternative is the use of Fourier transforms together with $L^2$-estimates for weighted products in the $L^2$-norms. This leads to $H^2$ estimates naturally. A closely related alternative is to use Fourier transforms as we did in \cite{KB2}, but the first method simplifies some arguments. First we consider $L^2$ theory. We start with $L^2$ estimates, and refine the argument such that in the end we have $H^m$-estimates for arbitrary $m\geq 0$. The $C^0\left( (l-1,l), H^m\right)$-estimates and $C^p\left( (l-1,l), H^m\right)$ for $p\geq 1$ are obtained similarly.

In order to prove the contraction property (\ref{contract}) we first consider representations of $\delta v^{\rho,l,k}_i=v^{\rho,l,k}_i-v^{\rho,l,k-1}_i$ in terms of the fundamental solution 
$G_l$ ( a heat kernel) defined above.
First observe that the equation for $\delta v^{\rho,l,k}_i$ is
\begin{equation}\label{NavlerayII*}
\left\lbrace \begin{array}{ll}
\frac{\partial \delta v^{*,\rho,l,k}_i}{\partial \tau}-\rho_l\nu\sum_{j=1}^n \frac{\partial^2 \delta v^{*,\rho,l,k}_i}{\partial x_j^2} 
=\\
\\ 
-\rho_l\sum_{j=1}^n v^{*,\rho,l,k-1}_j\frac{\partial \delta v^{*,\rho,l,k-1}_i}{\partial x_j}-\rho_l\sum_{j=1}^n \delta v^{*,\rho,l,k-1}_j\frac{\partial  v^{*,\rho,l,k-1}_i}{\partial x_j}
\\
\\
+\rho_l\sum_{j,m=1}^n\int_{{\mathbb R}^n}\left( \frac{\partial}{\partial x_i}K_n(x-y)\right) \sum_{j,m=1}^n\left( \frac{\partial v^{*,\rho,l,k-1}_m}{\partial x_j}\frac{\partial v^{*,\rho,l,k-1}_j}{\partial x_m}\right) (\tau,y)dy\\
\\
-\rho_l\sum_{j,m=1}^n\int_{{\mathbb R}^n}\left( \frac{\partial}{\partial x_i}K_n(x-y)\right) \sum_{j,m=1}^n\left( \frac{\partial v^{*,\rho,l,k-2}_m}{\partial x_j}\frac{\partial v^{*,\rho,l,k-2}_j}{\partial x_m}\right) (\tau,y)dy,\\
\\
\delta \mathbf{v}^{*,\rho,l,k}(l-1,.)=0.
\end{array}\right.
\end{equation}
Again this equation is considered on the domain $[l-1,l]\times {\mathbb R}^n$.
Next in terms of the fundamental solution $G_l$, and for all $1\leq i\leq n$ we have the representation
\begin{equation}\label{solrepk}
\begin{array}{ll}
\delta v^{*,\rho,l,k}_i(\tau,x)=\\
\\
+\rho_l\int_{l-1}^{\tau}\int_{{\mathbb R}^n}\left( -\sum_{j=1}^n \delta v^{*,\rho,l,k-1}_j\frac{\partial  v^{*,\rho,l,k-1}_i}{\partial x_j}\right)(s,y)G_l(\tau-s;x-y)dyds\\
\\
+\rho_l\int_{l-1}^{\tau}\int_{{\mathbb R}^n}\left( -\sum_{j=1}^n  v^{*,\rho,l,k-1}_j\frac{\partial  \delta v^{*,\rho,l,k-1}_i}{\partial x_j}\right)(s,y)G_l(\tau-s;x-y)dyds\\
\\
+\rho_l\int_{l-1}^{\tau}\int_{{\mathbb R}^n}\int_{{\mathbb R}^n}\left( \frac{\partial}{\partial x_i}K_n(z-y)\right)\times\\
\\
\times \left( \sum_{m,j=1}^n \frac{\partial \delta v^{*,\rho,l,k-1}_j}{\partial x_m}\left(\frac{\partial v^{*,\rho,l,k-1}_m}{\partial x_j}+\frac{\partial v^{*,\rho,l,k-2}_m}{\partial x_j}\right)\right)  (s,y)\times \\
\\
\times G_l(\tau-s,x-z)dydzds,
\end{array}
\end{equation}
where for convenience we rewrite the difference of the Leray projection terms observing that
\begin{equation}
\begin{array}{ll}
 \sum_{j,m=1}^n\left( \frac{\partial v^{*,\rho,l,k-1}_m}{\partial x_j}\frac{\partial v^{*,\rho,l,k-1}_j}{\partial x_m}\right)
 -\sum_{j,m=1}^n\left( \frac{\partial v^{*,\rho,l,k-2}_m}{\partial x_j}\frac{\partial v^{*,\rho,l,k-1}_j}{\partial x_m}\right)\\
 \\
 +\sum_{j,m=1}^n\left( \frac{\partial v^{*,\rho,l,k-2}_m}{\partial x_j}\frac{\partial v^{*,\rho,l,k-1}_j}{\partial x_m}\right)
 -\sum_{j,m=1}^n\left( \frac{\partial v^{*,\rho,l,k-2}_m}{\partial x_j}\frac{\partial v^{*,\rho,l,k-2}_j}{\partial x_m}\right)\\
\\
=\sum_{j,m=1}^n\left( \frac{\partial \delta v^{*,\rho,l,k-1}_m}{\partial x_j}\frac{\partial v^{*,\rho,l,k-1}_j}{\partial x_m}\right)+\sum_{j,m=1}^n\left( \frac{\partial v^{*,\rho,l,k-2}_m}{\partial x_j}\frac{\partial \delta v^{*,\rho,l,k-1}_j}{\partial x_m}\right)\\
\\
=\sum_{m,j=1}^n\left( \frac{\partial \delta v^{*,\rho,l,k-1}_j}{\partial x_m}\frac{\partial v^{*,\rho,l,k-1}_m}{\partial x_j}\right)+\sum_{j,m=1}^n\left( \frac{\partial v^{*,\rho,l,k-2}_m}{\partial x_j}\frac{\partial \delta v^{*,\rho,l,k-1}_j}{\partial x_m}\right)\\
\\
=\sum_{m,j=1}^n \frac{\partial \delta v^{*,\rho,l,k-1}_j}{\partial x_m}\left(\frac{\partial v^{*,\rho,l,k-1}_m}{\partial x_j}+\frac{\partial v^{*,\rho,l,k-2}_m}{\partial x_j}\right).
\end{array}
\end{equation}
The latter observation is natural if you want to extract functional increments $\delta v^{*,\rho,k,l}_i$ in order to obtain local contraction results.
The integrals in the representation in (\ref{solrepk}) are convolutions with respect to space and time. It is clear that it is convenient to estimate time and space in Hilbert spaces of the same type (,i.e, $H^m\left((l-1,l)\times {\mathbb R}^n \right) $-norms, because we may then apply tools such as Young's inequality to space and time simultaneously.  Later, we consider $C^0\left( (l-1,l), H^m\right)$- estimates, where we shall see that we can treat time variables and spatial variables differently. Note that we can treat time and space variables separately due to applications of Fubbini's theorem. For this reason the latter type of estimates is not more difficult than the former type.
\subsection{$L^2$-estimates}
We shall use Young's inequality.
 Let $g\in L^1\left([a,a+1]\times {\mathbb R}^n \right) $ and $f\in L^2\left([a,t]\times {\mathbb R}^n \right)$ for some $a+1\geq t> a\geq 0$. It is convenient we also consider trivial extensions of these functions which are zero outside a compact time interval.  Basically we use the observation (known as the standard form of Young's inequality) that
\begin{equation}
\begin{array}{ll}
{\big |}\int_{a}^t\int_{{\mathbb R}^n}f(s,y)g(t-s,x-y)dyds{\big |}_{L^2\times L^2}\\
\\
={\big |}\int_{a}^t\int_{{\mathbb R}^n}f(t-s,x-y)g(s,y)dyds{\big |}_{L^2\times L^2}\\
\\
\leq {\big |}\int_{a}^{a+1}\int_{{\mathbb R}^n}|f^{-s,-y}(t,x)||g(s,y)|dyds{\big |}_{L^2\times L^2}\\
\\
\leq |f|_{L^2\times L^2}|g|_{L^1\times L^1}ds
\end{array}
\end{equation}
where $f^{-s,-y}(x):=f(t-s,x-y)$ denotes the function $f$ shifted by $-s$ and $-y$, and we may use Minkowski's inequality. The functions considered are local with respect to time but we may consider $ s\rightarrow  |f(s,.)|_{L^p},~t\rightarrow |g^{-s}(t,.)|_{L^1}$ as functions in $L^1\left( {\mathbb R}\right)$ by defining them to be zero in ${\mathbb R}\setminus [a,a+1]$. 

 Now this type of Young inequalities may not be applied immediately in our situation, where one part of the convolution is a Gaussian.
The application of Fourier transforms of fundamental solutions of the heat equation with respect to space {\it and} time variables shows this. For the partial (not normed) Fourier transformation with respect to the spatial variables we get for $t>s$
\begin{equation}
  \begin{array}{ll}
\int_{{\mathbb R}^n}\exp(2\pi i\xi z){\big |}\frac{1}{(2\sqrt{\epsilon \pi (t-s)})^n}\exp\left(-\frac{z^2}{4\epsilon (t-s) } \right){\big |}dz\\
\\
=\exp\left(-4\epsilon(t-s)\pi^2\xi^2\right). 
\end{array}
\end{equation}
Now as $s\uparrow t$ the right side of this equation becomes $1$ (the formal Fourier transform of the $\delta$-distribution), and this is not in $L^1$. However, for the truncated fundamental solution
\begin{equation}
\phi_1(z)G_{\epsilon}(t-s,z):=\phi_1(z)
\frac{1}{(2\sqrt{\epsilon \pi (t-s)})^n}
\exp\left(-\frac{z^2}{4\epsilon (t-s)}\right)  
\end{equation}
the situation is different. Here, let $\phi_1\in C^{\infty}\left(B_1(0))\right)$, i.e., with support in $B_1(0)$, and such that $\phi_1(x)=1$ for $|x|\leq 0.5$. Note that $\phi_1$ and $1-\phi_1$ build a smooth partition of unity on ${\mathbb R}^n$. So the idea for estimating the increments $\delta v^{*,\rho,l,k}_i$ is to split up the integral of their representation and estimate one convolution summand with factor
\begin{equation}
G^B_{\epsilon}(t-s,z):=\phi_1(z)G_{\epsilon}(t-s,z)
\end{equation}
via the Young inequality and the other convolution summand with factor
\begin{equation}
G^{(1-B)}_{\epsilon}(t-s,z):=(1-\phi_1(z))G_{\epsilon}(t-s,z)
\end{equation}
via a weighted convolution estimate for $L^2$ functions. The following estimate holds for the Gaussian itself, and a fortiori for the truncated Gaussian. We note that we have it for localized Gaussian, because we need it in this case. Indeed for the truncated Gaussian $\phi_1(z)G_{\epsilon}(t-s,z)$ we have for $t\neq s$ and $x\neq y$
 \begin{equation}\label{simpleest}
  \begin{array}{ll}
{\big |}\phi_1(x-y)G_{\epsilon}(t-s,x-y){\big |}={\big |}\phi_1(x-y)\frac{1}{(2\sqrt{\epsilon \pi (t-s)})^n}\exp\left(-\frac{(x-y)^2}{4\epsilon (t-s) } \right){\big |}\\
\\
={\big |}\phi_1(x-y)(t-s)^{m-n/2}\frac{1}{(x-y)^{2m}}\left( \frac{(x-y)^2}{ (t-s)}\right)^m \frac{1}{(2\sqrt{\epsilon \pi })^n}\exp\left(-\frac{(x-y)^2}{4\epsilon (t-s) } \right){\big |}\\
\\
\leq {\big |}\phi_1(x-y)(t-s)^{m-n/2}\frac{1}{(x-y)^{2m}}\left( \frac{(x-y)^2}{ (t-s)}\right)^m \frac{1}{(2\sqrt{\epsilon \pi })^n}\exp\left(-\frac{(x-y)^2}{4\epsilon (t-s) } \right){\big |}\\
\\
\leq {\big |}\phi_1(x-y)C(t-s)^{m-n/2}\frac{1}{(x-y)^{2m}} {\big |},
\end{array}
\end{equation}
where for the local estimates considered here (fixed time step $l$) we consider a constant
\begin{equation}
C:=\sup_{|z|\geq 0,t>0}{\big |}\frac{1}{(2\sqrt{\epsilon \pi })^n}\left( \frac{z^2}{ t}\right)^m\exp\left(-\frac{z^2}{4\epsilon t } \right){\big |}>0
\end{equation}
depending on $\epsilon$, but finite for each $\epsilon >0$. (The dependence on $\epsilon $ of order $\frac{1}{\sqrt{\epsilon}^n}$ for $z=0$ does no harm for our estimate anyway since we use integral norms for the Gaussian $G_l$ and its first spatial derivative for $n=3$ which gives a factor of the Lebesgues measure which cancels the singularity $\frac{1}{\sqrt{\epsilon}^n}$ as $\epsilon$ becomes small (transformed coordinates).)
We may use this estimate locally for $m=1$ and $n=3$, i.e., we may use the upper bound
\begin{equation}
{\big |}\phi_1(x-y)G_{\epsilon}(t-s,x-y){\big |}\leq {\big |}\phi_1(x-y)C(t-s)^{-1/2}\frac{1}{(x-y)^{2}} {\big |}
\end{equation}
which is $L^1$ for dimension $n=3$, because of the localisation, and this may be used with Young inequalities (of standard and of mixed form). 
We start with $L^2$ estimates for 
$\delta v^{*,\rho,l,k}_i(\tau,.)= v^{*,\rho,l,k}_{i}(\tau,.)-v^{*,\rho,l,k-1}_{i}(\tau,.)$,
where we define
\begin{equation}\label{vrep}
\begin{array}{ll}
\delta v^{*,\rho,l,k}_{ i}(\tau,x)=:\delta v^{*,\rho,l,k,1}_{ i}(\tau,x)+\delta v^{*,\rho,l,k,2}_{i}(\tau,x)+\delta v^{*,\rho,l,k,3}_{ i}(\tau,x)\\
\\
:=\rho_l\int_{l-1}^{\tau}\int_{{\mathbb R}^n}\left( -\sum_{j=1}^n \delta v^{*,\rho,l,k-1}_j\frac{\partial  v^{*,\rho,l,k-1}_i}{\partial x_j}\right)(s,y)\times\\
\\
\times G_l(\tau-s;x-y)dyds\\
\\
+\rho_l\int_{l-1}^{\tau}\int_{{\mathbb R}^n}\left( -\sum_{j=1}^n  v^{*,\rho,l,k-1}_j\frac{\partial  \delta v^{*,\rho,l,k-1}_i}{\partial x_j}\right)(s,y)\times\\
\\
\times G_l(\tau-s;x-y)dyds\\
\\
+\rho_l\int_{l-1}^{\tau}\int_{{\mathbb R}^n}\int_{{\mathbb R}^n}\left( \frac{\partial}{\partial x_i}K_n(z-y)\right)\times\\
\\
\times \left( \sum_{m,j=1}^n \frac{\partial \delta v^{*,\rho,l,k-1}_j}{\partial x_m}\left(\frac{\partial v^{*,\rho,l,k-1}_m}{\partial x_j}+\frac{\partial v^{*,\rho,l,k-2}_m}{\partial x_j}\right)\right)  (s,y)\times\\
\\
\times G_l(\tau-s,x-z)dydzds.
\end{array}
\end{equation}
Here, the right side of the first definition in (\ref{vrep}) is defined by the right side of the second definition with a correspondence which is summand by summand.
Furthermore, note that we have convolutions with respect to time and with respect to the spatial variables. We do some estimates for the $L^2\times L^2$-norm and for the squared $|.|_{L^2\times L^2}$-norm. For fixed $x\in {\mathbb R}^n$ we may consider $\tau\rightarrow \delta v^{*,\rho,l,k,1}_{\epsilon i}(\tau,x)=1_{[l-1,l]}\delta v^{*,\rho,l,k,1}_{\epsilon i}(\tau,x)$ as a $L^p$-function for $p\geq 1$ on ${\mathbb R}$ where $1_{[l-1,l]}$ is the function which equals $1$ on the interval $[l-1,l]$ and is $0$ on ${\mathbb R}\setminus [l-1,l]$. We implicitly assume this for convenience without changing the symbol of the function. We now apply a Young inequality to the first term on the right side of (\ref{vrep}). In the following we write 
\begin{equation}
G^B_l=\phi_{1}G_l, \mbox{ and }G^{1-B}_l=(1-\phi_{1})G_l,
\end{equation}
where $\phi_1=\phi_{\epsilon}$ with $\epsilon=1$.
For the first term $\delta v^{*,\rho,l,k,1}_{B i}$ it suffices to use the assumption
\begin{equation}\label{smallass}
\max_{1\leq j\leq n}\sup_{s\in [l-1,l],y\in {\mathbb R}^3}{\Big |}\frac{\partial  v^{*,\rho,l,k-1}_i}{\partial x_j} (s,y){\Big |}\leq C^l_{k-1}.
\end{equation}
We have
\begin{equation}
\begin{array}{ll}
{\big |}\delta v^{*,\rho,l,k,1}_{B i}(\tau,.){\big |}^2_{L^2\times L^2}:=\\
\\
={\big |}\rho_l\int_{l-1}^{\tau}\int_{{\mathbb R}^n}\left( -\sum_{j=1}^n \delta v^{*,\rho,l,k-1}_j\frac{\partial  v^{*,\rho,l,k-1}_i}{\partial x_j}\right)(s,y)\times\\
\\
\times G^B_l(\tau-s;.-y)dyds{\big |}^2_{L^2\times L^2}\\
\\
\leq {\big |}\rho_l\int_{l-1}^{\tau}\int_{{\mathbb R}^n}{\big |}\left( \sum_{j=1}^n \delta v^{*,\rho,l,k-1}_j\frac{\partial  v^{*,\rho,l,k-1}_i}{\partial x_j}\right) (s,y){\big |}\times\\
\\
\times {\big |}G^B_l(\tau-s;.-y){\big |}dyds{\big |}^2_{L^2\times L^2}\\
\\
\leq {\big |}\rho_lnC^l_{k-1}\int_{l-1}^{\tau}\int_{{\mathbb R}^n}{\big |}\left( \sum_{j=1}^n |\delta v^{*,\rho,l,k-1}_j|\right) (s,y){\big |}\times\\
\\
\times {\big |}G^B_l(\tau-s;.-y){\big |}dyds{\big |}^2_{L^2\times L^2}\\
\\
\leq \rho_l^2(C^l_{k-1})^2n^2\max_{j\in \left\lbrace 1,\cdots ,n\right\rbrace } 
{\big |} \delta v^{*,\rho,l,k-1}_j{\big |}^2_{L^2\times L^2}
{\big |}G^B_l{\big |}^2_{L^1\times L^1}\\
\\
\leq \rho_l^2(C^l_{k-1})^2n^2(C^B_{G})^2\max_{j\in \left\lbrace 1,\cdots ,n\right\rbrace } |\delta v^{*,\rho,l,k-1}_j{\big |}^2_{L^2\times L^2},
\end{array}
\end{equation}
where $D^{\alpha}_x$ denotes the multivariate partial derivative of order $\alpha$ with multiindex $\alpha=(\alpha_1,\alpha_2,\alpha_3)$. Furthermore, we used the notation
\begin{equation}
C^B_{G}:={\big |}G^B_l{\big |}_{L^1\times L^1}.
\end{equation}
An analogous argument leads to
\begin{equation}
\begin{array}{ll}
{\big |}\delta v^{*,\rho,l,k,1}_{(1-B) i}{\big |}^2_{L^2\times L^2}:=\\
\\
={\big |}\rho_l\int_{l-1}^{\tau}\int_{{\mathbb R}^n}\left( -\sum_{j=1}^n \delta v^{*,\rho,l,k-1}_j\frac{\partial  v^{*,\rho,l,k-1}_i}{\partial x_j}\right)(s,y)\times\\
\\
G^{(1-B)}_l(\tau-s;.-y)dyds{\big |}^2_{L^2\times L^2}\\
\\
\leq \rho_l^2(C^l_{k-1})^2n^2(C^{1-B}_{G})^2\max_{j\in \left\lbrace 1,\cdots ,n\right\rbrace }  |\delta v^{*,\rho,l,k-1}_j(\tau,.){\big |}^2_{L^2\times L^2},
\end{array}
\end{equation}
where
\begin{equation}
C^{1-B}_{G}:={\big |}G^{1-B}_l{\big |}_{L^1\times L^1}.
\end{equation}
Summing up our result for the first term on the right side of (\ref{vrep}) we have for all $1\leq i\leq n$ 
\begin{equation}
\begin{array}{ll}
{\big |}\delta v^{*,\rho,l,k,1}_{ i}{\big |}^2_{L^2\times L^2}\leq \\
\\
\leq \rho_l^2(C^l_{k-1})^2n^2\left(C^B_G+ C^{1-B}_{G}\right)^2 \times\\
\\
\times \max_{j\in \left\lbrace 1,\cdots ,n\right\rbrace }  |\delta v^{*,\rho,l,k-1}_j(\tau,.){\big |}^2_{L^2}{\big |}G_l{\big |}_{L^1\times L^1}\\
\\
\leq \rho_l^2(C^l_{k-1})^2n^2C_G^2 \max_{j\in \left\lbrace 1,\cdots ,n\right\rbrace }  |\delta v^{*,\rho,l,k-1}_j{\big |}^2_{L^2\times L^2},
\end{array}
\end{equation}
where
\begin{equation}
C_{G}=C^B_G+ C^{1-B}_{G},
\end{equation}
and we used $(C^B_G)^2+(C^{(1-B)}_G)^2\leq \left(C^B_G+C^{(1-B)}_G\right)^2$.
The latter estimate holds for all  $1\leq i\leq n$, hence we have
\begin{equation}
\begin{array}{ll}
\max_{j\in \left\lbrace 1,\cdots ,n\right\rbrace } {\big |}\delta v^{*,\rho,l,k,1}_{ i}{\big |}^2_{L^2\times L^2}\leq \\
\\
\leq \rho_l^2(C^l_{k-1})^2n^2C_G^2 \max_{j\in \left\lbrace 1,\cdots ,n\right\rbrace } |\delta v^{*,\rho,l,k-1}_j(\tau,.){\big |}^2_{L^2\times L^2}.
\end{array}
\end{equation}
For the second term on the right side of (\ref{vrep}), i.e., for $\delta v^{*,\rho,l,k,2}_{ i}$ we get
the analogous estimate
\begin{equation}
\begin{array}{ll}
\max_{j\in \left\lbrace 1,\cdots ,n\right\rbrace } {\big |}\delta v^{*,\rho,l,k,2}_{ i}(\tau,){\big |}^2_{L^2}\leq \\
\\
\leq \rho_l^2(C^l_{k-1})^2n^2C_G^2 \max_{j\in \left\lbrace 1,\cdots ,n\right\rbrace }  | \frac{\partial}{\partial x_k}\delta v^{*,\rho,l,k-1}_j(\tau,.){\big |}^2_{L^2},
\end{array}
\end{equation}

Finally we look at the third term on the right side of (\ref{vrep}) which is a double convolution with respect to the spatial variables integrated over time. Here the local integrability of the first order derivatives of the Laplacian kernel in dimension $n=3$ is a further reason to specify to this dimension. Again we split up the function
\begin{equation}
\delta v^{*,\rho,l,k,3}_{ i}(\tau,.)=\delta v^{*,\rho,l,k,3}_{B i}(\tau,.)+\delta v^{*,\rho,l,k,3}_{(1-B) i}(\tau,.),
\end{equation}
corresponding to summands with a truncated Gaussian and its complement as above.
Each of this summand is again split up into two summands one of which corresponds to a truncated Laplacian kernel $\phi_1K_{,i}$ (more precisely: its $i$th partial derivative), and the complement $(1-\phi_1)K_{,i}$. We define
\begin{equation}
\delta v^{*,\rho,l,k,3}_{B i}(\tau,.)=\delta v^{*,\rho,l,k,3}_{BB i}(\tau,.)+\delta v^{*,\rho,l,k,3}_{B(1-B) i}(\tau,.)
\end{equation}
and
\begin{equation}
\delta v^{*,\rho,l,k,3}_{(1-B) i}(\tau,.)=\delta v^{*,\rho,l,k,3}_{(1-B)B i}(\tau,.)+\delta v^{*,\rho,l,k,3}_{(1-B)B(1-B) i}(\tau,.),
\end{equation}
where the references will be made precise in the following estimations. We start with the third term $\delta v^{*,\rho,l,k,3}_{BB i}$ which is defined via
\begin{equation}
\begin{array}{ll}
\delta v^{*,\rho,l,k,3}_{BB i}(\tau,.)

=\rho_l\int_{l-1}^{\tau}\int_{{\mathbb R}^n}\int_{{\mathbb R}^n}\left( \phi_1(z-y)\frac{\partial}{\partial x_i}K_n(z-y)\right)\times\\
\\
\times \left( \sum_{m,j=1}^n \frac{\partial \delta v^{*,\rho,l,k-1}_j}{\partial x_m}\left(\frac{\partial v^{*,\rho,l,k-1}_m}{\partial x_j}+\frac{\partial v^{*,\rho,l,k-2}_m}{\partial x_j}\right)\right)  (s,y)\times \\
\\
\times G^B_l(\tau-s,x-z)dydzds\\ 
\end{array}
\end{equation}
This means that we have a truncated kernel $\phi_1(.)K_{,i}(.)$ and a truncated Gaussian $G^B_l$. The double subscript $B$ indicates that we have bounded support in both cases. Recall that we set $\delta v^{*,\rho,l,k,3}_{BB i}(\tau,.)=0$ for $\tau\in {\mathbb R}\setminus [l-1,l]$ in order to write the time integrals conveniently.
We obtain
\begin{equation}\label{vrepineq3}
\begin{array}{ll}
|\delta v^{*,\rho,l,k,3}_{BB i}|^2_{L^2\times L^2}
=\int_{l-1}^l\int_{{\mathbb R}^n}|\delta v^{*,\rho,l,k,3}_{ BBi}(\tau,x)|^2dxd\tau\\
\\
=\int_{l-1}^l\int_{{\mathbb R}^n}
{\big |}\rho_l\int_{{\mathbb R}^n}\int_{{\mathbb R}^n}\left(  \phi_1(z-y)\frac{\partial}{\partial x_i}K_n(z-y)\right) \times\\
\\
\times \left( \sum_{m,j=1}^n \frac{\partial \delta v^{*,\rho,l,k-1}_j}{\partial x_m}\left(\frac{\partial v^{*,\rho,l,k-1}_m}{\partial x_j}+\frac{\partial v^{*,\rho,l,k-2}_m}{\partial x_j}\right)\right)  (s,y)\times \\
\\
\times G^B_l(\tau-s,x-z)dydz{\big |}^2dxd\tau\\
\\
\leq \int_{l-1}^l\int_{{\mathbb R}^n}
{\big |}\rho_l\int_{{\mathbb R}^n}{\big |}\int_{{\mathbb R}^n}{\big |} \phi_1(z-y)\frac{\partial}{\partial x_i}K_n(z-y){\big |}\times\\
\\
\times \left( \sum_{m,j=1}^n {\big |}\frac{\partial \delta v^{*,\rho,l,k-1}_j}{\partial x_m}{\big |}{\big |}\frac{\partial v^{*,\rho,l,k-1}_m}{\partial x_j}+\frac{\partial v^{*,\rho,l,k-2}_m}{\partial x_j}{\big |}\right)  (s,y){\big |}dy{\big |}\times \\
\\
{\big |}G^{B}_l(\tau-s,x-z){\big |}dz{\big |}^2dxd\tau\\
\\
\leq {\Big |}\rho_l\int_{{\mathbb R}^n}{\big |}\int_{{\mathbb R}^n}{\big |} \phi_1(z-y)\frac{\partial}{\partial x_i}K_n(z-y){\big |}\times\\
\\
\times \left( \sum_{m,j=1}^n {\big |}\frac{\partial \delta v^{*,\rho,l,k-1}_j}{\partial x_m}{\big |}2C^l_k\right)  (s,y){\big |}dy{\big |} {\big |}G^{B}_l(.-s,x-z){\big |}dz{\Big |}^2_{L^2\times L^2}\\
\\
\leq 2n^2(C^B_G)^2(C^l_k)^2\max_{1\leq m,j\leq n}{\Big |}\rho_l\int_{{\mathbb R}^n}{\big |}\int_{{\mathbb R}^n}{\big |} \phi_1(.-y)K_{n,i}(.-y){\big |}\times\\
\\
\times  {\big |} \delta v^{*,\rho,l,k-1}_{j,m}{\big |}  (.,y){\big |}dy{\big |}^2_{L^2\times L^2},

\end{array}
\end{equation}
where we may use the constant $C_G$ above, and where we still used only the assumption (\ref{smallass}) above. Since the function $\phi_1(.)\frac{\partial}{\partial x_i}K_n(.)$ is in $L^1$ we may apply the Young inequality again and obtain
\begin{equation}\label{vrepineq4}
\begin{array}{ll}
|\delta v^{*,\rho,l,k,3}_{BB i}(\tau,.)|^2_{L^2\times L^2}\\
\\
\rho_l^24(C^l_{k-1})^2(C^B_G)^2C_{K_3\phi_1}^2 
 \max_{m,j\in \left\lbrace 1,\cdots ,n\right\rbrace } {\big |}
\frac{\partial \delta v^{*,\rho,l,k-1}_j}{\partial x_m}{\big |}^2_{L^2\times L^2},
\end{array}
\end{equation}
where
\begin{equation}
C^B_G={\big |}G^B_l{\big |}_{L^1\times L^1}
\end{equation}
and
\begin{equation}
{\big |}\phi_{B_{1}(0)}\frac{\partial}{\partial x_i}K_n(.){\big |}_{L^1}\leq C_{K_3\phi_1}
\end{equation} 
for a finite constant $C_{K_3\phi_1}$.

Next for the second summand of $\delta v^{*,\rho,l,k,3}_{i}$, i.e., for the summand of the form $\delta v^{*,\rho,l,k,3}_{B(1-B) i}$ we have by an analogous argument 
\begin{equation}\label{vrepineq4}
\begin{array}{ll}
|\delta v^{*,\rho,l,k,3}_{B(1-B) i}|^2_{L^2\times L^2}
\leq \rho_l^24(C^l_{k-1})^2(C^{(1-B)}_G)^2n^2C_{K_3\phi_1}^2\times\\
\\
\times\max_{j,m\in\left\lbrace 1,\cdots,n\right\rbrace }
{\big |}\frac{\partial \delta v^{*,\rho,l,k-1}_j}{\partial x_m} (\tau,.){\big |}^2_{L^2}
\end{array}
\end{equation}
For the other two summands of $\delta v^{*,\rho,l,k,3}_{i}$, i.e., for the summand of the form $\delta v^{*,\rho,l,k,3}_{(1-B)B i}$ and $\delta v^{*,\rho,l,k,3}_{(1-B)(1-B) i}$ we have to estimate kernels of the form 
\begin{equation}
(1-\phi_1)K_{,i}
\end{equation} 
 while the treatment of the convolution with the Gaussian maintains. We follow the argument above as long as we can and get
\begin{equation}\label{vrepineq34}
\begin{array}{ll}
|\delta v^{*,\rho,l,k,3}_{(1-B)B i}|^2_{L^2\times L^2}\\
\\
\leq 2n^2(C^B_G)^2\max_{1\leq m,j\leq n}{\Big |}\rho_l\int_{{\mathbb R}^n}{\big |}\int_{{\mathbb R}^n}{\big |}(1- \phi_1(.-y))K_{n,i}(.-y){\big |}\times\\
\\
\times  {\big |} \delta v^{*,\rho,l,k-1}_{j,m}{\big |}{\big |}\frac{\partial v^{*,\rho,l,k-1}_m}{\partial x_j}+\frac{\partial v^{*,\rho,l,k-2}_m}{\partial x_j}{\big |}  (.,y){\big |}dy{\big |}^2_{L^2\times L^2}.
\end{array}
\end{equation}
Now we do {\it not } have $(1-\phi_1)(.)\frac{\partial}{\partial x_i}K_n(.)\in L^1$, and for this reason we introduced the constant $C^l_{k-1}$ in the for above. This means that we can give another factor $\frac{1}{1+|y|^{2-\alpha}}$ to the convolution. We get
\begin{equation}\label{vrepineq345}
\begin{array}{ll}
|\delta v^{*,\rho,l,k,3}_{(1-B)B i}|^2_{L^2\times L^2}\\
\\
\leq 2n^2(C^B_G)^2\max_{1\leq m,j\leq n}{\Big |}\rho_l\int_{{\mathbb R}^n}{\big |}\int_{{\mathbb R}^n}{\big |}(1- \phi_1(.-y))K_{n,i}(.-y){\big |}\times\\
\\
\times \frac{1}{1+|y|^{1-\alpha}}C^l_{k-1}{\big |}\frac{\partial v^{*,\rho,l,k-1}_m}{\partial x_j}+\frac{\partial v^{*,\rho,l,k-2}_m}{\partial x_j}{\big |}  (.,y){\big |}dy{\big |}^2_{L^2\times L^2},
\end{array}
\end{equation}
now with the constant $C^l_{k-1}$ and some $\alpha\in (0,0.5)$. Now we can apply the Young inequality and it follows that
\begin{equation}\label{vrepineq456}
\begin{array}{ll}
|\delta v^{*,\rho,l,k,3}_{(1-B)B i}|^2_{L^2\times L^2}
\leq \rho_l^24(C^l_{k-1})^2(C^{(1-B)}_G)^2n^2C_{K_3(1-\phi_1)}^2\times\\
\\
\times\max_{j,m\in\left\lbrace 1,\cdots,n\right\rbrace }
{\big |}\frac{\partial \delta v^{*,\rho,l,k-1}_j}{\partial x_m} (\tau,.){\big |}^2_{L^2\times L^2}
\end{array}
\end{equation}
We can use the same argument to get
\begin{equation}\label{vrepineq456}
\begin{array}{ll}
|\delta v^{*,\rho,l,k,3}_{(1-B)(1-B) i}|^2_{L^2\times L^2}
\leq \rho_l^24(C^l_{k-1})^2(C^{(1-B)}_G)^2n^2C_{K_3(1-\phi_1)}^2\times\\
\\
\times\max_{j,m\in\left\lbrace 1,\cdots,n\right\rbrace }
{\big |}\frac{\partial \delta v^{*,\rho,l,k-1}_j}{\partial x_m} (\tau,.){\big |}^2_{L^2\times L^2}
\end{array}
\end{equation}
Summing up and recalling that we have first order derivatives on the right side for some summands we write a $H^1$ norm on the right side (which suffices for our purposes). We have
\begin{equation}\label{vrepineq6e}
\begin{array}{ll}
|\delta v^{*,\rho,l,k}_{i}|^2_{L^{2}\times L^2}
\leq 
\rho_l^24(C^l_{k-1})^2C_G^2n^2\times\\
\\
\times(1+\left( C_{K_3\phi_1}+C_{K_3L^2}\right) ^2(1+C_s^2)) \max_{j\in \left\lbrace 1,\cdots ,n\right\rbrace }
{\big |} \delta v^{*,\rho,l,k-1}_j {\big |}^2_{L^2\times H^1}.
\end{array}
\end{equation}
This close the argument with the assumption (\ref{inductivehyplminus1}) above. 
We note that for the scheme considered in \cite{KB2} and \cite{KNS} we would have an estimate with right side in $L^2$ which is otherwise the same as in (\ref{vrepineq6e}) up to a constant related to the fact that we have one source term less and that the Gaussian estimates for the local fundamental solutions with variable drift term and their adjoints produce different constants.
Finally we show that the assumption 
\begin{equation}
v^{*,\rho,l-1}(l-1,.)\in H^2\cap C^2
\end{equation}
is sufficient in order to get the $L^2\times L^2$-estimate above (up to a constant factor). We do this in a remark, since we do not need these considerations for the controlled global scheme above which is our major objective.
\begin{rem}
Now we may use the fact that a function $u$ is in $L^2$ if $s>\frac{n}{2}$, and for all $x\in {\mathbb R}^n$
\begin{equation}
u(x):=\int(1+|y|^2)^{-s/2}v(x-y)w(y)dy
\end{equation}
for functions $v,w\in L^2$, and such that
\begin{equation}
|u|_{L^2}\leq C_s|v|_{L^2}|w|_{L^2}.
\end{equation}
for a constant $C_s>0$. Indeed, we have 
\begin{equation}
(1-\phi_1)(y)\frac{\partial}{\partial x_i}K_n(y)\in L^2,
\end{equation}
for $n\geq 3$ where we may denote an $L^2$-bound by $K_{3L^2}$, and 
\begin{equation}
y\rightarrow {\big |}\frac{\partial \delta v^{*,\rho,l,k-1}_j}{\partial x_m}{\big |} (\tau,y)\in L^2.
\end{equation}
Hence, we obtain
\begin{equation}\label{vrepineq6}
\begin{array}{ll}
|\delta v^{*,\rho,l,k,3}_{(1-B) i}|^2_{L^{2}\times L^2}=|\delta v^{*,\rho,l,k,3}_{(1-B)B i}|^2_{L^{2}\times L^2}+|\delta v^{*,\rho,l,k,3}_{(1-B)(1-B) i}|^2_{L^{2}\times L^2}\\
\\
\leq \rho_l^24(C^l_{k-1})^2C_G^2C_{K_3L^2}^2n^2C_s^2\max_{j,m\in \left\lbrace 1,\cdots ,n\right\rbrace }
{\big |}\frac{\partial \delta v^{*,\rho,l,k-1}_j}{\partial x_m}(\tau,.){\big |}^2_{L^2}
\end{array}
\end{equation}
\end{rem}

\subsection{$L^2\times H^1$ estimates and $L^2\times H^2$-estimates}
In addition to the $L^2$-estimates we need estimates for the first order partial derivatives of the components of the value function. We start with estimates for norms which include spatial derivatives, i.e., we start with $L^2\times H^1$-estimates. At each stage $k$ of the construction we may differentiate under the integral and start with the pointwise valid expression
\begin{equation}\label{vrepH1}
\begin{array}{ll}
\frac{\partial}{\partial x_j}\delta v^{*,\rho,l,k}_{i}(\tau,x)\\
\\
=:\frac{\partial}{\partial x_j}\delta v^{*,\rho,l,k,1}_{ i}(\tau,x)+\frac{\partial}{\partial x_j}\delta v^{*,\rho,l,k,2}_{\epsilon i}(\tau,x)+\frac{\partial}{\partial x_j}\delta v^{*,\rho,l,k,3}_{ i}(\tau,x)\\
\\
:=\rho_l\int_{l-1}^{\tau}\int_{{\mathbb R}^n}\left( -\sum_{j=1}^n \delta v^{*,\rho,l,k-1}_j\frac{\partial  v^{*,\rho,l,k-1}_i}{\partial x_j}\right)(s,y)G_{l,j}(\tau-s;x-y)dyds\\
\\
+\rho_l\int_{l-1}^{\tau}\int_{{\mathbb R}^n}\left( -\sum_{j=1}^n  v^{*,\rho,l,k-1}_j\frac{\partial  \delta v^{*,\rho,l,k-1}_i}{\partial x_j}\right)(s,y)G_{l,j}(\tau-s;x-y)dyds\\
\\
+\rho_l\int_{l-1}^{\tau}\int_{{\mathbb R}^n}\int_{{\mathbb R}^n}\left( \frac{\partial}{\partial x_i}K_n(z-y)\right)\times\\
\\
\times \left( \sum_{m,j=1}^n \frac{\partial \delta v^{*,\rho,l,k-1}_j}{\partial x_m}\left(\frac{\partial v^{*,\rho,l,k-1}_m}{\partial x_j}+\frac{\partial v^{*,\rho,l,k-2}_m}{\partial x_j}\right)\right)  (s,y)\times \\
\\
\times G_{l,j}(\tau-s,x-z)dydzds,
\end{array}
\end{equation}
where the subscript $_{,j}$ denotes partial derivatives with respect to the $j$th spatial variable (as usual in Einstein notation).
From the representation in (\ref{vrepH1}) we observe that the argument for the $L^2$ estimates can be repeated, if we have a $L^1\times L^1$-bound for the first order partial derivatives of the Gaussian $G_{l}$ (first order partial derivatives with respect to the spatial variables). We have to refine the simple estimate in (\ref{simpleest}) a bit, observing that for $\alpha\in (1.5,2)$ and $n=3$ we have
 \begin{equation}\label{simpleest2}
  \begin{array}{ll}
{\big |}G_{\epsilon,i}(t-s,x-y){\big |}={\big |}\frac{1}{(2\sqrt{\epsilon \pi (t-s)})^n}\frac{-(x_i-y_i)}{2(t-s)}\exp\left(-\frac{(x-y)^2}{4\epsilon (t-s) } \right){\big |}\\
\\
\leq {\big |}(t-s)^{\alpha-1-n/2}\frac{1}{|x-y|^{2\alpha}}\left( \frac{(x-y)^2}{ (t-s)}\right)^{\alpha} \frac{1}{(2\sqrt{\epsilon \pi })^n}\exp\left(-\frac{(x-y)^2}{4\epsilon (t-s) } \right){\big |}\\
\\
\leq {\big |}(t-s)^{\alpha-1-n/2}\frac{1}{|x-y|^{2\alpha-1}}\left( \frac{(x-y)^2}{ (t-s)}\right)^m \frac{1}{(2\sqrt{\epsilon \pi })^n}\exp\left(-\frac{(x-y)^2}{4\epsilon (t-s) } \right){\big |}\\
\\
\leq {\big |}C(t-s)^{\alpha-1-n/2}\frac{1}{|x-y|^{2\alpha-1}} {\big |}.
\end{array}
\end{equation}
It follows that we have local integrability with respect to time, since $\alpha-1-n/2\in (-1,0)$ and in space since $\frac{1}{|x-y|^{2\alpha-1}}$ is locally integrable in dimension $n=3$ for $2\alpha-1\in (2,3)$. Note that we have the same constant $C$ as in the simple estimate (\ref{simpleest}) above. Hence, we have
\begin{equation}
{\big |}\phi_1(x-y)G_{\epsilon,i}(t-s,x-y){\big |}_{L^1\times L^1}\leq C_G^{B1}, 
\end{equation}
for some constant $C_G^{B1}>0$, and it is clear that
\begin{equation}
{\big |}(1-\phi_1)(x-y)G_{\epsilon,i}(t-s,x-y){\big |}_{L^1\times L^1}\leq C_G^{(1-B)1}
\end{equation}
for some constant $C_G^{(1-B)1}>0$.
Hence we may apply the same arguments as for $L^2\times $-estimates where we have to replace the constants $C_G$ for the Gaussian by $C_G^1=C^{B1}_G+C^{1(1-B)}_G$, and get 
\begin{equation}\label{vrepineq6f}
\begin{array}{ll}
|\delta v^{*,\rho,l,k}_{i}|^2_{L^{2}\times H^1}=\sum_{j=1}^3\left( |\delta v^{*,\rho,l,k,j}_{B i}|^2_{L^{2}\times H^1}+|\delta v^{*,\rho,l,k,j}_{(1-B) i}|^2_{L^{2}\times H^1}\right) \\
\\
\leq (n+1)\rho_l^24(C^l_{k-1})^2(C^1_G)^2n^2(1+(C_{K_{3\phi_1}}+K_{3L^1})^2(1+C_s^2))\times \\
\\\times \max_{j\in \left\lbrace 1,\cdots ,n\right\rbrace }
{\big |} \delta v^{*,\rho,l,k-1}_j {\big |}^2_{L^2\times H^1}
\end{array}
\end{equation}
Note that the additional factor $n+1$ takes account of the fact that we need to estimate $n+1=4$ terms with the method of $L^2$-estimates above.
For $L^2\times H^2$-estimates we use convolution rules and partial integration to get the following representation of the second order partial derivatives of thevalue function. We have
\begin{equation}\label{vrepH2second}
\begin{array}{ll}
\frac{\partial^2}{\partial x_m\partial x_j}\delta v^{*,\rho,l,k}_{i}(\tau,x)\\
\\
=:\frac{\partial^2}{\partial x_m\partial x_j}\delta v^{*,\rho,l,k,1}_{ i}(\tau,x)+\frac{\partial^2}{\partial x_m\partial x_j}\delta v^{*,\rho,l,k,2}_{\epsilon i}(\tau,x)+\frac{\partial^2}{\partial x_m\partial x_j}\delta v^{*,\rho,l,k,3}_{ i}(\tau,x)\\
\\
:=\rho_l\int_{l-1}^{\tau}\int_{{\mathbb R}^n}\left( -\sum_{j=1}^n \delta v^{*,\rho,l,k-1}_j\frac{\partial  v^{*,\rho,l,k-1}_i}{\partial x_j}\right)_{,j}(s,y)G_{l,m}(\tau-s;x-y)dyds\\
\\
+\rho_l\int_{l-1}^{\tau}\int_{{\mathbb R}^n}\left( -\sum_{j=1}^n  v^{*,\rho,l,k-1}_j\frac{\partial  \delta v^{*,\rho,l,k-1}_i}{\partial x_j}\right)_{,j}(s,y)G_{l,m}(\tau-s;x-y)dyds\\
\\
+\rho_l\int_{l-1}^{\tau}\int_{{\mathbb R}^n}\int_{{\mathbb R}^n}\left( \frac{\partial}{\partial x_i}K_n(z-y)\right)\times\\
\\
\times \left( \sum_{m,j=1}^n \frac{\partial \delta v^{*,\rho,l,k-1}_j}{\partial x_m}\left(\frac{\partial v^{*,\rho,l,k-1}_m}{\partial x_j}+\frac{\partial v^{*,\rho,l,k-2}_m}{\partial x_j}\right)\right)_{,j} (s,y)\times \\
\\
\times G_{l,m}(\tau-s,x-z)dydzds.
\end{array}
\end{equation}
Here it becomes clear why we included second derivatives in the definition of $C_{k-1}$. Proceeding as before we need to apply the product rule in order to expand the derivatives $_{,j}$ of the value functions above. This gives an additional factor $2$ at $C^l_{k-1}$. Furthermore we have $1+n+n^2$ terms that we have to estimate. Hence,
\begin{equation}\label{vrepineq6g}
\begin{array}{ll}
|\delta v^{*,\rho,l,k}_{i}|^2_{L^{2}\times H^2}=\sum_{j=1}^3\left( |\delta v^{*,\rho,l,k,j}_{B i}|^2_{L^{2}\times H^2}+|\delta v^{*,\rho,l,k,j}_{(1-B) i}|^2_{L^{2}\times H^2}\right) \\
\\
\leq (n^2+n+1)\rho_l^24(C^l_{k-1})^2(C^1_G)^2n^2(1+(C_{K_3\phi_1}+C_{K_3L^1})^2(1+C_s^2))\times \\
\\\times \max_{j\in \left\lbrace 1,\cdots ,n\right\rbrace }
{\big |} \delta v^{*,\rho,l,k-1}_j (\tau,.){\big |}^2_{L^2\times H^2}.
\end{array}
\end{equation}

\subsection{Higher order estimates}
Compared to the considerations in \cite{KB2} the argument simplifies, because we simplified the scheme avoiding the use of the adjoint and estimates of the fundamental solutions for linear equations with variable first order terms. We reconsider this argument in the context of estimates for higher order derivatives.  
In the representation of the functions $v^{*,\rho,k,l}_i,~1\leq i\leq n$ as in (\ref{solrepk}) we have three summands
\begin{equation}\label{solrepkls}
\begin{array}{ll}
\rho_l\int_{l-1}^{\tau}\int_{{\mathbb R}^n}\left( -\sum_{j=1}^n \delta v^{*,\rho,l,k-1}_j\frac{\partial  v^{*,\rho,l,k-1}_i}{\partial x_j}\right)(s,y)G_l(\tau-s;x-y)dyds,\\
\\
\rho_l\int_{l-1}^{\tau}\int_{{\mathbb R}^n}\left( -\sum_{j=1}^n  v^{*,\rho,l,k-1}_j\frac{\partial  \delta v^{*,\rho,l,k-1}_i}{\partial x_j}\right)(s,y)G_l(\tau-s;x-y)dyds,\\
\\
\rho_l\int_{l-1}^{\tau}\int_{{\mathbb R}^n}\int_{{\mathbb R}^n}\left( \frac{\partial}{\partial x_i}K_n(z-y)\right)\times\\
\\
\times \left( \sum_{m,j=1}^n \frac{\partial \delta v^{*,\rho,l,k-1}_j}{\partial x_m}\left(\frac{\partial v^{*,\rho,l,k-1}_m}{\partial x_j}+\frac{\partial v^{*,\rho,l,k-2}_m}{\partial x_j}\right)\right)  (s,y)\times \\
\\
\times G_l(\tau-s,x-z)dydzds.
\end{array}
\end{equation} 
All these summands have products of functions and derivatives of  functions of the form $\delta v^{*,\rho,l,k-1}_j, v^{*,\rho,l,k-1}_j$ known from the previous iteration step. Hence if these functions $\delta v^{*,\rho,l,k-1}_j, v^{*,\rho,l,k-1}_j$, and derivatives of these functions say of order up to $|\alpha|\leq m$ have polynomial decay of order $p$ then products have polynomial decay of order $2p$. This order of polynomial decay may be weakened by the convolution with the Gaussian or by the convolution with the Laplacian kernel in the Leray projection term, but we may expect that for $p>n$ the polynomial decay of order $p$ may be preserved by the scheme in the sense that $\delta v^{*,\rho,l,k}_j, v^{*,\rho,l,k}_j$, and derivatives of these functions say of order up to $|\alpha|\leq m$ have polynomial decay of order $p$. In the next lemma we analyze this. We shall see below that the assumptions of the following lemma are satisfied for the controlled scheme $v^{r,*,\rho,l,k}_i$. 

\begin{lem} Let $\tau\in[l-1,l]$ for some time step number $l\geq 1$. Assume that for $1\leq i\leq n$ the functions $v^{*,\rho,l-1}_i(l-1,.)$ and $\delta v^{*,\rho,l,1}(\tau,.)$ for $\tau\in [l-1,l]$ and derivatives of these functions up to order $m=2$ have polynomial decay of order $m\geq 2$. Then for all $1 \leq n$ and all $k\geq 1$ the functions $v^{*,\rho,k,l}_i,~1\leq i\leq n$ and $v^{*,\rho,k,l}_i,~1\leq i\leq n$ and their derivatives up to order $m\geq 2$ are of polynomial decay of order $m$. Especially, $v^{*,\rho,l,k+1}(\tau,.)\in C^m\cap H^m$, and $\delta v^{*,\rho,l,k+1}(\tau,.)\in C^m\cap H^m$ for all $\tau \in [l-1,l]$.
 \end{lem}
 \begin{proof}
We show that for $m\geq 1$ and $0\leq |\alpha|\leq 2$ we have
 \begin{equation}
 |D^{\alpha}_x\delta v^{*,\rho, l,k+1}_i|\leq \frac{1}{|x|^{m}},\mbox{ if }|x|\geq 1
 \end{equation}
if this holds for $|D^{\alpha}_x\delta v^{*,\rho, l,k}_i|$ and for $D^{\alpha}_xv^{*,\rho,l,k-1}_i$. Similarly for higher order derivatives $D^{\alpha}_xv^{*,\rho, l, k+1}_i$ with $|\alpha|>2$ and some $\rho_l>0$, where $D^{\alpha}_x=D^{\alpha_1}_xD^{\alpha_2}_x\cdots D^{\alpha_n}_x$ denotes the multivariate partial derivative with respect to the multiindex $\alpha=(\alpha_1,\cdots ,\alpha_n)$.
We have to estimate convolutions with the Gaussian $G_{\epsilon}$ for some $\epsilon>0$. The expressions for $\delta v^{*,\rho_l,k+1,l,k+1}_i$ and $v^{\epsilon,k+1}_i$ and their derivatives involve terms which are essentially of the form
 \begin{equation}
 \int_{l-1}^{\tau}\int_{{\mathbb R}^3}h(y)G_{\epsilon}(t-s,x-y)dy ds,
 \end{equation}
 or
\begin{equation}
 \int_{l-1}^{\tau}\int_{{\mathbb R}^3}h(y)G_{\epsilon ,j}(t-s,x-y)dyds,
 \end{equation} 
where $h$ is some function which is a functional of $$D^{\beta}_xv^{*,\rho,l,k}(s,.) \mbox{ and } D^{\beta}_x\delta v^{*,\rho,l,k}(s,.)$$ for $0\leq |\beta|\leq m$, and where the latter functions are in $C^m\cap H^m$ and such that the functions themselves and their derivatives up to order $m$  are assumed to be of polynomial decay of order $m$ according to inductive assumption. Here you may  cf. (\ref{vrepH2second}) for second order spatial derivatives; higher order spatial derivatives can be represented similarly with first order spatial derivatives of the Gaussian. Furthermore we have a Gaussian factor of the form $G_{\epsilon}(t-s,x-y)$ (defined analogously as $G_l$ above), or first order partial derivatives of this factor. We split up the integral of the convolution into two parts where one part is the integral for $|y|\leq \frac{|x|}{2}$. On this domain
  we observe that the Gaussian has polynomial decay of any order. Indeed we have for $0<|t-s|\leq 1$, $m\geq \frac{n}{2}$, and $|y|\leq \frac{|x|}{2}>0$
 \begin{equation}\label{simple}
  \begin{array}{ll}
{\big |}G_{\epsilon}(t-s,x-y){\big |}={\big |}\frac{1}{(2\sqrt{\epsilon \pi (t-s)})^n}\exp\left(-\frac{(x-y)^2}{4\epsilon (t-s) } \right){\big |}\\
\\
={\big |}(t-s)^{m-n/2}\frac{1}{(x-y)^{2m}}\left( \frac{(x-y)^2}{(t-s)}\right)^m \frac{1}{(2\sqrt{\epsilon \pi })^n}\exp\left(-\frac{(x-y)^2}{4\epsilon (t-s) } \right){\big |}\\
\\
\leq {\big |}C(t-s)^{m-n/2}\frac{1}{(x-y)^{2m}} {\big |}\leq \frac{C'}{|x|^{2m}},
\end{array}
\end{equation}
where $C'>0$ is some constant, and where with $z=\frac{(x-y)}{\sqrt{t-s}}$ we define
\begin{equation}
C:=\sup_{z > 0}
{\big |}\frac{1}{(2\sqrt{\epsilon \pi })^n}
\left( z^2\right)^m
\exp\left(-\frac{z^2}{4\epsilon} \right){\big |}>0.
\end{equation}
Furthermore, for
the first order partial derivatives of the Gaussian we have a similar estimate, i.e., we have for $1\leq i\leq n$, and $0<|t-s|\leq 1$, $m> \frac{n}{2}$, and $|y|\leq \frac{|x|}{2}>0$
 \begin{equation}\label{simple2}
  \begin{array}{ll}
{\big |}G_{\epsilon ,i}(t-s,x-y){\big |}={\big |}\frac{1}{(2\sqrt{\epsilon \pi (t-s)})^n}\frac{-(x-y)_i}{2\epsilon (t-s)}\exp\left(-\frac{(x-y)^2}{4\epsilon (t-s) } \right){\big |}\\
\\
={\big |}(t-s)^{m-n/2-1}\frac{|x-y|}{(x-y)^{2m}}\left( \frac{(x-y)^2}{ (t-s)}\right)^m \frac{1}{2\epsilon(2\sqrt{\epsilon \pi })^n}\exp\left(-\frac{(x-y)^2}{4\epsilon (t-s) } \right){\big |}\\
\\
\leq {\big |}(t-s)^{m-n/2-1}\frac{1}{|x-y|^{2m-1}}\left( \frac{(x-y)^2}{ (t-s)}\right)^m \frac{1}{(2\epsilon 2\sqrt{\epsilon \pi })^n}\exp\left(-\frac{(x-y)^2}{4\epsilon (t-s) } \right){\big |}\\
\\
\leq {\big |}C(t-s)^{m-n/2-1}\frac{1}{|x-y|^{2m-1}} {\big |}\leq (t-s)^{m-n/2-1}\frac{C'}{|x|^{2m-1}},
\end{array}
\end{equation}
for some constant $C'>0$, and with a locally  integrable time factor which becomes nonsingular for $m\geq 3$ in the case of dimension $n=3$.
On the complementary domain $|y|>\frac{|x|}{2}$ we need some properties of the integrand $h$.
We observe for $|x|>0$ and some constants $C,C'>0$
\begin{equation}
\begin{array}{ll}
\int_{l-1}^{\tau}\int_{ \left\lbrace |y|\geq \frac{|x|}{2}\right\rbrace }\frac{C}{y^{2p}}G_{\epsilon}(t-s,x-y)dyds\\
\\
\leq \int_{l-1}^{\tau}\int_{ \left\lbrace |y|\geq \frac{|x|}{2}\right\rbrace\&\left\lbrace |x-y|\leq 1\right\rbrace }\frac{C}{y^{2p}}G_{\epsilon}(t-s,x-y)dyds\\
\\
+\int_{l-1}^{\tau}\int_{ \left\lbrace |y|\geq \frac{|x|}{2}\right\rbrace\&\left\lbrace |x-y|> 1\right\rbrace }\frac{C}{y^{2p}}G_{\epsilon}(t-s,x-y)dyds\\
\\
\leq \int_{l-1}^{\tau}\int_{\left\lbrace |y|\geq \frac{|x|}{2}\right\rbrace\& \left\lbrace |x-y|\leq 1\right\rbrace }\frac{C}{y^{2p}}{\big |}C(t-s)^{-1/2}\frac{1}{(x-y)^{2}} {\big |}dyds\\
\\
+\int_{l-1}^{\tau}\int_{\left\lbrace |y|\geq \frac{|x|}{2}\right\rbrace\& \left\lbrace |x-y|> 1\right\rbrace }\frac{C}{y^{2p}}dyds\leq \frac{C'}{|x|^{2p-n}}.
\end{array}
\end{equation}
Note if the functions $v^{*,\rho,l,k}_i, \delta v^{*,\rho,l,k}_i$ have polynomial decay of order $m$ then all terms in the source function $h$ except that Leray projection term have polynomial decay of order $2m$ hence for $m>n$ the preceding observation indicates that the polynomial decay might be preserved. Hence the proof reduces to the observation that the integrand $h$ in the form it has in the representation of $\delta v^{k+1}_i$ and their derivatives is of polynomial decay of order larger than $p+n$. Let us look at arbitrary partial derivatives of some maximal order $m$ which are assumed to be of polynomial decay of order $m$ inductively. Now from (\ref{vrep}) we get for each $1\leq j\leq n$ and $\alpha=(\alpha_1,\cdots,\alpha_i,\cdots,\alpha_n)=:\beta+1_j:=(\alpha_1,\cdots ,\beta_j+1,\cdots,\alpha_n)$  the representation 
\begin{equation}\label{vrepalpha}
\begin{array}{ll}
D^{\alpha}_x\delta v^{*,\rho,l,k}_{ i}(\tau,x)\\
\\
=:D^{\alpha}_x\delta v^{*,\rho,l,k,1}_{ i}(\tau,x)+D^{\alpha}_x\delta v^{*,\rho,l,k,2}_{i}(\tau,x)+D^{\alpha}_x\delta v^{*,\rho,l,k,3}_{ i}(\tau,x)\\
\\
:=\rho_l\int_{l-1}^{\tau}\int_{{\mathbb R}^n}\left( -\sum_{j=1}^n \delta v^{*,\rho,l,k-1}_j\frac{\partial  v^{*,\rho,l,k-1}_i}{\partial x_j}\right)_{,\beta}(s,y)G_{l,j}(\tau-s;x-y)dyds\\
\\
+\rho_l\int_{l-1}^{\tau}\int_{{\mathbb R}^n}\left( -\sum_{j=1}^n  v^{*,\rho,l,k-1}_j\frac{\partial  \delta v^{*,\rho,l,k-1}_i}{\partial x_j}\right)_{,\beta}(s,y)G_{l,j}(\tau-s;x-y)dyds\\
\\
+\rho_l\int_{l-1}^{\tau}\int_{{\mathbb R}^n}\int_{{\mathbb R}^n}\left( \frac{\partial}{\partial x_i}K_n(z-y)\right)\times\\
\\
\times \left( \sum_{m,j=1}^n \frac{\partial \delta v^{*,\rho,l,k-1}_j}{\partial x_m}\left(\frac{\partial v^{*,\rho,l,k-1}_m}{\partial x_j}+\frac{\partial v^{*,\rho,l,k-2}_m}{\partial x_j}\right)\right)_{,\beta} (s,y)\times\\
\\
\times G_{l,j}(\tau-s,x-z)dydzds,
\end{array}
\end{equation}
where the subscript $_{,\beta}$ denotes multivariate partial derivatives with respect to the multiindex $\beta$. Now concerning the first two terms in (\ref{vrepalpha}), i.e., 
$D^{\alpha}_x\delta v^{*,\rho,l,k,1}_{ i}(\tau,x)$ and
$D^{\alpha}_x\delta v^{*,\rho,l,k,2}_{i}(\tau,x)$ the respective integrands
\begin{equation}
\left( -\sum_{j=1}^n \delta v^{*,\rho,l,k-1}_j\frac{\partial  v^{*,\rho,l,k-1}_i}{\partial x_j}\right)_{,\beta},
\end{equation}
\begin{equation}
\left( -\sum_{j=1}^n  v^{*,\rho,l,k-1}_j\frac{\partial  \delta v^{*,\rho,l,k-1}_i}{\partial x_j}\right)_{,\beta}
\end{equation}
are sums of products of functions $D^{\beta}_x\delta v^{*,\rho,l,k-1}_j$ and $D^{\gamma}_x v^{*,\rho,l,k-1}_j$ of order $|\gamma|,|\beta|\leq m$, which are assumed to be of polynomial decay of order $m$. Hence all the integrands except for the integrand related to the Leray projection term are of polynomial decay of order $2m$, i.e., the argument above concerning the estimate for polynomial decay of order $m$ for convolutions with Gaussians shows that $D^{\alpha}_x\delta v^{*,\rho,l,k,1}_{ i}(\tau,x)$ and
$D^{\alpha}_x\delta v^{*,\rho,l,k,2}_{i}(\tau,x)$ are indeed of polynomial decay of order $m$ for $|\alpha|\leq m$.
It remains to check the polynomial decay of order $m$ for the term $D^{\alpha}_x\delta v^{*,\rho,l,k,3}_{ i}(\tau,x)$ for $|\alpha|\leq m$.
Note that for $|y|\leq \frac{|x|}{2}$ we can use the Gaussian polynomial decay as above.
 Hence it is sufficient to estimate integrals of the form
\begin{equation}
\int_{\left\lbrace |y|\geq \frac{|x|}{2}\right\rbrace}
K_{,i}(y-z)g(z)dz
\end{equation}
where $g$ is an integrand of the form
\begin{equation}
\left( \sum_{m,j=1}^n \frac{\partial \delta v^{*,\rho,l,k-1}_j}{\partial x_m}\left(\frac{\partial v^{*,\rho,l,k-1}_m}{\partial x_j}+\frac{\partial v^{*,\rho,l,k-2}_m}{\partial x_j}\right)\right)_{,\beta} 
\end{equation}
with $|\beta|\leq m-1$. Note that $g$ is again a functional determined by sums of products of functions  $D^{\beta}_x\delta v^{*,\rho,l,k-1}_j$ and $D^{\gamma}_x v^{*,\rho,l,k-1}_j$ of order $|\gamma|,|\beta|\leq m$ such that $g$ is of polynomial decay of order $m$ by assumption, i.e.,
we have 
\begin{equation}
|g(z)|\leq \frac{C}{z^{2m}}\mbox{ for } |z|\geq \frac{|x|}{4}.
\end{equation}
We can use a similar argument as above and write
\begin{equation}
\begin{array}{ll}
{\big |}\int_{\left\lbrace |y|\geq \frac{|x|}{2}\right\rbrace}K_{,i}(y-z)g(z)dz{\big |}\\
\\
\leq {\big |}\int_{\left\lbrace |y|\geq \frac{|x|}{2}\right\rbrace \& \left\lbrace |z|\leq \frac{|y|}{2}\right\rbrace}K_{,i}(y-z)g(z)dz{\big |}\\
\\
+{\big |}\int_{\left\lbrace |y|\geq \frac{|x|}{2}\right\rbrace \& \left\lbrace |z|> \frac{|y|}{2}\right\rbrace}K_{,i}(y-z)g(z)dz{\big |}\\
\\
\leq {\big |}\int_{\left\lbrace |y|\geq \frac{|x|}{2}\right\rbrace \& \left\lbrace |z|\leq \frac{|y|}{2}\right\rbrace}\frac{\partial^{2m-2}}{\partial x_i^{2m-2}}K_{,i}(y-z)\frac{C}{z^2}dz{\big |}\\
\\
+{\big |}\int_{\left\lbrace |y|\geq \frac{|x|}{2}\right\rbrace \& \left\lbrace |z|> \frac{|y|}{2}\right\rbrace}K_{,i}(y-z)\frac{C}{y^{2m}}dz{\big |}\in O\left(\frac{C}{|x|^{2m-n}}. \right) 
\end{array}
\end{equation}
This shows that $D^{\alpha}_x\delta v^{*,\rho,l,k,3}_{ i}(\tau,x)$ for $|\alpha|\leq m$ and $m>n$ these functions are all of polynomial decay of order $m$, too.
\end{proof}

Next we consider the higher order estimates. Especially, we need the $H^1\times H^m$ estimate for some $m$ and with $H^1$ with respect to time. Product rules for Sobolev spaces allow us to reduce $H^1\times H^m$-estimates to  $L^2\times H^{m-2}$-estimates. Next the estimates are considered in the case $n=3$. 
First we observe that for all $k\geq 0$ and $m>\frac{5}{2}$ we have
\begin{equation}\label{seriesl1hm}
\frac{\partial v^{\rho,l,k}_i}{\partial \tau}\in L^2\times H^{m-2},
\end{equation}
and more generally for $m>\frac{5}{2}+2p$ we have
\begin{equation}
\frac{\partial^p v^{\rho,l,k}_i}{\partial \tau^p}\in L^2\times H^{m-2p}.
\end{equation}
In this context note that the local-time functions $\frac{\partial^p v^{\rho,l,k}_i}{\partial \tau^p}$ are considered to be trivially extended to the whole time as mentioned above. 
Consider the first equation of (\ref{Navleray}) and the case $k=1$. We have the representation 
\begin{equation}\label{Navlerayvkpr}
 \begin{array}{ll}
\frac{\partial v^{\rho,l,k}_i}{\partial \tau}=\rho_l\nu\sum_{j=1}^n \frac{\partial^2 v^{\rho,l,k}_i}{\partial x_j^2} 
-\rho_l\sum_{j=1}^n v^{\rho,l,k-1}_j\frac{\partial v^{\rho,l,k}_i}{\partial x_j}\\
\\ +\rho_l\sum_{j,m=1}^n\int_{{\mathbb R}^n}\left( \frac{\partial}{\partial x_i}K_n(x-y)\right) \sum_{j,m=1}^n\left( \frac{\partial v^{\rho,l,k-1}_m}{\partial x_j}\frac{\partial v^{\rho,l,k-1}_j}{\partial x_m}\right) (\tau,y)dy.
\end{array}
\end{equation}
Knowing that $v^{\rho,l,k}_i\in L^2\times H^{m}$ for the first term on the right side of (\ref{Navlerayvkpr}) we have 
\begin{equation}
\rho_l\nu\sum_{j=1}^n \frac{\partial^2 v^{\rho,l,k}_i}{\partial x_j^2}\in L^2\times H^{m-2}. 
\end{equation}
Furthermore, for $v^{\rho,l,k}_i\in L^2\times H^{m}$ we observe that for $m>\frac{5}{2}$ concerning the convection term on the right side of (\ref{Navlerayvkpr}) we have for all $\tau\in [l-1,l]$
\begin{equation}
-\rho_l\sum_{j=1}^n v^{\rho,l,k-1}_j(\tau,.)\frac{\partial v^{\rho,l,k}_i}{\partial x_j}(\tau,.)\in H^{m-1}
\end{equation}
since the factors  satisfy $v^{\rho,l,k-1}_j(\tau,.)\in H^m$ and $\frac{\partial v^{\rho,l,k}_i}{\partial x_j}(\tau,.)\in H^{m-1}$, hence the prodcut rule
\begin{equation}\label{prodproof}
|fg|_{H^s}\leq C_s|f|_{H^s}|g|_{H^s} \mbox{ for }f,g\in H^s,~s>\frac{n}{2}
\end{equation}
applies for $m>\frac{5}{2}$ in case $n=3$. From our construction we know that the right side of (\ref{Navlerayvkpr}) is locally continuous with respect to time. Hence, we have
\begin{equation}
 -\rho_l\sum_{j=1}^n v^{\rho,l,k-1}_j\frac{\partial v^{\rho,l,k}_i}{\partial x_j}\in L^2\times H^{m-1}.
\end{equation}
Finally, concerning the Leray projection term in (\ref{Navlerayvkpr}) we observe again that
locally, i.e., for $x-y\in B_r(0)$ for some ball $B_r(0)$ of radius $r>0$ we have
\begin{equation}
K_{,i}\in L^1,
\end{equation}
hence with an appropriate partition of unity, for example with $\phi_1$ defined above we have
\begin{equation}
\phi_1K_{,i}\in L^1,~\mbox{ and }(1-\phi_1)K_{,i}\in L^2
\end{equation}
 for the first order partial derivatives of the Laplacian kernel. For $m>\frac{5}{2}$ in case $n=3$ the product rule (\ref{prodproof}) can be applied to the (relevant part) of the integrand in the Leray projection term, and using the pointwise rule $ab\leq \frac{1}{2}a^2+\frac{1}{2}b^2$ in addition we conclude that
\begin{equation}
\begin{array}{ll}
\sum_{j,m=1}^n\left( \frac{\partial v^{\rho,l,k-1}_m}{\partial x_j}\frac{\partial v^{\rho,l,k-1}_j}{\partial x_m}\right)(\tau,.)\in L^1\cap H^{m-1}.
\end{array}
\end{equation}
Hence for all $\tau\in [l-1,l]$
\begin{equation}\label{NavlerayvkLeray}
 \begin{array}{ll}
\rho_l\sum_{j,m=1}^n\int_{{\mathbb R}^n}\left( \phi_1(.-y)\frac{\partial}{\partial x_i}K_n(.-y)\right) \sum_{j,m=1}^n\left( \frac{\partial v^{\rho,l,k-1}_m}{\partial x_j}\frac{\partial v^{\rho,l,k-1}_j}{\partial x_m}\right) (\tau,y)dy\\
\\
+\rho_l\sum_{j,m=1}^n\int_{{\mathbb R}^n}\left( (1-\phi_1(.-y))\frac{\partial}{\partial x_i}K_n(.-y)\right)\times\\
\\
\times \sum_{j,m=1}^n\left( \frac{\partial v^{\rho,l,k-1}_m}{\partial x_j}\frac{\partial v^{\rho,l,k-1}_j}{\partial x_m}\right) (\tau,y)dy \in H^{m-1},
\end{array}
\end{equation}
applying the product rule and different appropriate types of Young's inequality to both summands (cf. also part II of this investigation). Again continuity with respect to time leads to
\begin{equation}
\begin{array}{ll}
\rho_l\sum_{j,m=1}^n\int_{{\mathbb R}^n}\frac{\partial}{\partial x_i}K_n(.-y)\times \\
\\
\times \sum_{j,m=1}^n\left( \frac{\partial v^{\rho,l,k-1}_m}{\partial x_j}\frac{\partial v^{\rho,l,k-1}_j}{\partial x_m}\right) (.,.-y)dy\in L^2\times H^{m-1}.
\end{array}
\end{equation}
Hence, we have 
\begin{equation}
\frac{\partial \delta v^{*,\rho,l,k}_i}{\partial \tau}\in L^2\times H^{m-2}
\end{equation}
for all $k\geq 0$.
The next step is to show that we have a contraction for some $\rho_l>0$.
We observe
\begin{equation}\label{NavlerayIIa}
\begin{array}{ll}
{\Big |}\frac{\partial \delta v^{*,\rho,l,k}_i}{\partial \tau}{\Big |}_{L^2\times H^{m-2}}\leq {\Big |}\frac{\partial \delta v^{*,\rho,l,k}_i}{\partial \tau}{\Big |}_{L^2\times H^{m}}\leq \rho_l\nu\sum_{j=1}^n {\Big |}\frac{\partial^2 \delta v^{*,\rho,l,k}_i}{\partial x_j^2}{\Big |}_{L^2\times H^{m}}\\
\\ 
+\rho_l\sum_{j=1}^n {\Big |}v^{*,\rho,l,k-1}_j\frac{\partial \delta v^{*,\rho,l,k-1}_i}{\partial x_j}{\Big |}_{L^2\times H^{m}}+\rho_l\sum_{j=1}^n {\Big |}\delta v^{*,\rho,l,k-1}_j\frac{\partial  v^{*,\rho,l,k-1}_i}{\partial x_j}{\Big |}_{L^2\times H^{m}}
\\
\\
+\rho_l{\Big |}\int_{l-1}^{\tau}\int_{{\mathbb R}^n}\int_{{\mathbb R}^n}\left( \frac{\partial}{\partial x_i}K_n(z-y)\right)\times\\
\\
\times \left( \sum_{m,j=1}^n \frac{\partial \delta v^{*,\rho,l,k-1}_j}{\partial x_m}\left(\frac{\partial v^{*,\rho,l,k-1}_m}{\partial x_j}+\frac{\partial v^{*,\rho,l,k-2}_m}{\partial x_j}\right)\right)  (s,y){\Big |}_{L^2\times H^m}
\end{array}
\end{equation}
We can estimate the right side of (\ref{NavlerayIIa}) and for $m>\frac{5}{2}$ we have the upper bound
\begin{equation}\label{ss}
\begin{array}{ll}
\rho_l\nu n\max_{j\in \left\lbrace 1,\cdots ,n\right\rbrace}{\Big |} \delta v^{*,\rho,l,k}_i{\Big |}_{L^2\times H^{m-2}}\\
\\ 
+\rho_lC_{5/2}n\max_{j\in \left\lbrace 1,\cdots ,n\right\rbrace}  {\Big |}v^{*,\rho,l,k-1}_j{\Big |}_{L^2\times H^{m-1}}{\Big |}\frac{\partial \delta v^{*,\rho,l,k-1}_i}{\partial x_j}{\Big |}_{L^2\times H^{m-1}}\\
\\
+\rho_lC_{5/2}n\max_{j\in \left\lbrace 1,\cdots ,n\right\rbrace} {\Big |}\delta v^{*,\rho,l,k-1}_j{\Big |}_{L^2\times H^{m-1}}{\Big |}\frac{\partial  v^{*,\rho,l,k-1}_i}{\partial x_j}{\Big |}_{L^2\times H^{m-1}}
\\
\\
+\rho_lC_{5/2}C_Kn^2\max_{j,m\in \left\lbrace 1,\cdots ,n\right\rbrace} 
{\Big |}\frac{\partial \delta v^{*,\rho,l,k-1}_j}{\partial x_m}{\Big |}_{L^2\times H^{m-1}}\times\\
\\
\times {\Big |}\frac{\partial v^{*,\rho,l,k-1}_m}{\partial x_j}
+\frac{\partial v^{*,\rho,l,k-2}_m}{\partial x_j}{\Big |}_{L^2\times H^{m-1}}\\
\end{array}
\end{equation}
From (\ref{NavlerayIIa}) and (\ref{ss}) we get
\begin{equation}\label{sss}
\begin{array}{ll}
{\Big |}\frac{\partial \delta v^{*,\rho,l,k}_i}{\partial \tau}{\Big |}_{L^2\times H^{m-2}}\leq \rho_l\nu n\max_{j\in \left\lbrace 1,\cdots ,n\right\rbrace}{\Big |} \delta v^{*,\rho,l,k}_i{\Big |}_{L^2\times H^{m-2}}\\
\\ 
+\rho_lC_{5/2}n\max_{j\in \left\lbrace 1,\cdots ,n\right\rbrace}  {\Big |}v^{*,\rho,l,k-1}_j{\Big |}_{L^2\times H^{m-2}}{\Big |} \delta v^{*,\rho,l,k-1}_i{\Big |}_{L^2\times H^{m-2}}\\
\\
+\rho_lC_{5/2}n\max_{j\in \left\lbrace 1,\cdots ,n\right\rbrace} {\Big |}\delta v^{*,\rho,l,k-1}_j{\Big |}_{L^2\times H^{m-2}}{\Big |}  v^{*,\rho,l,k-1}_i{\Big |}_{L^2\times H^{m-2}}
\\
\\
+\rho_lC_{5/2}C_Kn^2\max_{j,m\in \left\lbrace 1,\cdots ,n\right\rbrace} 
{\Big |}\frac{\partial \delta v^{*,\rho,l,k-1}_j}{\partial x_m}{\Big |}_{L^2\times H^{m-2}}\times\\
\\
\times \left( {\Big |}v^{*,\rho,l,k-1}_m{\Big |}_{L^2\times H^{m-2}}
+{\Big |}v^{*,\rho,l,k-2}_m{\Big |}_{L^2\times H^{m-2}}\right) \\
\end{array}
\end{equation}
Now form previous estimates we have the upper bound
\begin{equation}
{\Big |}v^{*,\rho,l,k-1}_m{\Big |}_{L^2\times H^{m-2}}
+{\Big |}v^{*,\rho,l,k-2}_m{\Big |}_{L^2\times H^{m-2}}\leq 2C^l_{k-1}
\end{equation}
for some constant $C^l_{k-1}>0$, hence
\begin{equation}\label{sssa}
\begin{array}{ll}
{\Big |}\frac{\partial \delta v^{*,\rho,l,k}_i}{\partial \tau}{\Big |}_{L^2\times H^{m-2}}\leq \rho_l\nu n\max_{j\in \left\lbrace 1,\cdots ,n\right\rbrace}{\Big |} \delta v^{*,\rho,l,k}_i{\Big |}_{L^2\times H^{m-2}}\\
\\ 
+\rho_lC_{5/2}(2n+n^2)\max_{j\in \left\lbrace 1,\cdots ,n\right\rbrace}(1+C_k)C^l_{k-1}{\Big |} \delta v^{*,\rho,l,k-1}_i{\Big |}_{L^2\times H^{m-2}}
\end{array}
\end{equation}
Furthermore, from our previous estimates and for some $\rho_l^0$ we have contraction for the first term on the right side of (\ref{sssa}). Hence for some $\rho_l>0$ independent of $k$ we get 
\begin{equation}\label{sssa}
\begin{array}{ll}
{\Big |}\frac{\partial \delta v^{*,\rho,l,k}_i}{\partial \tau}{\Big |}_{L^2\times H^{m-2}}\leq \frac{1}{4}\max_{j\in \left\lbrace 1,\cdots ,n\right\rbrace}(1+C_k)C_{k-1} \delta v^{*,\rho,l,k-1}_i{\Big |}_{L^2\times H^{m-2}}
\end{array}
\end{equation}
Similar construction estimates can be obtained by analogous methods successively for higher mixed derivatives
the functions
\begin{equation}
D^{\alpha}_x\frac{\partial \delta v^{*,\rho,l,k}_i}{\partial \tau}
\end{equation}
(for multiindex $\alpha$), then for
\begin{equation}
\frac{\partial^2 \delta v^{*,\rho,l,k}_i}{\partial \tau^2},
\end{equation}
and then successively for higher order mixed and higher order time derivatives. 
Finally, let us make a remark concerning $C^m\left((l-1,l),H^{m}\right) $ contraction estimates. Here, we may use again the generalized Young inequality, i.e., the fact that for $1\leq p,q,r\leq \infty$ 
\begin{equation}\label{Y1}
 f\in L^q~\mbox {and}~ g\in L^p~\rightarrow  f\ast g\in L^r,\mbox{if}~\frac{1}{p}+\frac{1}{q}=1+\frac{1}{r},
\end{equation}
and
\begin{equation}\label{Y2}
|f\ast g|_{L^r}\leq |f|_{L^p} |g|_{L^q}
\end{equation}
for  a convolution $f\ast g$. We may treat time and space differently as we observed in \cite{KB2}. Applying Fubbini's theorem we may fix time first and treat the spatial variables as before with the Gaussian in $L^1$ and the recursively defined terms involving the approximatively defined value functions and their derivatives in $L^2$. For the time variables we then may use the Young inequality with $r=\infty$ and where the requirement $\frac{1}{p}+\frac{1}{q}=1$ gives even more flexibility in order to deal with the convolution estimates. An explicit treatment for these estimates can be found in \cite{KHyp}, and - for a slightly different scheme in \cite{KB2}, where the argument there can be easily adapted to the scheme used in this paper.
 
\section{Global regular linear upper bounds and a proof of theorem \ref{uncontrolledthm1} and theorem \ref{uncontrolledthm2} }
Let us summarize first: we construct the velocity function components of the Navier Stokes equation, i.e., functions $v_i$ which solve the equation \ref{Navlerayorg}, by iterated local controlled schemes $v^{r,*,\rho,l,k}_i,~1\leq i\leq n,~k\geq 1,~l\geq 1$, or by a similar scheme $v^{r,\rho,l,k}_i,~1\leq i\leq n,~k\geq 1,~l\geq 1$. At each time step a control function $r^l_i$ is determined in addition, such that the information of $v_i$ can be restored by the local functions $v^{*,\rho,l}_i:[l-1-l]\times {\mathbb R}^n\rightarrow {\mathbb R},~l\geq 1$, which are determined by
\begin{equation}
v^{*,\rho,l}_i=v^{r,*,\rho,l}_i-r^l_i.
\end{equation}
The local limit functions are determined by local limits
\begin{equation}
v^{r,*,\rho,l}_i=\lim_{k\uparrow \infty}v^{r,*,\rho,l,k}=v^{r^{l-1},*,\rho,l,k}_i+\delta r^{l-1}_i,
\end{equation}
where the function $v^{r^{l-1},*,\rho,l,k}_i,~1\leq i\leq n$ solves an uncontrolled Navier Stokes equation with controlled data, i.e., it solves the equation system
\begin{equation}\label{Navlerayconuncon1}
\left\lbrace \begin{array}{ll}
\frac{\partial v^{r^{l-1},*\rho,l,k}_i}{\partial \tau}-\rho_l\nu\sum_{j=1}^n \frac{\partial^2 v^{r^{l-1},*,\rho,l,k}_i}{\partial x_j^2} 
+\rho_l\sum_{j=1}^n v^{r^{l-1},*,\rho,l,k-1}_j\frac{\partial v^{r^{l-1},*,\rho,l,k-1}_i}{\partial x_j}=\\
\\ \rho_l\sum_{j,m=1}^n\int_{{\mathbb R}^n}\left( \frac{\partial}{\partial x_i}K_n(x-y)\right) \sum_{j,m=1}^n\left( \frac{\partial v^{r^{l-1},*,\rho,l,k-1}_m}{\partial x_j}\frac{\partial v^{r^{l-1},*,\rho,l,k-1}_j}{\partial x_m}\right) (\tau,y)dy,\\
\\
\mathbf{v}^{r^{l-1},*,\rho,l,k}(l-1,.)=\mathbf{v}^{r,*,\rho,l-1}(l-1,.),
\end{array}\right.
\end{equation}

Note that in the scheme with upper script $*$ the information about the convection term at local iteration step $k\geq 1$ is from the previous local iteration step $k-1$. This allows us to consider classical representations of the approximative local solutions  $v^{r^{l-1},*,\rho,l,k}_i$ in terms of convolutions with a fundamental solution of Gaussian type, which simplifies the argument of local contraction in the previous sections. Note that (\ref{Navlerayconuncon1}) is a time-local uncontrolled Navier Stokes equation with controlled data $\mathbf{v}^{r,*,\rho,l-1}(l-1,.)$. We get the controlled velocity functions at time step $l\geq 1$ by adding a control function increment $\delta r^l_i$, where we have some freedom of choice. 
The local limits
\begin{equation}
v^{r,*,\rho,l}_i=\lim_{k\uparrow \infty}v^{r^{l-1},*,\rho,l,k}_i+\delta r^l_i=\lim_{k\uparrow \infty}v^{r,*\rho,l,k}_i
\end{equation}
are the same as for the local scheme $v^{r,\rho,k,l}_i=v^{r^{l-1},\rho,l,k}_i+\delta r^l_i$, i.e., we have
\begin{equation}\label{note1}
v^{r,*,\rho,l}_i=\lim_{k\uparrow \infty}v^{r^{l-1},*,\rho,l,k}_i+\delta r^l_i=v^{r,\rho,l}_i.
\end{equation}
\begin{rem}
Recall that in this alternative scheme $v^{r^{l-1},\rho,l,k}_i$ solves
\begin{equation}\label{Navlerayconuncon2}
\left\lbrace \begin{array}{ll}
\frac{\partial v^{r^{l-1},\rho,l,k}_i}{\partial \tau}-\rho_l\nu\sum_{j=1}^n \frac{\partial^2 v^{r^{l-1},\rho,l,k}_i}{\partial x_j^2} 
+\rho_l\sum_{j=1}^n v^{r^{l-1}\rho,l,k-1}_j\frac{\partial v^{r^{l-1},\rho,l,k}_i}{\partial x_j}=\\
\\ \rho_l\sum_{j,m=1}^n\int_{{\mathbb R}^n}\left( \frac{\partial}{\partial x_i}K_n(x-y)\right) \sum_{j,m=1}^n\left( \frac{\partial v^{r^{l-1},\rho,l,k-1}_m}{\partial x_j}\frac{\partial v^{r^{l-1},\rho,l,k-1}_j}{\partial x_m}\right) (\tau,y)dy,\\
\\
\mathbf{v}^{r^{l-1},\rho,l,k}(l-1,.)=\mathbf{v}^{r,\rho,l-1}(l-1,.),
\end{array}\right.
\end{equation}
We considered local contraction for equations of the form (\ref{Navlerayconuncon2}) in \cite{KB2}. This is slightly more complicated as the local contraction argument of this paper as we use the adjoint of the fundamental solution. However, this adjoint can also be used in the case of variable coefficients, and this shows that the conclusions for global existence can be extended to models with variable viscosity.
\end{rem}
For the global conclusions it does not matter which local iteration scheme we use since they lead all to the same local solution function in regular function spaces $C^1\left( \left[l-1,l\right],H^m\cap C^m \right)$ for $m\geq 2$ by local contraction. Especially, by (\ref{note1}) we have
\begin{equation}
v^{r^{l-1},\rho,l}_i=v^{r^{l-1},*,\rho,l}_i,
\end{equation}
and we may use the functions with upper script $*$ and without upper script $*$ interchangeably as long as no local iteration index $k$ occurs.

The index $l\geq 1$ is a time step number index. At each time step $l\geq 1$ we assume that the function $v^{r,*,\rho,l-1}_i(l-1,.),~1\leq i\leq n$ is determined. These are the data for a time-local Cauchy problem for an uncontrolled Navier Stokes equation with controlled data and solution $v^{r^{l-1},*,\rho,l}_i$, which solves
\begin{equation}\label{Navlerayconunconlimit}
\left\lbrace \begin{array}{ll}
\frac{\partial v^{r^{l-1},*\rho,l}_i}{\partial \tau}-\rho_l\nu\sum_{j=1}^n \frac{\partial^2 v^{r^{l-1},*,\rho,l,k}_i}{\partial x_j^2} 
+\rho_l\sum_{j=1}^n v^{r^{l-1},*,\rho,l}_j\frac{\partial v^{r^{l-1},*,\rho,l}_i}{\partial x_j}=\\
\\ \rho_l\sum_{j,m=1}^n\int_{{\mathbb R}^n}\left( \frac{\partial}{\partial x_i}K_n(x-y)\right) \sum_{j,m=1}^n\left( \frac{\partial v^{r^{l-1},*,\rho,l}_m}{\partial x_j}\frac{\partial v^{r^{l-1},*,\rho,l}_j}{\partial x_m}\right) (\tau,y)dy,\\
\\
\mathbf{v}^{r^{l-1},*,\rho,l}(l-1,.)=\mathbf{v}^{r,*,\rho,l-1}(l-1,.).
\end{array}\right.
\end{equation}
The solution $v^{r.*,\rho,l}_i,~1\leq i\leq n$ of the latter equation is obtained by a local iteration scheme $v^{r.*,\rho,l,k}_i,~1\leq i\leq n,~k\geq 1$, where the latter functions solve (\ref{Navlerayconuncon1}). The local contraction results of section 9 ensure that there exists such a unique local solution in a regular function space $C^1\left(\left[l-1,l\right],H^m\cap C^m \right)$ for $m\geq 2$. The controlled velocity function at time step $l\geq 1$ is then determined as
\begin{equation}
v^{r,\rho,l}_i=v^{r^{l-1},\rho,l}_i+\delta r^l_i
\end{equation}
where we have some freedom to choose the control function increments $\delta r^l_i,~1\leq i\leq n$. We want to choose them such that for given $m\geq 2$ and data $h_i,~1\leq i\leq n$ we have a time step size $\rho_l$ (constant or at least $\frac{1}{l}\precsim \rho_l$) such that for some constant $C>0$ with $\max_{1\leq i\leq n}{\big |}h_i{\big |}_{H^m\cap C^m}\leq C$ we have
\begin{equation}
{\big |}v^{r,,\rho,l}_i(l-1,.){\big |}_{H^m\cap C^m}\leq C\Rightarrow {\big |}v^{r,,\rho,l}_i(l,.){\big |}_{H^m\cap C^m}\leq C
\end{equation}
inductively for all $l\geq 1$ for the controlled scheme, where the growth of the control function can be bounded linearly in time. This can be achieved if the volume of the increment in $[l-1,l]\times H^m\cap C^m$ is proportional to $1$ independently of the time step number.  A possibility to achieve this is the choice of the simple control function
\begin{equation}\label{lbincrementcontrol}
\begin{array}{ll}
\delta r^l_i(\tau,x)=\int_{l-1}^{\tau}\left( -\frac{v_i^{r,*,\rho,l-1}}{C}(l-1,.)\right) G_l(\tau-s,x-y)dyds
\end{array}
\end{equation}
at each time step $l\geq 1$. For small time step size $\rho_l\sim \frac{1}{C^3}$ the diffusion effect of the Gaussian becomes small and the control function increment
\begin{equation}\label{deltaobserv1}
\delta r^l_i(l,x)=\int_{l-1}^{l}\left( -\frac{v_i^{r,*,\rho,l-1}}{C}(l-1,.)\right) G_l(\tau-s,x-y)dyds.
\end{equation}
is close to
\begin{equation}\label{deltaobserv2}
-\frac{v_i^{r,*,\rho,l-1}}{C}(l-1,.).
\end{equation}
Note that a control function as in (\ref{lbincrementcontrol}) has the effect of a decrease of norms over one time step if the time step size $\rho_l$ becomes small.
 As data $v_i^{r,*,\rho,l-1}$ become large in some norm this term dominates the growth of $\delta v^{r,*,\rho,l}_i(l,.)$ which depends on the time size factor $\rho_l$, i.e., we may choose a step size $\rho_l$ such that
\begin{equation}\label{vincrement}
\begin{array}{ll}
|v^{r,*,\rho,l}_i(l,.)-v^{r,\rho,l-1}_i(l-1,.)|_{H^m\cap C^m}\\
\\
=|v^{r,*,\rho,l}_i(l,.)-v^{r^{l-1},\rho,l-1}_i(l-1,.)|_{H^m\cap C^m}\\
\\
={\big |}\sum_{k=1}^{\infty}\delta v^{r^{l-1},\rho,*,l,k}_i(l,.){\big |}_{H^m\cap C^m}\leq \frac{1}{4}+\frac{1}{4}=\frac{1}{2}.
\end{array}
\end{equation}
This is a consequence of our local contraction results. Let us consider the situation where we have an upper bound for the controlled value functions $v^{r,\rho,l-1}_i(l-1,.), 1\leq i\leq n$ inductively from the previous time step of the form
\begin{equation}
|v_i^{r,\rho,l-1}(l-1,.)|_{H^m\cap C^m}\geq C
\end{equation}
for some $m\geq 2$.
Now, assuming that $C\geq 4$ w.l.o.g., and assuming that data become large (otherwise we get the persistence of the upper bound at time step $l1$ for free), i.e., assuming that
\begin{equation}
|v_i^{r,\rho,l-1}(l-1,.)|_{H^m\cap C^m}\geq C-\frac{1}{2},
\end{equation}
as $\rho_l>0$ becomes small the statement that (\ref{deltaobserv1}) is close to (\ref{deltaobserv2}) implies that
\begin{equation}\label{controlincrement}
|\delta r^l_i(l,.)|_{H^m\cap C^m}\geq  {\big|}\frac{v_i^{r,\rho,l-1}}{C}(l-1,.){\big |}_{H^m\cap C^m}-\frac{1}{8}\geq \frac{3}{4},
\end{equation}
such that (\ref{controlincrement}) dominates (\ref{vincrement}). It follows that for all $x\in {\mathbb R}^n$ that
\begin{equation}
|v^{r,*,\rho,l}_i(l,.)|_{H^m\cap C^m}\leq |v^{r,\rho,l-1}_i(l-1,)|_{H^m\cap C^m}.
\end{equation}
This reasoning holds for all $1\leq i\leq n$ such that the controlled velocity functions  $v^{r,*,\rho,l}_i$ have a uniform upper bound $C$. Furthermore, as the controlled velocity functions have the upper bound $C$, the corresponding control function increments satisfy
\begin{equation}
|\delta r^l_i(l,x)|_{H^m\cap C^m}\leq c,
\end{equation}
for some constant $c\geq 1$ (which is independent of the time step number $l$) such that for $l\geq 1$ we have
\begin{equation}
|r^l_i(l,x)|\leq |r^0_i(0,.)|+\sum_{p=1}^{l}|\delta r^l_i(l,x)|\leq C+cl,
\end{equation}
where $C>0$ is an upper bound of the modulus of the initial data of the $i$th control function $|r^0_i(0,.)|$, i.e., the control function has linear growth (at most). Note that without smoothing (convolution with a density) in (\ref{lbincrementcontrol}) the constant $c$ can be chosen to be $c=1$. Note that this is also a possible choice for the control function increments. 

We have obtained global regular upper bounds in $C^1\left( [0,T],H^m\cap C^m\right)$ for $m\geq 2$ and, hence a global regular solution in the same function space for the controlled velocity function $\mathbf{v}^r=\left(v^r_1,\cdots,v^r_n \right)^T$ which equals $v^{r,l}_i(t,.):=v^{r,*,\rho,l}_i(\tau,.)$ for $t\in \left[ \sum_{m=1}^{l-1}\rho_m,\sum_{m=1}^{l}\rho_m\right]$ and $\tau\in [l-1,l]$ respectively via a time recursive scheme. In the construction of the scheme we obtained a global regular upper bound for the control function $\mathbf{r}=\left(r_1,\cdots,r_n \right)^T$ where each component $r_i(t,.)$ equals $r^l_i(\tau,.)$ on $[l-1,l]\times {\mathbb R}^n$ in transformed time coordinates $\tau$ (recall that $\rho_l\tau=t$ at each time step). 

The global regular solution $\mathbf{v}^r=\mathbf{v}+\mathbf{r}$ 
of (\ref{Navleray}) in the function space $C^1\left( [0,T],H^m\cap C^m\right)$  and arbitrary $T>0$ with global regular upper bounds for the controlled velocity function components $v^r_i,~1\leq i\leq n$ and for the control function components $r_i$ in the same function space imply that the solution function $\mathbf{v}=(v_1,\cdots,v_n)^T$ of the incompressible Navier Stokes equation has the property
\begin{equation}
v_i=v^r_i-r_i\in C^1\left( [0,T],H^m\cap C^m\right)
\end{equation}
for $m\geq 2$ and all $1\leq i\leq n$ and arbitrary $T>0$.
Note that we can compute the solution in the scheme via $v^{\rho,*,l}_i$ with
\begin{equation}
v^{\rho,*,l}_i=v^{r,\rho,l}_i-r^l_i
\end{equation}
time step by time step. Note that all the reasoning above is essentially pointwise and we could have replaced the ${|}.|_{H^m\cap C^m}$-norm by a simple modulus and then observe that  a similar reasoning holds for the time growth of the spatial derivatives such that for appropriate $\rho_l=\rho>0$ we have a global upper bound (generic) $C>0$ of the modulus of  $D^\alpha_xv^{r,*,\rho,l}_i$ (independent of the time step number $l\geq 1$) and a linear upper bound of the modulus of the control functions $D^{\alpha}_x r^l_i$ (with respect to the time step number $l\geq 1$). The existence of higher order time derivatives for smooth data follows straightforwardly.Global regular existence follows from this observation and the local existence result. This proves theorem \ref{uncontrolledthm1}. As the upper bounds for the controlled velocity functions are uniform with respect to time preserved we get a global scheme with a constant time step size $\rho_l=\rho\sim \frac{1}{C^3}$ which is independent of the time step number $l\geq 1$. The choice of a decreasing time step size $\rho_l\sim \frac{1}{l}$ improves the estimate of the upper bound of the control function to be logarithmic in time. This implies growth of the value function $v^{\rho,l}_i$ of the form ${\big |}v^{\rho,l}_i(l,.){\Big |}_{H^m\cap C^m}\leq C+C\sqrt{l}$ for $m\geq 2$, and this impelies in turn that we get a linear global regular upper bound for the Leray projection term as stated in theorem \ref{uncontrolledthm1}.

Next we want give an alternative proof and show that a simplified control function without consumption source terms leads to a global linear upper bound for the Leray projection term and therefore to a global scheme. We refer also to (\cite{KHyp}). Assume inductively (with respect to the time step number $l\geq 1$ that we have realized the bound 
\begin{equation}
D^{\alpha}_xv^{r,*,\rho,l-1}_i(l-1,.)\sim \sqrt{l-1} \mbox{ for }|\alpha|\leq m
\end{equation}
for some $m\geq 2$ which is fixed in advance. First note that the local contraction result
\begin{equation}\label{loccontrupp1}
{\big |}\delta v^{*,\rho,l,k}_i{\big |}_{C^m\left((l-1,l), H^{2m}\right) }\leq \frac{1}{2}{\big |}\delta v^{*,\rho,l,k-1}_i{\big |}_{C^m\left((l-1,l), H^{2m}\right)}
\end{equation}
easily extends to 
\begin{equation}\label{loccontrupp2}
{\big |}\delta v^{r,*,\rho,l,k}_i{\big |}_{C^m\left((l-1,l), H^{2m}\right)}\leq \frac{1}{2}{\big |}\delta v^{r,*,\rho,l,k-1}_i{\big |}_{C^m\left((l-1,l), H^{2m}\right)}
\end{equation}
for all $1\leq i\leq n$ a bit. The reason is that it extends to $\delta v^{r^{l-1},*,\rho,l,k}_i$ first (same prove with different initial data) and then the additional control increments $\delta r^{l}_i$ cancel in contraction estimates of the controlled function increments $\delta v^{r,*,\rho,l}_i$ such that we have (\ref{loccontrupp2}).
We  may refine the latter local contraction result for all $1\leq i\leq n$ a bit. 

In the form (\ref{loccontrupp1}) it just ensures that the local limit 
\begin{equation}\label{funciv2}
\mathbf{v}^{*,\rho ,l }=\mathbf{v}^{*,\rho,l-1}+\sum_{k=1}^{\infty} \delta \mathbf{v}^{\rho,l,k}
=\mathbf{v}^{*,\rho,l,1}+\sum_{k=2}^{\infty} \delta \mathbf{v}^{*,\rho,l,k}
\end{equation} 
of the corresponding local functional series represents a local solution of the incompressible Navier Stokes equation on the domain $[l-1,l]\times {\mathbb R}^n$. 
Now consider the first increment $\delta v^{r,*,\rho,l,1}_i$ for $1\leq i\leq n$.  
In the form (\ref{loccontrupp2}) it just ensures that the local limit of the controlled scheme exists, i.e. 
\begin{equation}\label{funciv22}
\mathbf{v}^{r,*,\rho ,l }=\mathbf{v}^{r,*,\rho,l-1}+\sum_{k=1}^{\infty} \delta \mathbf{v}^{r^{l-1},\rho,l,k}+\delta r^l_i
=:\mathbf{v}^{r,*,\rho,l,1}+\sum_{k=2}^{\infty} \delta \mathbf{v}^{*,\rho,l,k}
\end{equation} 
exists. Note again that we used the simplified notation, i.e., $\delta \mathbf{v}^{*,\rho,l,k}$ is synonymous with  $\delta \mathbf{v}^{r^{l-1},\rho,l,k}$ in (\ref{funciv22}) and should not be confused with $\delta \mathbf{v}^{\rho,l,k}$ in (\ref{funciv2})- they just satisfy equation of analogous structure where in case of the controlled scheme the initial data $v^{r,*,\rho,l-1}_i(l-1,.)$ and in the other case the initial data $v^{*,\rho,l-1}_i$ are involved. 
For simplicity of notation let us consider the uncontrolled scheme again, i.e., consider the data $v^{*,\rho,l-1}_i(l-1,.)$ and the local scheme without control function. The reasoning transfers to the controlled scheme straightforwardly as we shall observe below. Consider the first increment $\delta v^{*,\rho,l,1}_i$ for $1\leq i\leq n$. 

For the classical Navier Stokes equation with constant viscosity the functions
$v^{*,\rho,1,l}_i$ have the classical representation
 \begin{equation}\label{scalparasystlin10v*}
 \begin{array}{ll}
v^{*,\rho,1,l}_i(\tau,x)=\int_{{\mathbb R}^n}v^{*,\rho,l-1}_i(l-1,y)G_l(\tau,x-y)dy\\
\\ 
-\rho_l\int_{l-1}^{\tau}\int_{{\mathbb R}^n}\sum_{j=1}^n v^{*,\rho,l-1}_j(s,y)\frac{\partial v^{*,\rho,l-1}_i}{\partial x_j}(s,y)G_l(\tau-s,x-y)dyds\\
\\
+\rho_l\int_{l-1}^{\tau}\int_{{\mathbb R}^n}\int_{{\mathbb R}^n}\sum_{j,m=1}^n \left( \frac{\partial v^{*,\rho,l-1}_j}{\partial x_m}\frac{\partial v^{*,\rho,l-1}_m}{\partial x_j}\right) (l-1,y)\frac{\partial}{\partial x_i}K_n(z-y)\times\\
\\
\times G_l(\tau-s,x-z)dydzds.
\end{array}
\end{equation} 
For a time step size $\rho_l$ of order
\begin{equation}
\rho_l\sim \frac{1}{l},
\end{equation}
we observe that
\begin{equation}
D^{\alpha}_x\delta v^{*,\rho,l,1}_i=D^{\alpha}_xv^{*,\rho,l,1}_i-D^{\alpha}_xv^{*,\rho,l-1}_i(l-1,.)\sim 1
\end{equation}
for the following reasons. From (\ref{scalparasystlin10v*}) we observe that we have to estimate the term
\begin{equation}\label{firstterm}
\int_{{\mathbb R}^n}v^{*,\rho,l-1}_i(l-1,y)G_l(\tau,x-y)dy-v^{*,\rho,l-1}_i(l-1,x),
\end{equation}
together with two term which have the coefficient $\rho_l$. Since the term in (\ref{firstterm}) corresponds to a function which is the solution of the Cauchy problem
\begin{equation}
\left\lbrace \begin{array}{ll}
 \frac{\partial}{\partial \tau}\delta v^{*,\rho,l,1t}_i-\rho_l \Delta \delta v^{*,\rho,l,1t}_i=\rho_l\Delta v^{*,\rho,l-1}_i(l-1,.),\\
 \\
\delta v^{*,\rho,l,1t}_i=0, 
\end{array}\right.
\end{equation}
such that for $|\alpha|\leq m$ we have we may use the inductive information
\begin{equation}
\Delta D^{\beta}_xv^{*,\rho,l-1}_i(l-1,.)\sim \sqrt{l-1} \mbox{ for } |\beta|\leq m-2
\end{equation}
in order to conclude that we even have
\begin{equation}
D^{\beta}_x\delta v^{*,\rho,l,1t}_i\sim \frac{1}{\sqrt{l}} \mbox{ for } |\beta|\leq m-2. 
\end{equation}
However for $|\alpha|\geq m-1$ we need a different reasoning. Since the first term in (\ref{firstterm}) is a convolution we may write
\begin{equation}\label{firsttermeq}
D^{\alpha}_x\delta v^{*,\rho,l,1t}_i=\int_{{\mathbb R}^n}D^{\alpha}_xv^{*,\rho,l-1}_i(l-1,y)G_l(\tau,x-y)dy-D^{\alpha}_xv^{*,\rho,l-1}_i(l-1,x)
\end{equation}
for all $|\alpha|\leq m$, and the fact that we have polynomial decay of order $m\geq 2$ for $D^{\alpha}_xv^{*,\rho,l-1}_i(l-1,.)$ we have a limit (with respect to the supremum norm) of the from
\begin{equation}\label{firsttermeq2}
\lim_{\rho_l\downarrow 0}{\Big |}\int_{{\mathbb R}^n}D^{\alpha}_xv^{*,\rho,l-1}_i(l-1,y)G_l(\tau,.-y)dy-D^{\alpha}_xv^{*,\rho,l-1}_i(l-1,.){\Big |}=0
\end{equation}
independent of the time step number $l\geq 1$, and this leads to the conclusion that
\begin{equation}
D^{\alpha}_x\delta v^{*,\rho,l,1t}_i\sim 1\mbox{ for } |\alpha|\leq m. 
\end{equation}
Next the second term (\ref{scalparasystlin10v*}) is a convolution with the fundamental solution $G_l$ of products of functions of the type
\begin{equation}
v^{*,\rho,l-1}_j(s,y),\frac{\partial v^{*,\rho,l-1}_i}{\partial x_j}\sim \sqrt{l-1}
\end{equation}
integrated over time and space. Since these products have a factor $\rho_l$ and using the estimates for the fundamental solution  obtained in our local contraction results we have
\begin{equation}\label{scalparasystlin10v*}
 \begin{array}{ll}
-\rho_l\int_{l-1}^{\tau}\int_{{\mathbb R}^n}\sum_{j=1}^n v^{*,\rho,l-1}_j(s,y)\frac{\partial v^{*,\rho,l-1}_i}{\partial x_j}(s,y)G_l(\tau-s,x-y)dyds\sim 1\\
\\
+\rho_l\int_{l-1}^{\tau}\int_{{\mathbb R}^n}\int_{{\mathbb R}^n}\sum_{j,m=1}^n \left( \frac{\partial v^{*,\rho,l-1}_j}{\partial x_m}\frac{\partial v^{*,\rho,l-1}_m}{\partial x_j}\right) (l-1,y)\frac{\partial}{\partial x_i}K_n(z-y)\times\\
\\
\times G_l(\tau-s,x-z)dydzds\sim 1.
\end{array}
\end{equation} 
Hence we can realise the bound
\begin{equation}\label{firstincrg}
D^{\alpha}_x\delta v^{*,\rho,l,1}_i\sim \sqrt{l-1}+1,~\mbox{ for }|\alpha|\leq m,
\end{equation}
and this has some consequence for a refine of contraction of the higher order approximations.
Again we have the classical representation
\begin{equation}\label{deltaurhok0rep}
 \begin{array}{ll}
\delta v^{*,\rho,k+1,l}_i(\tau,x)=-\rho_l\int_{l-1}^{\tau}{{\mathbb R}^n}\sum_{j=1}^n v^{*,\rho,l,k-1}_j\frac{\partial \delta v^{*,\rho,l,k}_i}{\partial x_j}(s,y)G_l(\tau-s,x-y)dyds\\
\\
-\rho_l\int_{l-1}^{\tau}\int_{{\mathbb R}^n}\sum_j\delta v^{*,\rho,l,k}_j\frac{\partial v^{*,\rho,l,k}}{\partial x_j}(s,y)G_l(\tau-s,x-y)dyds+\\ 
\\
+\rho_l\int_{l-1}^{\tau}{{\mathbb R}^n}\int_{{\mathbb R}^n}K_{n,i}(z-y){\Big (} \left( \sum_{j,m=1}^n\left( v^{*,\rho,l,k}_{m,j}+v^{*,\rho,l,k-1}_{m,j}\right)(s,y) \right)\times\\
\\
\times  \delta v^{*,\rho,l,k}_{j,m}(s,y) {\Big)}G_l(\tau-s,x-z)dydzds.
\end{array}
\end{equation}
For $k=1$ we have the factors 
\begin{equation}
v^{*,\rho,0,l}_j=v^{*,\rho,l-1}_j(l-1.),\frac{\partial v^{\rho,l,1}}{\partial x_j}\sim \sqrt{l-1}+1
\end{equation}
for functional increments $\delta v^{\rho,k,l}_j$, and this leads to a contraction constant of order $\sim\frac{1}{\sqrt{l}}$ for the second substep $k=2$ of the time step $l$ by arguments which are completely analogous to the local contraction arguments above. Similar for derivatives up to order $|\alpha|\leq m$, where one derivative is taken by the fundamental solution as we did in the local contraction result above.
Note that for $k\geq 2$ we have
\begin{equation}
D^{\alpha}_x\delta v^{r,*\rho,l,k}_i=D^{\alpha}_x\delta v^{*,\rho,l,k}_i,
\end{equation}
so that we do not need to establish a contraction result for the controlled velocity functions. The same analysis can be done with initial data $v^{r,*,\rho,l-1}_i$, of course.
Hence, we get for $k\geq 2$
\begin{equation}\label{secondincrgr}
D^{\alpha}_x\delta v^{r,*,\rho,l,k}_i\sim \left( \frac{1}{\sqrt{l}}\right)^{k-1},~\mbox{ for }|\alpha|\leq m,
\end{equation}
for the controlled scheme, since for $k\geq 2$ we have
\begin{equation}
D^{\alpha}_x\delta v^{r,*,\rho,l,k}_i=D^{\alpha}_x\delta v^{*,\rho,l,k}_i(:=D^{\alpha}_x\delta v^{r^{l-1},*,\rho,l,k}_i),
\end{equation}
recalling our convention for the controlled scheme.
As we said these observations motivate our definition of a control functions $r^l_i$ (or a part of the control function) in \cite{KB1} and \cite{KB2}, where we defined
\begin{equation}\label{deltarl}
\delta r^l_i=r^l_i-r^{l-1}_i(l-1,.)=-\delta v^{r,*,\rho,l,1}_i.
\end{equation}
This implies that we have
\begin{equation}\label{funciv3}
\begin{array}{ll}
v^{r,*,\rho ,l }_i=v^{r,\rho,l-1}_i+\sum_{k=1}^{\infty} \delta v^{r,*,\rho, l,k}_i\\
\\
=v^{r,*,\rho,l-1}_i+\delta v^{r,*,\rho,l,1}_i+\sum_{k=2}^{\infty} \delta v^{r,*,\rho, l,k}_i\\
\\
=v^{r,*,\rho,l-1}_i+\sum_{k=2}^{\infty} \delta v^{*,\rho, l,k}_i
,~1\leq i\leq n,
\end{array}
\end{equation}
Hence, we have
\begin{equation}
D^{\alpha}_xv^{r,*,\rho ,l }_i\sim \sqrt{l} \mbox{ for } |\alpha|\leq m.
\end{equation}
Furthermore, note that
\begin{equation}
D^{\alpha}_xr^l_i\sim l\mbox{ for } |\alpha|\leq m.
\end{equation}
This closes the alternative argument for theorem \ref{uncontrolledthm1} and theorem \ref{uncontrolledthm2}.

\section{Further considerations concerning global regular upper bounds}
We supplement the discussion of the preceding section, and supplement the considerations of section 8 in order to show that a sharper upper bound as stated in theorem  \ref{uncontrolledthm3} can be constructed. As we remarked in section 5, this sharper result can be obtained also by auto-controlled schemes.
We have obtained a global bound of the controlled velocity functions $v^{r,*,\rho}_i,~1\leq i\leq n$ and a global linear bound of the control function $r_i$. It follows that we have a global linear bound of the velocity functions  $v^{\rho,*}_i,~1\leq i\leq n$ in time-transformed coordinates, and it follows that we have a linear global upper bound for the solution of the incompressible Navier Stokes equation itself. This is sufficient in order to prove global regular existence. The regularity obtained by the scheme is clearly related to the regularity of the spatial image $H^m$ of the function spaces $C^0\left([l-1,l],H^m \right)$ or $C^1\left([l-1,l],H^m \right)$. We have observed by local analysis that the spatial dependence of the local solutions $v^{*,\rho,l}_i=v^{r,*,\rho,l}_i-r^l_i,~1\leq i\leq n$ is of regularity  $C^m_0\cap H^ m$ on the domain $(l-1,l)\times {\mathbb R}^n$. Depending on the scheme a little additional work is needed to get the same regularity for the local time transformed solutions of the uncontrolled velocity functions $v^{*,\rho,l}_i$ with uncontrolled data for all times including the positive integer time values $l\geq 1$. Well it is rather routine to check this and we shall add some remarks below. 
Next we add some considerations about the extended scheme with small foresight.
We assume that such a $\epsilon >0$ has been chosen. We shall determine in the following how small it should be in order to facilitate the reasoning about uniform global upper bounds.  
Next, for some $1\leq i\leq n$ consider some argument $x\in {\mathbb R}^n$ and the function values $v^{r^{l-1},*,\rho,l}_i(l,x)$ and $r^{l-1}_i(x)$. The design of the control function for a scheme with small foresight is such that for small time step size $\rho_l$ after finitely many steps we have a transition from equal signs to equal signs as we shall see. In the case considered in section 7 we considered the upper bound $C^2+1$ for the control function over all time which becomes indeed an upper bound $C^2$ after finitely many time step. Next we start with an upper bound $(C')^2$ in a different case, where we may consider the constant $C$ to be generic and just write $C^2$.  Hence, assume that $v^{r,*,\rho,l-1}_i(l-1,x)$ and $r^{l-1}_i(l-1,x)$ have an equal sign where the control function value is not zero, and have the strict upper bounds $C$ and $C^2$ respectively (we use 'strict' upper bounds in order make the estimates simpler in the sense that the diffusion effects are accounted for by substituting $<$ by $\leq $).   Then, if the two function values $v^{r^{l-1},*,\rho,l}_i(l,x)$ and $r^{l-1}_i(x)$have different signs, i.e. if, $x\in V^{l-1,i}_{v+r-}\cup V^{l-1,i}_{v-r+}$, then we have
\begin{equation}
\begin{array}{ll}
{\big |}v^{r,*,\rho,l}_i(l,x){\big |}= {\big |}v^{r^{l-1},*,\rho,l}_i(l,x)+\delta r^l_i(l,x){\big |}\\
\\
\leq {\big |}v^{r,*,\rho,l-1}_i(l-1,x){\big |}+{\big |}\delta v^{r^{l-1},*,\rho,l}_i(l,x)+\delta r^l_i(l,x){\big |}\\
\\
={\big |}v^{r,*,\rho,l-1}_i(l-1,x){\big |}+{\Big |}\delta v^{r^{l-1},*,\rho,l}_i(l,x)\\
\\
-\int_{l-1}^{l}\int_{{\mathbb R}^n}g^l_i(y) p_l(\tau-s,x-y)dyds{\Big |},
\end{array}
\end{equation}
where the latter term involving the function $g^l_i$ is close to
\begin{equation}
-\left( 2v^{r^{l-1},*,\rho,l}_i(l,x)+\frac{r^{l-1}_i(l-1,x)}{C^2}\right)
\end{equation}
for small time step size $\rho_l>0$.
As $v^{r,*,\rho,l-1}_i(l-1,x)$ and $r^{l-1}_i(l-1,x)$ have the same sign and the modulus of the increment of the uncontrolled velocity function with controlled data $\delta v^{r^{l-1},*,\rho}_i(l,x)$ is bounded by $\epsilon >0$, in the case considered 
\begin{equation}
\delta v^{r^{l-1},*,\rho}_i(l,x)\in \left[-\epsilon ,0\right) 
\end{equation}
This implies that after one time step we have the lower and upper bounds
\begin{equation}
\epsilon \leq {\big |}v^{r,*,\rho,l}_i(l,x){\big |}\leq 1+\epsilon. 
\end{equation}
If the control function value ${\big |}r^{l-1}_i(l-1,x){\big |}\in \left[C,C^2\right] $, then an upper bound for the control function is
\begin{equation}
{\big |}r^{l}_i(l,x){\big |}\leq 2\epsilon + C-\frac{1}{C},
\end{equation}
If the control function value ${\big |}r^{l-1}_i(l-1,x){\big |}\in \left(0,C\right] $, then an upper bound for the control function is
\begin{equation}
{\big |}r^{l}_i(l,x){\big |}\leq 2\epsilon + C\leq 2C.
\end{equation}
In any case we have the strict upper bound $C^2$ for the control function after two time steps in the case considered. Such considerations lead to the observation that the moduli of the control functions $r^l_i$ have upper bounds $C^2$ after finitely many time steps, where for each $x\in {\mathbb R}^n$ the function $l\rightarrow v^{r,*,\rho,l}_i(l,x)+r^l_i(l,x)$ has the tendency to decrease in the sense that for each $l\geq 1$ there is a $m\geq 1$ such that
\begin{equation}
{\big |}v^{r,*,\rho,l}_i(l+m,x)+r^l_i(l+m,x){\big |} < {\big |}v^{r,*,\rho,l}_i(l,x)+r^l_i(l,x){\big |}.
\end{equation}
This leads to to the conclusion that the controlled velocity functions $v^{r,*,\rho,l}_i$ and the control functions $r^l_i$ both have uniform global upper bounds such that the the functions $v^{*,\rho,l}_i$ also have uniform upper bounds. This then holds for the original velocity function components $v_i$ as well. An analogous reasoning leads to the same conclusion for the spatial derivatives of the control functions and the controlled velocity functions up to order $m$, where $m$ is the order of derivatives for which the local contraction result holds. As we argued we have such a local contraction result for any order of derivatives, and the existence of uniform globally bounded smooth solutions for the Cauchy problem is a consequence. Furthemore, the scheme with small foresight is also appropriate in order to investigate the long time behavior of solutions. It seems that for the Cauchy problem on the whole space and a class of smooth data with polynomial decay and strictly positive viscosity $\nu >0$ we have a decay of the velocity functions to zero at infinity.

The observations above with small foresight have the advantage that we can study long time behavior, but there is the question whether the observations of an uniform global upper bound can be obtained with the simpler control function proposed above as well.  
Next we make further observations concerning global bounds of simpler control functions.
At the end of section 3 we have defined a simplified form of the global scheme via value functions $v^{r^{l-1},*,\rho,l,k}_i$ which describes the solutions of the local iteration scheme with controlled function data $v^{r,*,\rho,l-1}_i(l-1,.)$. This has the advantage that we can study the local limit $v^{r^{l-1},*,\rho,l-1}_i=\lim_{k\uparrow \infty}v^{r^{l-1},*,\rho,l,k}_i$ on $[l-1,l]\times {\mathbb R}^n$ similarly as the local limit of the uncontrolled scheme - the structure of the equations is the same  (cf. the remarks at (\ref{Navleray*})  at the end of section 3 of this paper). We can get the controlled function values of step $l$ by 
\begin{equation}
v^{r,*,\rho,l}_i=v^{r^{l-1},*,\rho,l}_i+\delta r^l_i.
\end{equation} 
In order to obtain global existence it is essential to get a global upper bound of the Leray projection term which may have linear growth with respect to the time step number $l\geq 1$. We may use the simple control functions which satisfy
\begin{equation}\label{deltargb1}
\delta r^{l}_i(\tau,x)=\int_{l-1}^{\tau}\frac{1}{C}\left(-v^{r,*,\rho,l-1}_i(l-1,y)\right)G_l(\tau-s,x-y)ds 
\end{equation}
or control functions which satisfy
\begin{equation}\label{deltargb2}
\delta r^{l}_i(\tau,x)=-\delta v^{r^{l-1},*,\rho,l,1}_i
\end{equation}
or linear combinations of (\ref{deltargb1}) and (\ref{deltargb2}) for this task. In any case we get a scheme with a global uniform or linear upper bound for the controlled value functions $v^{r,*,\rho,l}_i$ and a global linear upper bound for the control functions $r^l_i$. We shall consider this in the next section in more detail. Simple control functions as in (\ref{deltargb1}) and (\ref{deltargb2}) simplify the analysis. However an extended control function may be used to sharpen the results and obtain global upper bounds of regular solutions which do not depend (even linearly) on time. We analyze this problem in this section where we may use the analysis of the preceding section where we showed that a simpler analysis in the context of a simple control function is sufficient in order to obtain a global scheme and global regular solutions via global linear upper bounds in strong norms. We will start the analysis with some remarks further about simple control function schemes (which are sufficient in order to obtain global regular existence and are considered further in the next section), and then we consider variations and extensions of the arguments in the case of extended control functions.
Note again that we consider possible choices of the control function increments $\delta r^l_i$ which have to be chosen once at each time step. As the increments 
\begin{equation}
\delta v^{r^{l-1},*,\rho,l,k}_i=v^{r^{l-1},*,\rho,l,k}_i-v^{r^{l-1},*,\rho,l,k-1}_i
\end{equation}
satisfy equations which have the same structure as the equations for the increments $\delta v^{*,\rho,l,k}_i$ of the uncontrolled local scheme we simplify our notation and write
\begin{equation}\label{controllocalb}
v^{r^{l-1},*,\rho,l,k}_i=v^{r,*\rho,l-1}_i(l-1,.)+\sum_{p=1}^{k}\delta v^{r^{l-1},*,\rho,l,p}_i=:v^{r,*\rho,l-1}_i(l-1,.)+\sum_{p=1}^{k}\delta v^{*,\rho,l,p}_i
\end{equation}
where the increments $\delta v^{*,\rho,l,p}_i$ in (\ref{controllocalb}) are not to be confused with the increments of the local scheme but are just an abbreviation of $\delta v^{r^{l-1},*,\rho,l,p}_i$. This slight abuse of language may be justified by the simplification of notation and by the fact that both types of increments satisfy equations with the same structure (but different in the respect that they are induced by different initial data obtained at the previous time step or given).

   The procedure described makes it possible to use the local analysis and get a local limit by the local contraction results considered in the previous sections. If we denote this local limit by $v^{r^{l-1},*,\rho,l}_i$, then we may define $v^{r,*,\rho,l}_i=v^{r^{l-1},*,\rho,l}_i+\delta r^l_i$, where the control function increments $r^l_i-r^{l-1}_i,~1\leq i\leq n$ are chosen once at time step $l\geq 1$ such that the scheme becomes global. Since the local analysis does not depend of the control function such that only the size of the data $v^{r,*,\rho.l-1}_i(l-1,.)$ has influence on the the local scheme and for simplicity of notation we sometimes suppress the superscript $r^{l-1}$ in $v^{r^{l-1},*,\rho,l,k}_i$ and just write $v^{*,\rho,l,k}_i$. We have to keep in mind that the latter approximations are not independent of the control function $r^l_i$, certainly.

For a moment we may forget the control function 
and reconsider the inductive construction of local regular solutions on $[l-1,l]\times {\mathbb R}^n$ ($n=3$ the most intersting case) by the local scheme above (cf. also the discussion of global linear bounds of the Leray projection term below and in \cite{KHyp} for th following). At each time step $l\geq 1$ having constructed $v^{*,\rho,l-1}_i(l-1,.)\in C^m\cap H^m$ for $m\geq 2$ at time step $l-1$ (at $l=1$ these are just the initial data $h_i$),
as a consequence of the argument above for $m\geq 2$ we a have a time-local pointwise limit $v^{*,\rho,l}_i(\tau,.)=v^{*,\rho,l-1}_i(\tau,.)+\sum_{k=1}^{\infty}\delta v^{*,\rho,l,k}_i(\tau,.)\in H^m$ for all $1\leq i\leq n$, where for $n=3$ we have $H^2\subset C^{\alpha}$ uniformly in $\tau\in [l-1,l]$. Furthermore the functions of this series are even locally continuously differentiable with respect to $\tau\in [l-1,l]$ and hence H\"{o}lder continuous with respect to time.  Note that the first order time derivative $\frac{\partial}{\partial \tau}v^{*,\rho,l,k}_i(\tau,.)$ of these members exist in $H^m$ as well for natural $m>\frac{5}{2}$ as we proved in the previous section as a consequence of the product rule for Sobolev spaces. We observed that $v^{*,\rho,l,k}_i\in H^m\times H^{2m}$ can be obtained inductively for each $m$ and this leads to full local regularity of the limit function of the local scheme.
If we plug in the approximating function $v^{*,\rho,l,k}_i(\tau,.)$ into the local incompressible Navier-Stokes equation in the Leray projection form, then we have
\begin{equation}\label{Navlerayk}
\begin{array}{ll}
\frac{\partial v^{*,\rho,l,k}_i}{\partial \tau}-\rho_l\nu\sum_{j=1}^n \frac{\partial^2 v^{*,\rho,l,k}_i}{\partial x_j^2} 
+\rho_l\sum_{j=1}^n v^{*,\rho,l,k}_j\frac{\partial v^{*,\rho,l,k}_i}{\partial x_j}\\
\\
-\rho_l\sum_{j,m=1}^n\int_{{\mathbb R}^n}\left( \frac{\partial}{\partial x_i}K_n(x-y)\right) \sum_{j,m=1}^n\left( \frac{\partial v^{*,\rho,l,k}_m}{\partial x_j}\frac{\partial v^{*,\rho,l,k}_j}{\partial x_m}\right) (\tau,y)dy,\\
\\
=+\rho_l\sum_{j=1}^n \delta v^{*,\rho,l,k}_j\frac{\partial v^{*,\rho,l,k}_i}{\partial x_j}\\
\\
-\rho_l\sum_{j,m=1}^n\int_{{\mathbb R}^n}\left( \frac{\partial}{\partial x_i}K_n(x-y)\right) \sum_{j,m=1}^n\left( \frac{\partial v^{*,\rho,l,k}_m}{\partial x_j}\frac{\partial v^{*,\rho,l,k}_j}{\partial x_m}\right) (\tau,y)dy\\
\\
+\rho_l\sum_{j,m=1}^n\int_{{\mathbb R}^n}\left( \frac{\partial}{\partial x_i}K_n(x-y)\right) \sum_{j,m=1}^n\left( \frac{\partial v^{*,\rho,l,k-1}_m}{\partial x_j}\frac{\partial v^{*,\rho,l,k-1}_j}{\partial x_m}\right) (\tau,y)dy.
\end{array}
\end{equation}
As $k\uparrow \infty$ the right side of (\ref{Navlerayk}) becomes
\begin{equation}\label{Navleraykterm}
\begin{array}{ll}
\rho_l\sum_{j=1}^n \delta v^{*,\rho,l,k}_j\frac{\partial v^{*,\rho,l,k}_i}{\partial x_j}+\rho_l\sum_{j=1}^n  v^{*,\rho,l,k}_j\frac{\partial \delta v^{*,\rho,l,k}_i}{\partial x_j}\\
\\
\rho_l\int_{l-1}^{\tau}\int_{{\mathbb R}^n}\int_{{\mathbb R}^n}\left( \frac{\partial}{\partial x_i}K_n(z-y)\right)\times\\
\\
\times \left( \sum_{m,j=1}^n \frac{\partial \delta v^{*,\rho,l,k-1}_j}{\partial x_m}\left(\frac{\partial v^{*,\rho,l,k-1}_m}{\partial x_j}+\frac{\partial v^{*,\rho,l,k-2}_m}{\partial x_j}\right)\right)  (s,y)\times\\
\\
\times G_l(\tau-s,x-z)dydzds,
\end{array}
\end{equation}
and since  $\lim_{k\uparrow\infty}\delta v^{*,\rho,l,k}_j(\tau,x)=0$ and $\lim_{k\uparrow \infty}\frac{\partial \delta v^{*,\rho,l,k}_i}{\partial x_j}=0$ for all $(\tau,x)\in [l-1,l]\times {\mathbb R}^n$ pointwise by our local contraction result, and since $v^{*,\rho,l,k}_j$ and $\frac{\partial \delta v^{*,\rho,l,k}_i}{\partial x_j}$ are uniformly bounded, we observe that the right side expressed in (\ref{Navleraykterm}) goes to zero pointwise such that 
\begin{equation}
\lim_{k\uparrow \infty}v^{*,\rho,l,k}_i\in C^{1,2}\left([l-1,l]\times {\mathbb R}^n\right) ,~1\leq i\leq n
\end{equation}
satisfies the local Navier-Stokes equation pointwise, and in a classical sense. Higher regularity can be obtained then considering equations for the derivatives.   
Finally in order to show that the Leray projection term is globally linearly bounded, it suffices to show that the {\it squared } function
\begin{equation}
l\rightarrow |v^{*,\rho,l}_j(l,.)|^2_{H^2}
\end{equation}
grows linearly with the time step number $l$, or to show that a controlled scheme
\begin{equation}
l\rightarrow |v^{r,*,\rho,l}_j(l,.)|^2_{H^2}
\end{equation}
with linear bounded control functions $r_i$ have linear growth with respect to the time step number $l$.
 Now, let us set up a controlled scheme which we already sketched in the intoduction (similar as in \cite{KNS} and in \cite{K3}), i.e., let us consider a scheme
\begin{equation}
v^{r,*,\rho,l,k}_j=v^{*,\rho,l,k}_j+r^l_j
\end{equation}
for some functions $r^l_j$ which have a uniform upper bound with respect to some regular norm $|.|_{C^m\left((l-1,l), H^{2m}\right)}$. Especially this upper bound should be independent of the time step number $l$. We shall see that it is possible to simplify the control function considered in \cite{KNS} and \cite{K3}. We have considered this simplification in \cite{KB2} for the scheme considered there, and reconsider these considerations for the simplified scheme of this paper. Further simplifications are possible if the domain is a torus (cf. \cite{KIF2}). However the latter simplification seems to be related to the formulation in terms of Fourier modes, and it seems not possible to do the same on the whole domain.
Let us consider the uncontrolled scheme  for a moment in order to analyze the local growth. The considerations transfer easily to the controlled scheme. Observe that the contraction results allow us to estimate the higher correction terms
\begin{equation}\label{ldiff}
v^{*,\rho,l}_j-v^{*,\rho,l-1}_j(l-1,.)=\sum_{k=1}^{\infty}\delta v^{*,\rho,l,k}_j
\end{equation}
in the local functional series on $[l-1,l]\times {\mathbb R}^n$ at time step $l\geq 1$ in terms of the function
\begin{equation}\label{firstdiff}
v^{*,\rho,l,1}_j-v^{*,\rho,l-1}_j(l-1,.),
\end{equation}
i.e., the growth behavior of the latter function in (\ref{firstdiff}) determines the growth behavior of the former function (\ref{ldiff}) with respect to the relevant norm by the related contraction result. 
\begin{rem}
Note that in this paper for each time step number $l\geq 1$ we start with substeps $k\geq 1$ and use the convention
\begin{equation}
v^{*,\rho,l,0}_j(l-1,.)=v^{*,\rho,l-1}_j(l-1,.).
\end{equation}
In previous papers we started with $k\geq 0$ and used a similar convention with $k=-1$ instead of $k=0$.
\end{rem}
Assume for a moment that we can prove that
\begin{equation}\label{bound0}
\max_{1\leq j\leq n}\sup_{\tau\in[l-1,l]}{\Big |}v^{*,\rho,l,1}_j(\tau,.)-v^{*,\rho,l-1}_j(l-1,.){\Big |}_{H^2}\precsim \frac{1}{\sqrt{l}},
\end{equation}
and assume inductively that
\begin{equation}
\max_{j\in \left\lbrace1,\cdots ,n\right\rbrace }|v^{*,\rho,l-1}_j(l-1,.)|^2_{H^2}\leq C+(l-1)C,
\end{equation}
which holds for $l=1$ and for $\max_{j\in \left\lbrace1,\cdots ,n\right\rbrace }|h_j|_{H^2}^2$ for an appropriate constant for sure. Then assuming $C\geq 4$ w.l.o.g. we have
\begin{equation}
\max_{j\in \left\lbrace1,\cdots ,n\right\rbrace }|v^{*,\rho,l-1}_j(l-1,.)|_{H^2}\leq \sqrt{C+(l-1)C}\leq \frac{1}{2}(C+(l-1)C).
\end{equation}
Now we have for some $m\geq 2$ and $k\geq 2$
\begin{equation}\label{vrepineq7gproof}
\begin{array}{ll}
|\delta v^{*,\rho,l,k}_{i}|^2_{C^m\left((l-1,l), H^{2m}\right)} 
\leq \frac{1}{4}\max_{j\in \left\lbrace 1,\cdots ,n\right\rbrace }
{\big |} \delta v^{*,\rho,l,k-1}_j {\big |}^2_{C^m\left((l-1,l), H^{2m}\right)},
\end{array}
\end{equation}
where we may choose $\rho_l\sim\frac{1}{l}$ while for an appropriate choice of the constant factor (which transforms $\sim$ to $=$) we have
\begin{equation}
|v^{*,\rho,l,k}_i|_{C^m\left((l-1,l), H^{2m}\right)} \leq C^l_{k-1}\leq (\sqrt{C+C(l-1)}+1) \mbox{for all $k$}.
\end{equation}
Assuming that we can realize the bound
\begin{equation}\label{bound0a}
\max_{1\leq j\leq n}\sup_{\tau\in[l-1,l]}{\Big |}v^{*,\rho,l,1}_j-v^{*,\rho,l-1}_j(l-1,.){\Big |}_{C^m\left((l-1,l), H^{2m}\right)}\leq \frac{1}{2C\sqrt{l}},
\end{equation}
with a finite constant $C>0$ independent of the time step number $l\geq 1$ we can choose $\rho_l$ such that we have indeed a contraction estimate with contraction constant $\frac{1}{2C\sqrt{l}}$. We get
\begin{equation}
\begin{array}{ll}
\max_{j\in \left\lbrace1,\cdots ,n\right\rbrace }|v^{*,\rho,l}_j|_{C^m\left((l-1,l), H^{2m}\right)}\\
\\
\leq \max_{j\in \left\lbrace1,\cdots ,n\right\rbrace }|v^{*,\rho,l-1}_j(l-1,.)+\sum_{p=1}^{\infty} \delta v^{*,\rho,l,p}_j|_{C^m\left((l-1,l), H^{2m}\right)}\\
\\
\leq \max_{j\in \left\lbrace1,\cdots ,n\right\rbrace }|v^{*,\rho,l-1}_j(l-1,.)|_{C^m\left((l-1,l), H^{2m}\right)}+\sum_{k=1}^{\infty}| \delta v^{*,\rho,l-1,k}_j|_{C^m\left((l-1,l), H^{2m}\right)}\\
\\
\leq \sqrt{C+(l-1)C}+\frac{1}{C\sqrt{l}}.
\end{array}
\end{equation}
Hence,
\begin{equation}
\begin{array}{ll}
\max_{j\in \left\lbrace1,\cdots ,n\right\rbrace }|v^{*,\rho,l}_j(l,.)|^2_{C^m\left((l-1,l), H^{2m}\right)}\\
\\
\leq (\sqrt{C+(l-1)C}+\frac{1}{C\sqrt{l}})^2\\
\\
=C+(l-1)C+\sqrt{C+(l-1)C}\frac{1}{C\sqrt{l}}+\frac{1}{C^2l}
\leq C+lC.
\end{array}
\end{equation}
Hence, if we can realize the estimate (\ref{bound0a}), 
 then the scheme is global. 
 Equivalently, assuming that for a globally linearly bounded control function $\mathbf{r}$ we can realize the bound
\begin{equation}\label{bound0ar}
\max_{1\leq j\leq n}\sup_{\tau\in[l-1,l]}{\Big |}v^{r,*,\rho,l,1}_j-v^{r,*,\rho,l-1}_j(l-1,.){\Big |}_{C^m\left((l-1,l), H^{2m}\right)}\leq \frac{1}{2C\sqrt{l}},
\end{equation}
with a finite constant $C>0$ independent of the time step number $l\geq 1$, we can choose $\rho_l$ such that we have indeed a contraction estimate with contraction constant $\frac{1}{2C\sqrt{l}}$. We get
\begin{equation}
\begin{array}{ll}
\max_{j\in \left\lbrace1,\cdots ,n\right\rbrace }|v^{r,*,\rho,l}_j|_{C^m\left((l-1,l), H^{2m}\right)}\\
\\
\leq \max_{j\in \left\lbrace1,\cdots ,n\right\rbrace }|v^{r,*,\rho,l-1}_j(l-1,.)+\sum_{p=1}^{\infty} \delta v^{r,*,\rho,l,p}_j|_{C^m\left((l-1,l), H^{2m}\right)}\\
\\
\leq \max_{j\in \left\lbrace1,\cdots ,n\right\rbrace }|v^{r,*,\rho,l-1}_j(l-1,.)|_{C^m\left((l-1,l), H^{2m}\right)}+\sum_{k=1}^{\infty}| \delta v^{r,*,\rho,l-1,k}|_{C^m\left((l-1,l), H^{2m}\right)}\\
\\
\leq \sqrt{C+(l-1)C}+\frac{1}{C\sqrt{l}}.
\end{array}
\end{equation}
Hence,
\begin{equation}
\begin{array}{ll}
\max_{j\in \left\lbrace1,\cdots ,n\right\rbrace }|v^{r,*,\rho,l}_j(l,.)|^2_{C^m\left((l-1,l), H^{2m}\right)}\\
\\
\leq (\sqrt{C+(l-1)C}+\frac{1}{C\sqrt{l}})^2\\
\\
=C+(l-1)C+\sqrt{C+(l-1)C}\frac{1}{C\sqrt{l}}+\frac{1}{C^2l}
\leq C+lC.
\end{array}
\end{equation}
Hence, if we can realize the estimate (\ref{bound0ar}), then the scheme is global. The most simple choice in this context may be the choice of $\delta r^l_i=-\delta v^{r^{l-1},*,\rho,l,1}_i=:\delta v^{r^{l-1},*,\rho,l,1}_i$, where we recall that the latter notation is just an abbreviation (cf. the remarks at the beginning of this section). We considered such a possibility in the preceding section.  Other possibilities are considered in \cite{K3, KNS,KHyp}. Next we extend the scheme for  $v^{*,\rho,l,k}_j$ to a controlled scheme for
 \begin{equation}
 v^{r,*,\rho,l,k}_j=v^{r^{l-1},*,\rho,l,k}_j+\delta r^l_j,
 \end{equation}
where
\begin{equation}\label{controll1}
\begin{array}{ll}
r^{l}_j-r^{l-1}_j=-\left( v^{r^{l-1},*,\rho,l,1}_j-v^{r^{l-1},*,\rho,l-1}_j(l-1,.)\right)\\
\\
+\int_{l-1}^{\tau}\phi^{l}_j(s,y)G_l(\tau-s,x-y)dyds.
\end{array}
\end{equation}
In the preceding section we showed that the first term on the right side of (\ref{controll1}), i.e., the choice with $\phi^l_j\equiv 0$ is indeed sufficient in order to obtain a global linear bound for the Leray projection term. In this section we want to construct an upper bound for a controlled Navier Stokes equation system which is independent of the time step number $l\geq 1$. In order to do this we choose additional growth consumption functions. 
We choose a source term in (\ref{controll1}) is of the form
\begin{equation}\label{sourceleray}
\phi^{l}_j(s,y)=-\frac{v_j^{r,\rho,l-1}}{C}(l-1,.)-\frac{r^{l-1}}{C^2}(l-1,.).
\end{equation}
The idea of the latter choice is that for $\rho_l\lesssim\frac{1}{C^4}$ the increments
\begin{equation}\label{deltaremark}
\begin{array}{ll}
\delta v^{r,*,\rho,l}_i=v^{r,*,\rho,l}_i-v^{r,*,\rho,l-1}_i(l-1,.)\\
\\
=v^{r^{l-1},*,\rho,l}_i+\delta r^l_i-v^{r^{l-1},*,\rho,l-1}_i(l-1,.)\\
\\
=v^{r,*,\rho,l,1}_i+\sum_{k\geq 2}\delta v^{*,\rho,l,k}_i-v^{r,*,\rho,l-1}_i(l-1,.)
\end{array}
\end{equation}
have upper bounds in strong norms such that the growth of the controlled velocity function can be shown to be absorbed by the larger damping effect of the source function part $-\frac{v_j^{r,\rho,l-1}}{C}(l-1,.)$.
\begin{rem}\label{notrem}
Note again that in (\ref{deltaremark}) we use $\sum_{k\geq 2}\delta v^{*,\rho,l,k}_i$ as synonymous for $\sum_{k\geq 2}\delta v^{r^{l-1},*,\rho,l,k}_i$. 
\end{rem}
In the preceding section we proved that for the simple control function
\begin{equation}\label{controll1extra2}
\begin{array}{ll}
\delta r^{l,0}_j=-\delta v^{r^{l-1},*,\rho,l,1}_j
\end{array}
\end{equation}
there is a global linear bound (in strong norms) for the controlled value function and for the control function itself. This is sufficient for global regular existence. We also considered the alternative control functions (considered also in \cite{KNS,K3} essentially) 
\begin{equation}\label{controll1extra3}
\begin{array}{ll}
\delta r^{l,1}_j=\int_{l-1}^{\tau}\left( -\frac{v_j^{r,\rho,l-1}}{C}(l-1,y)\right) G_l(\tau-s,x-y)dyds.
\end{array}
\end{equation}
Let us recall why the controlled scheme for $v^{r_1,*,\rho,l}_i:=v^{*,\rho,l}_i+r^{l,1}_i$ and with $r^{l,1}_i=r^{l-1,1}_i(l-1,.)+\delta r^{l,1}_i$ (starting with some appropriate control at time step $l=1$) leads to a global  bound for the controlled value functions $v^{r_1,*,\rho,l}_i$ and for a global linear bound for the control function $r^{l,1}_i$ (both in strong norms), and, hence, to global regular solutions as well.
Assume that we have data
\begin{equation}
{\big |}D^{\alpha}_xv^{r_1,*,\rho,l-1}_i(l-1,.){\big |}\leq C
\end{equation}
for multivariate derivatives up to order $|\alpha|\leq m$, and where $|f(.)|$ is short for the supremum norm of a function $f$. Consider the case $\alpha=0$ first. We have
\begin{equation}
v^{r_1,*,\rho,l}_i(l,.)=v^{r_1,*,\rho,l-1}_i(l-1,.)+\sum_{k=1}^{\infty}\delta v^{r^{l-1,1},*,\rho,l,k}_i(l,.)+\delta r^{l,1}_i(l,.).
\end{equation}
Note that in this case the sum of increments starts with $k=1$.
For small time step size $\rho_l$ the first increment $\delta v^{r^{l-1,1},*,\rho,l,k}_i(l,.)$ becomes small as does the sum of the moduli of the higher order terms according to the local contraction result such that (for example)
\begin{equation}
{\big |}\sum_{k=1}^{\infty}\delta v^{r^{l-1,1},*,\rho,l,k}_i(l,.){\big |}\leq \frac{1}{4}.
\end{equation}
Now if ${\big |}v^{r_1,*,\rho,l-1}_i(l-1,.){\big |}\in \left[0,\frac{C}{2}\right]$ then we have 
\begin{equation}
{\big |}v^{r_1,*,\rho,l}_i(l,.){\big |}\leq {\big |}v^{r_1,*,\rho,l-1}_i(l-1,.){\big |}+\frac{1}{4}\leq C,
\end{equation}
while for ${\big |}v^{r_1,*,\rho,l-1}_i(l-1,.){\big |}\in \left[\frac{C}{2};C\right]$ we have 
\begin{equation}
{\big |}v^{r_1,*,\rho,l}_i(l,.){\big |}\leq {\big |}v^{r_1,*,\rho,l-1}_i(l-1,.){\big |}+\frac{1}{4}-\frac{1}{2}-\epsilon\leq C,
\end{equation}
where for small time step size there is $\epsilon >0$ such that
\begin{equation}
|\delta r^{l,1}_i(l,.)|\geq \frac{1}{2}-\epsilon
\end{equation}
since we have small diffusion for small time step size. Hence, the upper bound $C$ is preserved from time step number $l-1$ to time step number $l$ for $v^{r_1,*,\rho,l}_i(l-1,.)$.
This upper bound is independent of the time step number $l\geq 1$, and this implies that
\begin{equation}
 |r^{l,1}_i(l,.)|\leq |r^{1,1}_i(0,.)+\sum_{p=1}^l\delta r^{p,1}_i|\leq |r^{0}_i(0,.)|+l
\end{equation}
which leads to a linear bound of the control function. At time step $l=1$ an example of an appropriate choice for the control function is
\begin{equation}
r^{0,1}_i(0,.)=\frac{h_i(.)}{C}.
\end{equation}
Note that for $\alpha>0$ we can write
\begin{equation}\label{controll1extra3}
\begin{array}{ll}
D^{\alpha}_x\delta r^{l,1}_j=\int_{l-1}^{\tau}\left( -\left( D^{\alpha}_x\frac{v_j^{r,\rho,l-1}}{C}(l-1,y)\right)\right)  G_l(\tau-s,x-y)dyds,
\end{array}
\end{equation}
since we have inductively regularity with spatial decay and we can apply the convolution rule. Since we have a contraction result for derivatives of value functions as well it follows that the argument above can repeated for multivariate derivatives and we get a global bound $C>0$ (indpendent of the time step number $l\geq 1$ for the controlled functions $D^{\alpha}_xv^{r_1,*,\rho,l}_i$ and a linear global bound (with repsect to the time step number $l\geq 1$) of the multivariate derivatives of the control function $D^{\alpha}_xr^l_i$.
This implies the existence of a global regular solution.
The argument above can be repeated for the control function
\begin{equation}\label{controll1extra12}
\begin{array}{ll}
\delta r^{l}_j=-\delta v^{r^{l-1},*,\rho,l,1}_j\\
\\
+\int_{l-1}^{\tau}\left( -\frac{v_j^{r,\rho,l-1}}{C}(l-1,y)\right) G_l(\tau-s,x-y)dyds,
\end{array}
\end{equation}
of course. Finally, we mention that the sharper result of global uniform upper bounds can be obtained even using the more elaborate but simpler control function in (\ref{controll1extra}), i.e., we can eliminate the linear dependence on the time step number $l\geq 1$ even for this simple control function. As the linear upper bound for the controlled value functions and the control functions together with the local contraction results are sufficient in order to prove global regular (smooth) existence, this additional is just an alternative way to prove global uniform upper bounds. Therefore, we close this section with a few remarks and come back to this issue elsewhere, when we analyze long time behavior. Consider again the control function increment  
  \begin{equation}\label{controll1incrementrep}
\begin{array}{ll}
\delta r^l_i(\tau,x)=r^{l}_j(\tau,.)-r^{l-1}_j(l-1,.)=-\left( v^{r^{l-1},*,\rho,l,1}_j-v^{r^{l-1},*,\rho,l-1}_j(l-1,.)\right)\\
\\
+\int_{l-1}^{\tau}\left( -\frac{v_j^{r,\rho,l-1}}{C}(l-1,.)-\frac{r^{l-1}}{C^2}(l-1,.)\right) G_l(\tau-s,x-y)dyds.
\end{array}
\end{equation}
Having computed the control function $r^{l-1}_i(l-1,.)$ and the controlled velocity functions and assume again that
\begin{equation}\label{upperboundl-1}
{\big |}r^{l-1}_i(l-1,.){\big |}\leq C^2,~{\big |}v^{r,*,\rho,l-1}_i(l-1,.){\big |}\leq C
\end{equation}
for all $1\leq i\leq n$. 
Next compare this scheme with this control function with  schemes with the control function
  \begin{equation}\label{controll1incrementrep}
\begin{array}{ll}
\delta r^{l,0}_i(\tau,x)=r^{l,0}_j(\tau,.)-r^{l-1,0}_j(l-1,.)=-\left( v^{r^{l-1,0},*,\rho,l,1}_j-v^{r^{l-1,0},*,\rho,l-1}_j(l-1,.)\right)\\
\\
+\int_{l-1}^{\tau}\left( -\frac{v_j^{r_0,\rho,l-1}}{C}(l-1,.)\right) G_l(\tau-s,x-y)dyds.
\end{array}
\end{equation}
As we have global linear bounds for the controlled velocity functions  $v^{\rho,*,l}_i,~1\leq i\leq n$ and for the control functions $r^{l}_i,~1\leq i\leq n$ we have a global linear bound for the Navier Stokes solution function  
\begin{equation}
v^{*,\rho,l}_i=v^{r,*,\rho,l}_i-r^l_i=v^{r_0,*,\rho,l}_i-r^{l,0}_i,~1\leq i\leq n
\end{equation}
Note that these are two representations of a global regular solution function of the incompressible Navier Stokes equation. For the simpler scheme we have the upper bounds
\begin{equation}
{\big |}v^{r_0,*,\rho,l}_i(l,.){\big |}\leq C
\end{equation}
and (with the choice ${\big |}r^{0}_i(0,.){\big |}={\big |}r^{0,0}_i(0,.){\big |}={\Big |}\frac{h_i}{C}{\Big |}\leq 1$) we have 
\begin{equation}
{\big |}r^{l,0}_i(l,.){\big |}\leq 1+l.
\end{equation}
An alternative to obtain global uniform upper bounds is then by estimating the difference of both schemes.

\section{Final remarks}
Let us first summarize the ideas. In order to get a uniform global bound for the controlled velocity functions $v^{r,*,\rho,l}_i$ and a  linear bound in time for the control functions $r^l_i$ it is sufficient to define control functions recursively time step by time step
with the increment
\begin{equation}\label{controll1app}
\begin{array}{ll}
\delta r^l_i=r^{l}_j-r^{l-1}_j=
\int_{l-1}^{\tau}\frac{-v^{r,*,\rho,l-1}_i}{C}(l-1,y)G_l(\tau-s,x-y)dyds.
\end{array}
\end{equation}
These control functions ensure that the upper bound for the controlled velocity value functions is given by some $C$ where $r^1_i$ may be chosen proportional to the initial data (with the same signs, i.e. a positive proportionality constant for example), and the initial data have $C$ as an upper bound in suitable strong norms (depending on the regularity we want to achieve). It is then clear that the increments $\delta r^l_i$ in (\ref{controll1app}) are bounded by $\sim 1$ and we have a linear bound $\sim l$ of the control functions $r^l_i$. This is sufficient in order obtain global regular solutions for the incompressible Navier Stokes equation, where the regularity which can be obtained by the scheme equals the order $m$ of the local contraction result. Spatial dependence of the local solution functions for fixed time has a regularity in $H^m\cap C^m_0$, and this leads to classical regularity with respect to time as well. An alternative in order to obtain a global upper bound (linear with respect to the time step number $l$) are the control function increments of the form
\begin{equation}
\delta r^l_i=r^{l}_j-r^{l-1}_j=-\left( v^{*,\rho,l,1}_j-v^{*,\rho,l-1}_j(l-1,.)\right).
\end{equation}
In this case we get a linear bound for the controlled value functions and a linear bound for the control function. In addition extended control functions $r^l_i$ are defined recursively via the equation for the increment
\begin{equation}\label{controll1app1}
\begin{array}{ll}
\delta r^l_i=r^{l}_j-r^{l-1}_j=-\left( v^{*,\rho,l,1}_j-v^{*,\rho,l-1}_j(l-1,.)\right)\\
\\
+\int_{l-1}^{\tau}\phi^{l}_j(s,y)G_l(\tau-s,x-y)dyds.
\end{array}
\end{equation}
More explicitly we have, that
\begin{equation}\label{controll1extra}
\begin{array}{ll}
\delta r^{l}_{j}(l,x)=-\delta v^{r^{l-1},*,\rho,l,1}_j(l,x)\\
\\
+\int_{l-1}^{l}\left( -\frac{v_j^{r,*,\rho,l-1}(l-1,s)+\frac{1}{C}r^{l-1}_j(l-1,y) }{C}(l-1,y)\right) G_l(l-y,x-y)dyds.
\end{array}
\end{equation} 
In this definition of extended control functions the increment $$-\left( v^{*,\rho,l,1}_j-v^{*,\rho,l-1}_j(l-1,.)\right)$$ ensures that we can represent the local solution in the form
\begin{equation}
\begin{array}{ll}
v^{r,*,\rho,l}_i=v^{r,*,\rho,l-1}_i(l-1,.)+\sum_{j=1}^{\infty}\delta v^{r,*,\rho,l,k}_i\\
\\
=v^{r,*,\rho,l-1}_i(l-1,.)+\sum_{j=2}^{\infty}\delta v^{r,*,\rho,l,k}_i\\
\\
=v^{r,*,\rho,l-1}_i(l-1,.)+\sum_{j=2}^{\infty}\delta v^{*,\rho,l,k}_i\\
\\
:=v^{r,*,\rho,l-1}_i(l-1,.)+\sum_{j=2}^{\infty}\delta v^{r^{l-1},*,\rho,l,k}_i,
\end{array}
\end{equation}
where the definition of $\delta v^{*,\rho,l,k}_i=\delta v^{r^{l-1},*,\rho,l,k}_i$ reminds us that the increments depend on the initial data of the previous time step (which involves a control function), but satisfy equation for which we can apply the local analysis of the uncontrolled scheme.
This simplifies the analysis, since the definition of the control function does not depend on the substeps $k$ at each time step $l$. It follows that the representation of the approximative increments
\begin{equation}\label{deltafin}
v^{r,*,\rho,l,k}_i-v^{r,*,\rho,l-1}_i(l-1,.)=\sum_{j=2}^{k}\delta v^{*,\rho,l,m}_i.
\end{equation}
involves only product terms with factors of velocity value function and its spatial derivatives of the previous substeps. We mentioned that the latter equation involves a simplified notation and what we mean by (\ref{deltafin}) is 
\begin{equation}
v^{r,*,\rho,l,k}_i-v^{r,*,\rho,l-1}_i(l-1,.)=\sum_{j=2}^{k}\delta v^{r^{l-1},*,\rho,l,m}_i.
\end{equation}
This feature of the representation makes the inheritance of spatial polynomial decay of a certain order of the scheme easy to prove, and it shows that the scheme can be generalized to situations with more general type diffusions, especially variable viscosity and Navier Stokes equations on manifolds. The growth control function $\phi^l_i$ in (\ref{controll1app1}) controls the growth in time as we have seen in the previous section. However, if we do not add a growth control function then we can still obtain a global upper bound of the Leray projection term which grows linearly with respect to the time step number.
Finally, in this latest version of the paper we have added a more sophisticated control function with small foresight which is suitable in order to prove the existence of uniform global upper bounds, i.e., global upper bounds which do not depend on the time step number and hold for the controlled velocity function and for the control function as well. An analysis of many but finitely many cases shows that the control function  and the controlled velocity function have a tendency to decrease strictly in appropriate norms. This approach makes an analysis of long time behavior possible, which may be done elsewhere.

\footnotetext[1]{\texttt{{kampen@wias-berlin.de}, {kampen@mathalgorithm.de}}.}


\begin{thebibliography}{19}
\baselineskip=12pt


\bibitem{CFNT}
	{\sc Constantin, P., Foias, C., Nicolenko, B., Temam, R.}: {\em Integral manifolds and inertial manifolds for dissipative partial differential equations.} Springer, 1989.

\bibitem{Feff}
{\sc Fefferman, C.L.}, {\em Existence \& Smoothness of the Navier Stokes equation}
$\mbox{www.claymath.org/millennium/Navier-Stokes-Equations}$	
	
\bibitem{FJRT}
	{\sc Foias, C., Manley, O., Rosa, R., Temam, R.}: {\em Navier-Stokes equations and turbulence.} CUP, 2001.
	

\bibitem{FJKT}
	{\sc Foias, C., Jolly, M., Kravchenko, M., Tity, E.}: {\em Navier-Stokes equation, determing modes, dissipative dynamical systems.} arXiv1208,5434v1, August 2012.


\bibitem{KNS}
	{\sc Kampen, J\"org}: {\em Constructive analysis of the Navier-Stokes equation.} arXiv10044589v6, Juli 2012. 

\bibitem{K3}
 {\sc Kampen, J.,} {\em A global scheme for the incompressible Navier-Stokes equation on compact Riemannian manifolds}, arXiv: 1205.4888v5,  Aug. 2013. (a final version is in preperation)


 \bibitem{KB1}
{\sc Kampen, J.} {\em On the multivariate Burgers equation and the incompressible Navier-Stokes equation (part I)}, arXiv:0910.5672v5  [math.AP], Mar. 2011.	
	
\bibitem{KB2}
{\sc Kampen, J.} {\em On the multivariate Burgers equation and the incompressible Navier-Stokes equation (part II)}, arXiv:1206.6990v6  [math.AP], Jan. 2013. (in revision)

\bibitem{KHyp}
{\sc Kampen, J.} {\em On global schemes for highly degenerate Navier Stokes equation systems}, arXiv:13056385.v3 [math.AP] (Aug. 2013) 


\bibitem{KAC}
{\sc Kampen, J.} {\em
On an auto-controlled global existence scheme of the incompressible  Navier Stokes equation}, arxiv,  [math.AP], Oct. 2013, (final version in preperation).

\bibitem{KE1}
{\sc Kampen, J.} {\em
Singular vorticity solutions of the incompressible Euler equation via inviscid limits}, arXiv 1308.6082v4,  [math.AP], Nov. 2013. (final version in preperation)

\bibitem{KE2}
{\sc Kampen, J.} {\em
On a class of singular solutions to the incompressible 3-D Euler equation}, arXiv:1209.6250 [math.AP], Oct. 2012. (in revision)

\bibitem{KIF2}
{\sc Kampen, J.} {\em Some infinite matrix analysis, a Trotter product formula for dissipative operators, and an algorithm for the incompressible Navier-Stokes equation}, arXiv:1212.2403v4 [math.AP], Oct. 2013.


\bibitem{KT}
{\sc Koch, H.} {\em Well-posedness for the Navier Sokes equations}, Adv. Math. 157, no. 1, 22-35, 2001.

\bibitem{L}
{\sc Leray, J.} {\em Sur le Mouvement d'un Liquide Visquex Emplissent l'Espace}, Acta Math. J. (63), 193-248, 1934.


\bibitem{Tao}
{\sc Tao, T.} {\em Localisation and compactness properties of the Navier Stokes global regularity problem}, arXiv:1108.1165v4  [math.AP], 2011. 



 \end{thebibliography}
\end{document}